\documentclass[11pt]{amsart}
\usepackage[color=blue!20!white,textsize=tiny]{todonotes}

\usepackage{amsmath} 
\usepackage{amsthm}
\usepackage{amssymb} 
\usepackage{amsfonts}\usepackage{tikz-cd}
\usepackage{graphicx}
\usepackage{xcolor}
\definecolor{darkgreen}{rgb}{0,0.5,0}
\usepackage[normalem]{ulem}
\usepackage{comment}

\usepackage{caption,subcaption,pinlabel}
\usepackage{graphics}
\usepackage[pagebackref=true]{hyperref}

\usepackage{comment}
\usepackage{enumitem}
\usepackage{mathtools}
\usepackage{tikz}
\usetikzlibrary{matrix,arrows,decorations.pathmorphing}
\usepackage{mathabx}
\usetikzlibrary{matrix}
\usepackage{scrextend}

\newtheorem{theorem}{Theorem}[section]
\newtheorem{lemma}[theorem]{Lemma}
\newtheorem{proposition}[theorem]{Proposition}
\newtheorem{question}[theorem]{Question}

\newtheorem{corollary}[theorem]{Corollary}

\theoremstyle{definition}
\newtheorem{definition}[theorem]{Definition}
\newtheorem{example}[theorem]{Example}

\newtheorem{remark}[theorem]{Remark}

\newcommand{\im}{\textrm{im}}
\newcommand{\inv}{\iota}

\newcommand{\ita}{\iota \circ \tau}
\def\Inv{\mathfrak{I}}

\newcommand{\bunderline}[1]{\underline{#1\mkern-2mu}\mkern2mu }

\def\du {\bar{d}}
\def\dl {\bunderline{d}}

\def\H{\mathcal{H}}
\newcommand \bH {\overline{\H}}

\newcommand \bJ {\bar{J}}
\newcommand \bh {\mkern3mu \overline{\mkern-3mu H \mkern-1mu} \mkern1mu}
\newcommand \bHp {\overline{\Hp}}

\def\spinc {{\operatorname{spin^c}}}
\def\s{\mathfrak s}
\def\bs{\bar{\mathfrak{\s}}}
\def\w{\mathbf{w}}

\def\CF {\mathit{CF}}
\def\HF {\mathit{HF}}
\newcommand\HFhat{\widehat{\HF}}

\newcommand\HFp {\HF^+}
\newcommand \CFp {\CF^+}
\newcommand \CFm {\CF^-}
\newcommand \HFm {\HF^-}

\def\HIm{\mathit{HI}^-}
\def\fA{\mathfrak{A}}
\def\cG{\mathcal{G}}

\newcommand \Hc {\HF_{\mathrm{conn}}}

\def\Hp{\mathcal{H}}
\newcommand \Hconn {H_\mathrm{conn}}

\newcommand\alphas{\boldsymbol\alpha}
\newcommand\betas{\boldsymbol\beta}

\newcommand\Ta{\mathbb{T}_\alpha}
\newcommand\Tb{\mathbb{T}_\beta}

\newcommand{\G}{\Theta_\Z^{\tau}}

\newcommand{\Gdiff}{\smash{\Theta_\Z^{\text{diff}}}}

\def\ff {{\mathbb{F}}}

\newcommand{\Z}{\mathbb{Z}}
\newcommand{\Q}{\mathbb{Q}}
\newcommand{\R}{\mathbb{R}}

\let\int\relax
\newcommand{\int}{\mathring}

\DeclareMathOperator{\id}{{id}}

\def\Sym{\mathrm{Sym}}

\usepackage{accents}

\DeclareMathSymbol{\wtilde}{\mathord}{largesymbols}{"65}

\title{Corks, Involutions, and Heegaard Floer Homology}
\author{Irving Dai, Matthew Hedden, Abhishek Mallick}

\address{Department of Mathematics, Massachusetts Institute of Technology, Cambridge, MA 02142 \vskip.1in}
\email{idai@mit.edu }

 \address{Department of Mathematics, Michigan State University, East Lansing, MI 48824 \vskip.05in}
 \email{mhedden@math.msu.edu}

\address{Department of Mathematics, Michigan State University, East Lansing, MI 48824 \vskip.05in}
\email{mallicka@math.msu.edu }
\begin{document}

\maketitle


\begin{abstract}
Building on the algebraic framework developed by Hendricks, Manolescu, and Zemke, we introduce and study a set of Floer-theoretic invariants aimed at detecting corks. Our invariants obstruct the extension of a given involution over {\em any} homology ball, rather than a particular contractible manifold. Unlike previous approaches, we do not utilize any closed 4-manifold topology or contact topology. Instead, we adapt the formalism of local equivalence coming from involutive Heegaard Floer homology. As an application, we define a modification $\G$ of the homology cobordism group which takes into account an involution on each homology sphere, and prove that this admits a $\Z^\infty$-subgroup of strongly non-extendable corks. The group $\G$ can also be viewed as a refinement of the bordism group of diffeomorphisms. Using our invariants, we furthermore establish several new families of corks and prove that various known examples are strongly non-extendable. Our main computational tool is a monotonicity theorem which constrains the behavior of our invariants under equivariant negative-definite cobordisms, and an explicit method of constructing such cobordisms via equivariant surgery.
\end{abstract}

\section{Introduction}\label{sec:1}

Let $Y$ be an oriented three-dimensional integer homology sphere equipped with an orientation-preserving involution $\tau$. If $W$ is a smooth oriented 4-manifold whose boundary is $Y$, then it is natural to ask whether $\tau$ extends to a diffeomorphism of $W$. Of special interest is the case when $W$ is contractible, as then $\tau$ extends over $W$ as a homeomorphism by Freedman's theorem \cite{Freedman}. In \cite{Akbulutfake}, Akbulut provided the first example of a compact, contractible $W$  admitting a non-extendable involution on its boundary.  He proved this by exhibiting an embedding of $W$ into a blown-up $K3$-surface for which cutting out $W$ and re-gluing via $\tau$ produces an exotic smooth structure. Matveyev \cite{Matveyev} and Curtis-Freedman-Hsiang-Stong \cite{CFHS} later showed that in fact {\em any} two smooth structures on the same (simply-connected) topological 4-manifold are related by the cutting out and re-gluing of some contractible submanifold along a boundary involution. Such an operation is called a \textit{cork twist}. 

The aim of the present paper is to study the non-extendability of such involutions by utilizing their action on the Floer homology of $Y$. This idea has appeared in the literature before; see work of Saveliev \cite{Savelievnote}, Akbulut-Durusoy \cite{AD}, Akbulut-Karakurt \cite{AKar}, and Lin-Ruberman-Saveliev \cite{LRS}. Here, we take a slightly different point of view by encoding the induced action of $\tau$ in an algebraic object called a \textit{$\iota$-complex}. This  notion was introduced by Hendricks, Manolescu, and Zemke \cite{HMZ}, who defined $\iota$-complexes in their work on involutive Heegaard Floer homology. The reader should also consult the work of Alfieri, Kang, and Stipsicz \cite{AKS}, who carry out a similar strategy in the setting where $Y$ is the double branched cover of a knot. The primary tool for our study comes from the following theorem:

\begin{theorem}\label{thm:1.1}
Let $Y$ be an integer homology sphere with  involution $\tau : Y \rightarrow Y$. Then there are two Floer-theoretic invariants
\[
h_{\tau}(Y) = [(\CFm(Y)[-2], \tau)] \text{ and } h_{\ita}(Y) = [(\CFm(Y)[-2], \ita)]
\]
associated to the pair $(Y, \tau)$. If either $h_{\tau}(Y) \neq 0$ or $h_{\ita}(Y) \neq 0$, then $\tau$ does not extend to a diffeomorphism of any homology ball bounded by $Y$.
\end{theorem}

\noindent
Note that we obstruct the extension of $\tau$ as a diffeomorphism, which is stronger than obstructing the extension of $\tau$ as an involution. The invariants $h_{\tau}(Y)$ and $h_{\ita}(Y)$ take values in $\Inv$, the group of  $\iota$-complexes modulo local equivalence (see Section~\ref{sec:2.2}). This group has been studied by several authors in the context of involutive Heegaard Floer homology; see for example \cite{HMZ}, \cite{HHL}, \cite{DS}, \cite{DHSTcobordism}. In this paper, we will leverage these results to better understand the output of $h_\tau$ and $h_{\ita}$. Experts in involutive Heegaard Floer homology may find the appearance of $h_{\ita}(Y)$ somewhat surprising: it turns out that in addition to studying the induced action of $\tau$ on $\CFm(Y)$, it is also fruitful to consider the composition $\ita$, where $\iota$ is the homotopy involution on $\CFm(Y)$ coming from the conjugation symmetry \cite{HM}. Indeed, we give examples of pairs $(Y, \tau)$ whose non-extendability can be detected by one of $h_\tau$ and $h_{\ita}$, but not the other.


Existing methods for detecting corks have either relied on finding an embedding of $W$ in a closed 4-manifold with non-vanishing smooth 4-manifold invariant, or on requiring that $W$ be Stein and exploiting a slice-Bennequin inequality (see Remark \ref{rem:LRS}). Our approach is thus philosophically different, as $h_{\tau}$ and $h_{\ita}$ are constructed from a purely three-dimensional point-of-view and do not utilize any contact topology. The connection with corks arises from functoriality properties of our invariants with respect to definite cobordisms between 3-manifolds. Indeed, we stress that the 4-manifold $W$ does not appear as data in either $h_{\tau}(Y)$ or $h_{\ita}(Y)$, and that $h_{\tau}(Y)$ and $h_{\ita}(Y)$ obstruct the extension of $\tau$ over \textit{any} homology ball (contractible or otherwise) that $Y$ might bound. We refer to such  $\tau$ as being \textit{strongly non-extendable}. It was only recently that the first example of such an involution was exhibited. In \cite{LRS}, Lin, Ruberman, and Saveliev established this property for the involution on the boundary of Akbulut's original cork. Here, we will exhibit many other examples of corks with strongly non-extendable boundary involutions. The robust nature of our results might lead one to wonder if all cork boundary involutions are strongly non-extendable. (See Question~\ref{question:strong}.) Note that by work of Akbulut and Ruberman \cite{AR}, a fixed $Y$ can certainly bound many different contractible manifolds.

We quantify the profusion of corks with strongly non-extendable involutions by constructing a modification of the usual homology cobordism group $\Theta^3_{\Z}$, which we refer to as the \textit{homology bordism group of involutions}. Roughly speaking, this new object is generated by pairs $(Y, \tau)$, where $Y$ is an oriented integer homology sphere and $\tau$ is an involution on $Y$. We identify $(Y_1, \tau_1)$ and $(Y_2, \tau_2)$ whenever there exists an oriented homology cobordism between them that admits a diffeomorphism restricting to $\tau_i$ on each $Y_i$. The set of  such pairs modulo  equivalence forms a group, which we denote by $\G$ (we suppress reference to the dimension of the manifolds for brevity).   The precise definition of $\G$ is actually somewhat more complicated, due to the fact that we must take disjoint union as our group operation; see Section~\ref{sec:homologycobordism}. Despite this, there is 
the expected forgetful homomorphism (Proposition \ref{prop:forget}) to the homology cobordism group:
\[
F: \G \rightarrow \Theta^3_{\Z}.
\] 
The group $\G$ can be viewed as a refinement of the classical bordism group of diffeomorphisms $\Delta_n$, which was introduced by Browder in \cite{Browder}. The groups $\Delta_n$ are understood in all dimensions, by work of Kreck \cite{Kreckbordism} ($n \geq 4$), Melvin \cite{Melvin} ($n = 3$), and Bonahon \cite{Bonahon} ($n = 2$). In particular, in \cite{Melvin} it was shown that $\Delta_3 = 0$. However, it is natural to ask whether placing homological restrictions on the manifolds and bordisms in question results in a richer group structure. This parallels the situation in which the  three-dimensional oriented cobordism group $\Omega_3$ is trivial, but understanding the homology cobordism group $\Theta^3_{\Z}$ is difficult.   

To this end, we show that our invariants are homomorphisms from $\G$ to $\Inv$:

\begin{theorem}\label{thm:1.2}
The invariants $h_\tau$ and $h_{\ita}$ constitute homomorphisms
\[
h_{\tau}, h_{\ita} : \G \rightarrow \Inv.
\]
\end{theorem}

\noindent
Clearly, if $\tau$ extends over some homology ball $W$, then $(Y, \tau)$ is equivalent to $(S^3, \id)$. Indeed, we may isotope the extension of $\tau$ in the interior of $W$ to fix a ball; cutting out this ball gives the desired cobordism. The focus of this paper may thus be thought of as finding nontrivial elements in the kernel of  the forgetful map $F$.  In fact, we will find elements in the kernel represented by pairs $(Y,\tau)$ satisfying the stronger condition that $Y$ bound a contractible manifold; thus  $\tau$ extends topologically. Using the invariants $h_\tau$ and $h_{\ita}$, we prove that there is an infinite linearly independent family of such classes:

\begin{theorem}\label{thm:1.3}
The kernel of the forgetful homomorphism 
\[
F:\G\rightarrow \Theta^3_\Z
\]
contains a $\Z^\infty$-subgroup spanned by classes $[(Y_i, \tau_i)]$, where each $Y_i$ bounds a contractible Stein manifold. 
\end{theorem}

\noindent In other words,  strongly non-extendable involutions on the boundaries of corks generate an infinite-rank subgroup of $\G$. The pairs $(Y_i, \tau_i)$ are drawn  from the work of Akbulut and Yasui \cite{AY}, who showed that they  arise as the boundaries of Stein corks \cite{AY}; see Figure~\ref{fig:1.4}. We expect that the above classes in fact generate a $\mathbb{Z}^\infty$-summand of $\G$. 

Double branched covers of knots and links provide another interesting source of manifolds equipped with involutions. There is a well-known homomorphism 
\[ 
\Sigma_{2}: \mathcal{C} \rightarrow  \smash{\Theta^3_{\Z/2\Z}}
\]
from the smooth concordance group of knots to $\smash{\Theta^3_{\Z/2\Z}}$, given by sending the concordance class of a knot to the homology cobordism class of its double branched cover. In analogy with $\G$, one can define a $\Z/2\Z$-homology bordism group of involutions, which we denote by $\smash{\Theta^\tau_{\Z/2\Z}}$. Remembering the data of the covering action gives a refinement of $\Sigma_{2}$, in the sense that we have a factorization
\[
\begin{tikzpicture}[node distance=2cm]

\node (A) at (-1, 0) {$\mathcal{C}$};
\node (B) at (3, 0) {$\Theta^3_{\Z/2\Z}$};
\node (C) at (1, -1.4) {$\Theta^\tau_{\Z/2\Z}$};

\path[->,font=\scriptsize,>=angle 90]
(A) edge node[above]{$\Sigma_2$} (B)
(A) edge node[below left]{$\Sigma_2^{\tau}$} (C)
(C) edge node[below right]{$F$} (B);

\end{tikzpicture}
\]
\noindent
A secondary aim for our study is  to provide invariants of $\smash{\Theta^\tau_{\Z/2\Z}}$ which capture this additional concordance information. (Related ideas were independently considered and developed in \cite{AKS}.) The methods of this paper apply equally well to this modified group, and in fact all of our examples are strongly non-extendable over any $\Z/2\Z$-homology ball. To see that the map into $\smash{\Theta^\tau_{\Z/2\Z}}$ indeed constitutes a (significantly) stronger invariant than $\Sigma_{2}$,  note that the pairs $(Y_i, \tau_i)$ of Theorem~\ref{thm:1.3} can be constructed as double branched covers of knots. Hence:

\begin{corollary}\label{cor:branched}
 The kernel of the double branched cover homomorphism 
\[
\Sigma_{2}:\mathcal{C}\rightarrow  \smash{\Theta^3_{\Z/2\Z}}
\]
contains a $\Z^\infty$-subgroup which maps injectively into $\smash{\Theta^\tau_{\Z/2\Z}}$. In fact, we can choose this subgroup so that all of its elements are represented by knots whose double branched covers bound contractible manifolds. 
\end{corollary}
\noindent
It would be interesting to detect such subgroups that simultaneously lie in the subgroup of topologically slice knots. 

The construction of $h_{\tau}(Y)$ and $h_{\ita}(Y)$ is a fairly straightforward application of  the work of Hendricks, Manolescu, and Zemke \cite{HMZ}, in conjunction with naturality results of Juh{\'a}sz, Thurston, and Zemke \cite{JTZ}. In general, however, computing the action of the mapping class group on $\CFm(Y)$ is very difficult. Thus, the majority of this paper will be devoted to finding methods for indirectly constraining $h_{\tau}(Y)$ and $h_{\ita}(Y)$. To this end, a useful feature of the group $\Inv$ is that it is equipped with a partial order (see Definition~\ref{def:2.5}). A key tool will be the following monotonicity theorem satisfied by our invariants in the presence of equivariant negative-definite cobordisms:
\begin{theorem}\label{thm:1.4}
Let $(Y_1, \tau_1)$ and $(Y_2, \tau_2)$ be homology spheres equipped with involutions $\tau_1$ and $\tau_2$. 
\begin{enumerate}
\item Suppose there is a $\spinc$-fixing $(-1)$-cobordism from $(Y_1, \tau_1)$ to $(Y_2, \tau_2)$. Then we have $h_{\tau_1}(Y_1) \leq h_{\tau_2}(Y_2)$.
\item Suppose there is a $\spinc$-conjugating $(-1)$-cobordism from $(Y_1, \tau_1)$ to $(Y_2, \tau_2)$. Then we have $h_{\ita_1}(Y_1) \leq h_{\ita_2}(Y_2)$.
\item Suppose there is an interchanging $(-1, -1)$-cobordism from $(Y_1, \tau_1)$ to $(Y_2, \tau_2)$. Then we have $h_{\tau_1}(Y_1) \leq h_{\tau_2}(Y_2)$ and $h_{\ita_1}(Y_1) \leq h_{\ita_2}(Y_2)$.
\end{enumerate}
\end{theorem}
\noindent
For definitions of the terms appearing in Theorem~\ref{thm:1.4}, see Section~\ref{sec:4.2}. Theorem~\ref{thm:1.4} is a special case of a general monotonicity result for equivariant definite cobordisms, which we give in Proposition~\ref{lem:3.7}. We emphasize Theorem~\ref{thm:1.4}, however, since the cobordisms in question are particularly simple and  easily constructed via equivariant surgery methods (see Section~\ref{sec:4}). We apply these techniques to  a wide variety of examples in Section~\ref{sec:5}. 

Theorem~\ref{thm:1.4} enables one to bound the invariants of $(Y, \tau)$ away from zero by constructing cobordisms from $(Y, \tau)$ to manifolds that we better understand. These inequalities should be thought of as an analogue of the usual behavior of the Ozsv\'ath-Szabo $d$-invariant under negative-definite cobordisms \cite{OSabsgr}.  Theorem~\ref{thm:1.4}  will allow us to use  topological arguments to construct a wide range of interesting and nontrivial examples for which direct computation would be quite difficult. For instance, one of our principal ways of finding corks will be to consider surgeries on equivariant slice knots. In this vein, we have the following consequence of Theorem~\ref{thm:1.4}:

\begin{theorem}\label{thm:1.5}
Let $K$ be a knot in $S^3$ equipped with a strong inversion $\tau$, and let $k \neq -2, 0$. Then we have both $h_\tau(S_{1/(k+2)}(K)) \leq h_\tau(S_{1/k}(K))$ and $h_{\ita}(S_{1/(k+2)}(K)) \leq h_{\ita}(S_{1/k}(K))$.
\end{theorem}
\noindent
Here, we are using the fact that a symmetry of $K$ induces a symmetry of the manifold obtained by surgery on $K$ (see Section~\ref{sec:4.1} for details). To apply Theorem~\ref{thm:1.5}, let $K$ be a slice knot with $h_\tau(S_{+1}(K)) < 0$, so that $S_{+1}(K)$ is a strongly non-extendable cork by Theorem~\ref{thm:1.1}. Using Theorem~\ref{thm:1.5}, we can bootstrap from this to conclude that the same holds for $\smash{S_{1/k}(K)}$, whenever $k$ is positive and odd. A similar statement applies in the case where $h_{\ita}(S_{+1}(K)) < 0$.

The power of our techniques are best illustrated by the diversity and complexity of examples to which they can be applied. In the next subsection, we describe a number of these calculations. We produce various infinite families of pairs $(Y, \tau)$ for which $Y$ bounds a contractible manifold but $\tau$ does not extend over any homology ball (see Theorems~\ref{thm:1.8} and \ref{thm:1.10}). We also prove that several well-known examples of corks have this property (see Theorems~\ref{thm:1.11} and Theorem~\ref{thm:P}). To demonstrate the flexibility of our method, we enumerate some sporadic families formed by going through KnotInfo \cite{KnotInfo} and considering surgeries on slice knots (see Theorem~\ref{thm:1.9}). To the best of the authors' knowledge, many of these were not previously known to be (the boundaries of) corks, strong or otherwise.

Finally, a word on the conventions of this paper: all Floer-theoretic homology groups will be taken with $\ff = \Z/2\Z$-coefficients, while all usual homology groups will be understood to be over $\Z$. When describing knots, we use the notation and orientation conventions of KnotInfo \cite{KnotInfo}.

\subsection{Computations and definitions}\label{sec:1.1}
We begin by recalling the precise definition of a cork:

\begin{definition} \label{def:1.6} Let $Y$ be an integer homology sphere equipped with an (orientation-preserving) involution $\tau: Y \rightarrow Y$, and suppose that $Y$ bounds a compact, contractible manifold $W$. We say that the triple $(Y, W, \tau)$ is a \textit{cork} if $\tau$ does not extend over $W$ as a diffeomorphism.
\end{definition}
\noindent
Some authors additionally require  that $W$ be Stein (see for example \cite{AY}), but this is less natural for our purposes. We also do not require that $W$ be embedded in any particular 4-manifold. 
\begin{definition}\cite[Section 1.2.2]{LRS}\label{def:1.7} Let $(Y,W,\tau)$ be a cork. We say that the pair $(Y, \tau)$ is a \textit{strong cork} if for \textit{any} homology ball $X$ with $Y = \partial X$, $\tau$ does not extend over $X$ as a diffeomorphism.  We refer to the involution $\tau$ as being \textit{strongly non-extendable}.
\end{definition}
\noindent
We stress that the data of a strong cork is three-dimensional (rather than four-dimensional); the presence of $W$ is simply to ensure that $Y$ bound at least one contractible manifold. Arguably, replacing ``homology ball" in the above definition with ``contractible manifold" would be more natural, but since our approach actually constrains the former, we leave Definition~\ref{def:1.7} as it is. The reader should consult the work of Lin, Ruberman, and Saveliev, in which it is shown that  Akbulut's first example of a cork is strong \cite[Theorem D]{LRS}. 

\begin{remark}\label{rem:LRS}
The proof in \cite[Section 8]{LRS} proceeds by understanding the induced action of $\tau$ on the monopole Floer homology of $Y$. Like most preceding analyses of this example, their computation relies on the fact that the Akbulut cork can be embedded in a closed 4-manifold and that the corresponding cork twist changes a smooth 4-manifold invariant (namely, the Donaldson invariants \cite{Akbulutfake}, Seiberg-Witten invariant \cite{LRS}, or the Ozsv{\'a}th-Szab{\'o} invariant \cite{AD}). (Other methods employ a more three-dimensional perspective, relying instead on a slice-Bennequin inequality \cite{AMexotic}.) In contrast, our approach uses no closed 4-manifold topology or contact topology: it is irrelevant for our purposes whether our examples embed in any closed 4-manifold (with or without nontrivial 4-manifold invariants), or whether they fill contact structures on their boundaries. Moreover, for the arguments in \cite{LRS}, it is necessary to establish various technical results for monopole Floer homology with $\Q$-coefficients, while our approach utilizes Heegaard Floer homology with $\Z/2\Z$-coefficients (so as to be able to apply involutive Heegaard Floer theory).
\end{remark}

Some sample computations are displayed below. We begin with several examples of strong corks formed by surgeries on slice knots. 

\begin{theorem} \label{thm:1.8}
For $n > 0$, let $K_{-n, n+1}$ be the family of slice doubly-twist knots displayed in Figure~\ref{fig:1.1}. For $k$ positive and odd, let $V_{n,k}$ be the $(1/k)$-surgery
\[
V_{n,k}  = 
\begin{cases}
S_{1/k}(\overline{K}_{-n, n+1}) &\text{if } n \text{ is odd}\\
S_{1/k}(K_{-n, n+1}) &\text{if } n \text{ is even}.
\end{cases}
\]
Equip $V_{n, k}$ with the indicated involutions $\tau$ and $\sigma$.\footnote{For $n$ odd, we consider the obvious involutions on the mirrored diagram.} Then $(V_{n, k}, \tau)$ and $(V_{n, k}, \sigma)$ are both strong corks. Moreover, in each case $\tau \circ \sigma$ extends as a diffeomorphism over some Mazur manifold.
\end{theorem}

\begin{figure}[h!]
\center
\includegraphics[scale=0.7]{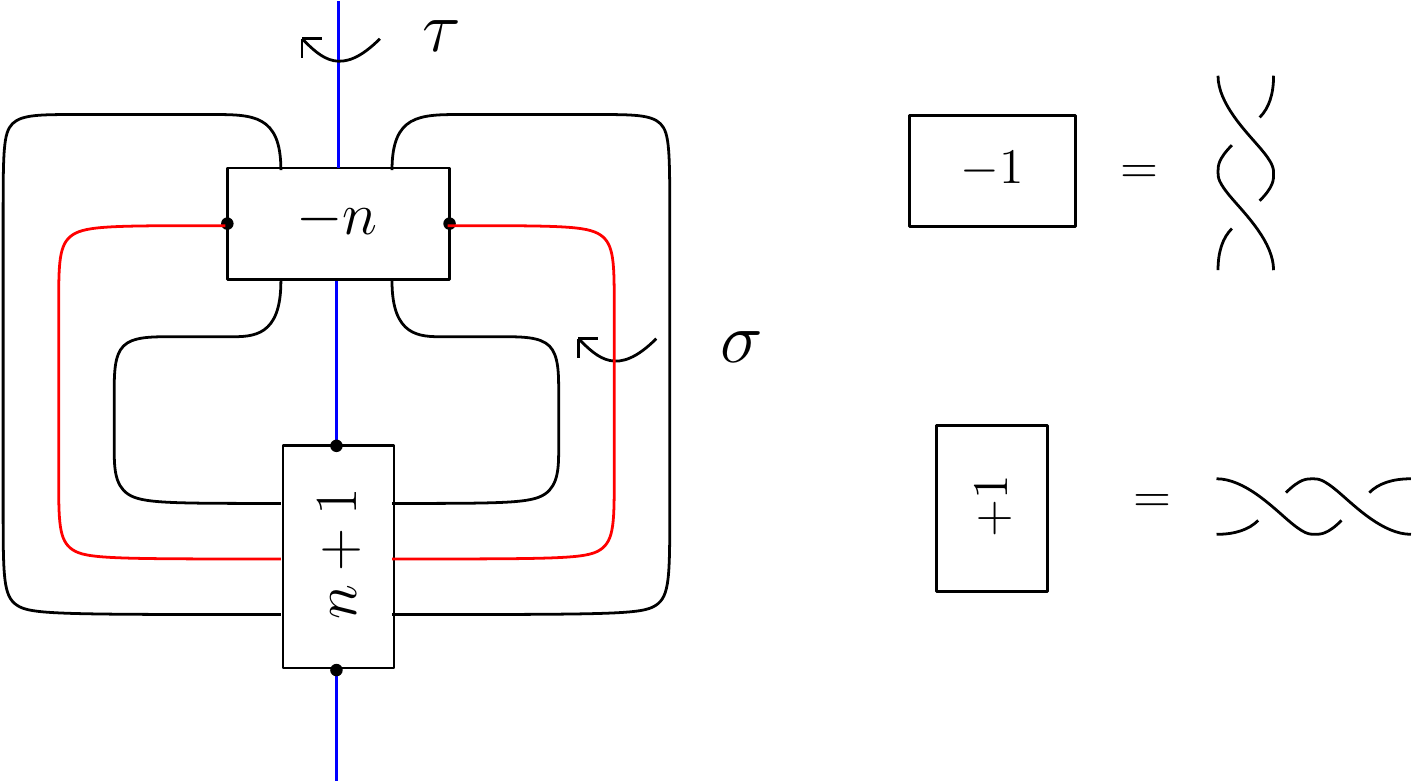}
\caption{Doubly-twist knot $K_{-n, n+1}$. The indicated symmetries $\tau$ and $\sigma$ are given by $180^{\circ}$ rotations about the blue and red axes, respectively. In the latter case, it may be helpful to view $K_{-n, n+1}$ as an annular knot; the action of $\sigma$ is given by rotation about the core of the solid torus. Black dots indicate the intersections of $K_{-n, n+1}$ with the axes of symmetry.}\label{fig:1.1}
\end{figure}
\noindent
To the best of the authors' knowledge, none of manifolds in Theorem~\ref{thm:1.8} were previously known to be (the boundaries of) corks, strong or otherwise. Note that $V_{1,1}$ is $(+1)$-surgery on the stevedore knot, which is the boundary of the Mazur manifold $W^-(0, 1)$ in the notation of Akbulut-Kirby \cite{AKir}. 

As we will see, our approach is especially conducive to situations where $Y$ is given as surgery on a slice knot. To this end, we have the following ``sporadic" families:
\begin{theorem}\label{thm:1.9}
For $k$ positive and odd, $(1/k)$-surgery on the following slice knots:
\[
\overline{9}_{41}, \overline{9}_{46}, 10_{35}, \overline{10}_{75}, 10_{155}, 11n_{49}
\]
are strong corks, with the symmetries displayed in Figure~\ref{fig:1.2}.
\end{theorem}

\begin{figure}[h!]
\center
\includegraphics[scale=0.586]{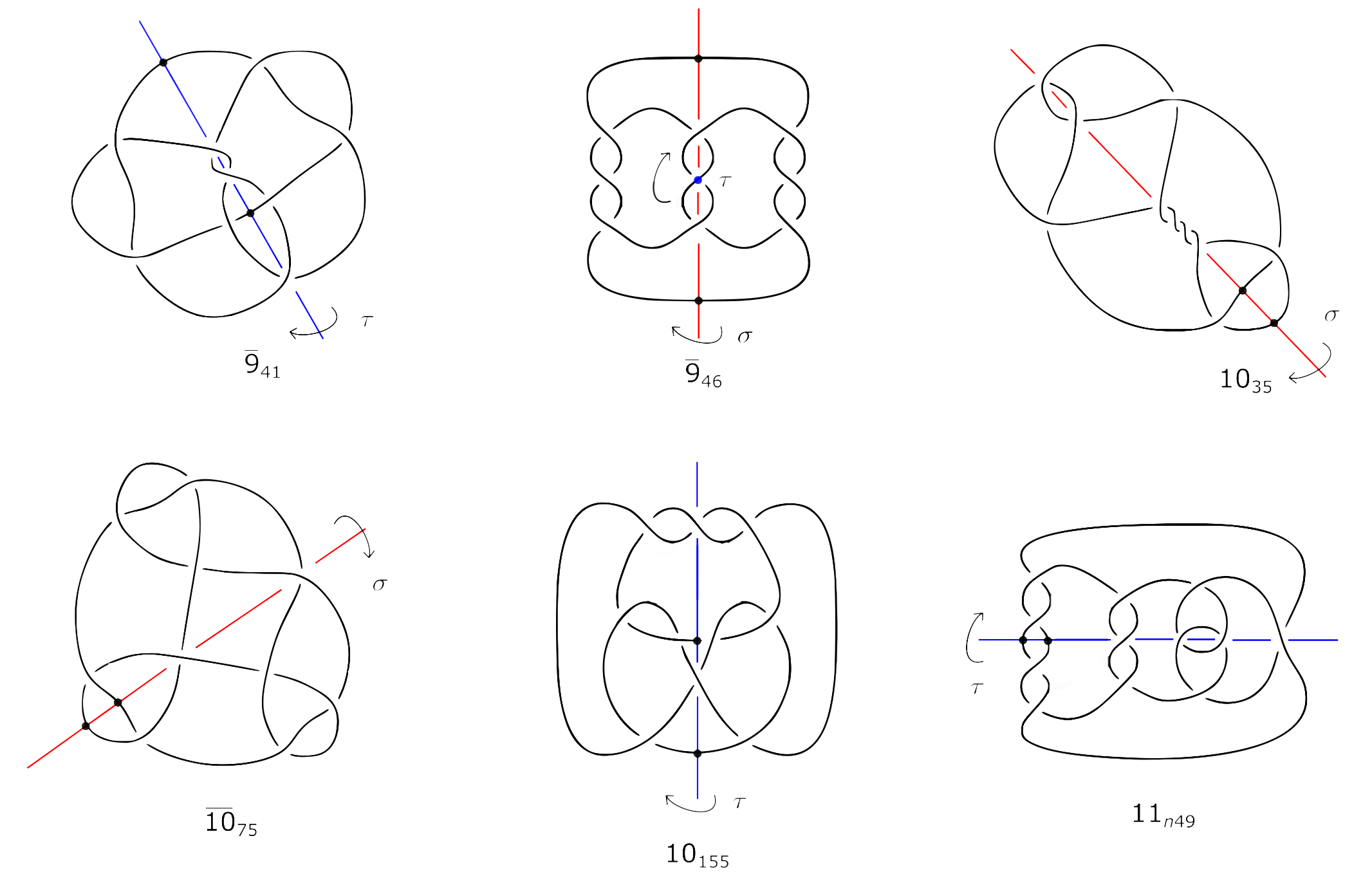}
\caption{Equivariant slice knots of Theorem~\ref{thm:1.9}. Note that $\overline{9}_{46}$ has two involutions for which Theorem~\ref{thm:1.9} holds, with the axis of $\tau$ going into/out of the page.}\label{fig:1.2}
\end{figure}
\noindent
The examples of Theorem~\ref{thm:1.9} were obtained by enumerating slice knots with fewer than twelve crossings in KnotInfo \cite{KnotInfo}, and investigating those possessing symmetric diagrams (the diagrams from \cite{LMdoubly} were helpful). The above list is by no means exhaustive. Note that $\overline{9}_{46}$ is the pretzel knot $P(-3, 3, -3)$, and $(+1)$-surgery on $\overline{9}_{46}$ is  the boundary of Akbulut's first example \cite{Akbulutfake}.

We now turn to some examples given by surgeries on links. Our first family is a straightforward generalization of the initial cork from  \cite{Akbulutfake}:
\begin{theorem}\label{thm:1.10}
For $n > 0$, let $M_n$ be the family of two-component link surgeries displayed on the left in Figure~\ref{fig:1.3}. Equip $M_n$ with the indicated involution $\tau$. Then $(M_n, \tau)$ is a strong cork. In fact, we may modify each $M_n$ by introducing any number of symmetric pairs of negative full twists, as on the right in Figure~\ref{fig:1.3}, and this conclusion still holds.
\end{theorem}

\begin{figure}[h!]
\center
\includegraphics[scale=.50]{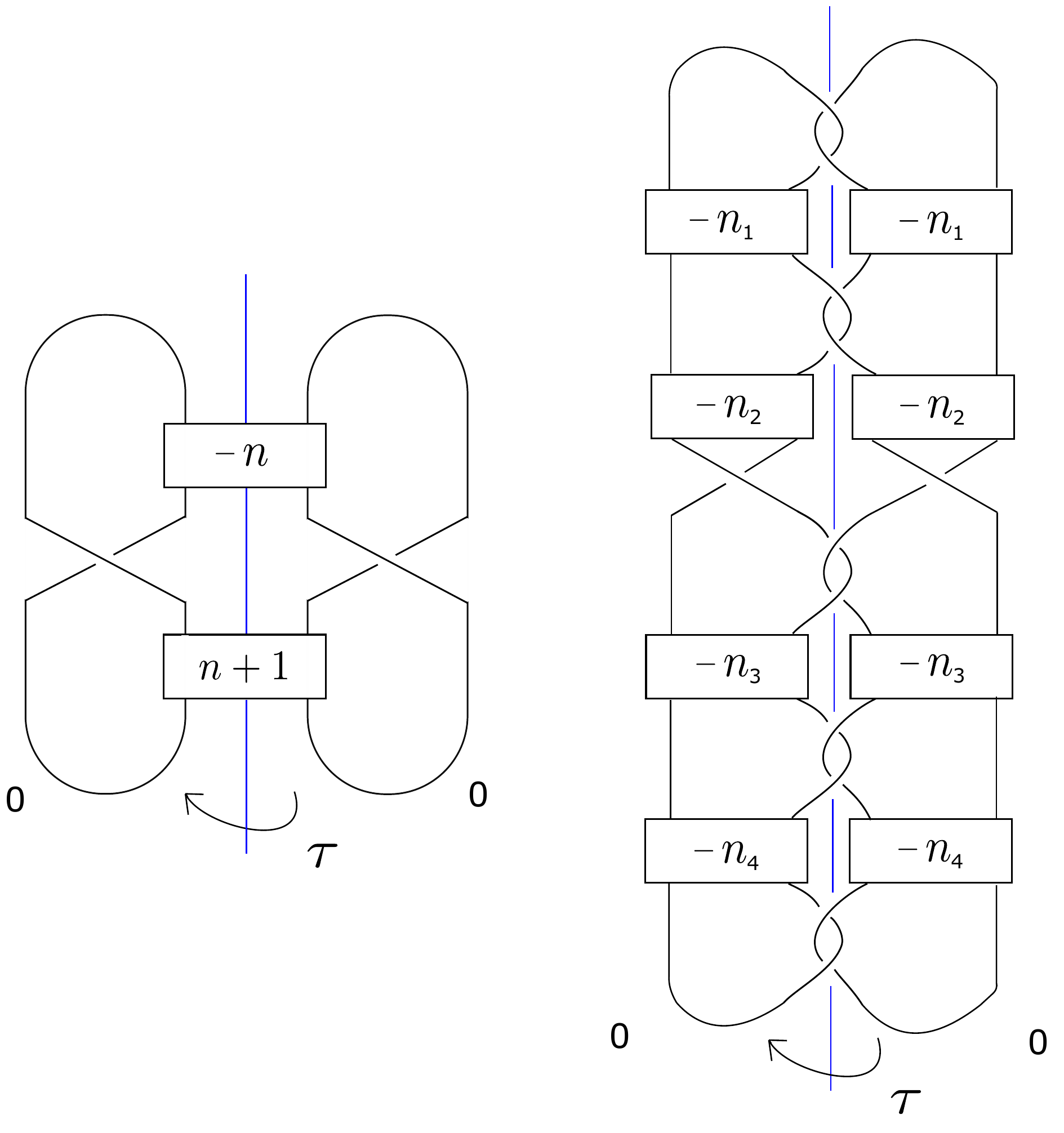}
\caption{Left: the manifold $M_n$. Right: an example of adding symmetric pairs of negative full twists to $M_2$.}\label{fig:1.3}
\end{figure}
\noindent
The family $M_n$ is similar to the family $W_n$ studied by Akbulut and Yasui in \cite{AY} (see also \cite{AKar}).\footnote{By slight abuse of notation, we use $W_n$ to refer to the boundary 3-manifold, rather than the  Mazur 4-manifold defined in \cite{AY}.} In this paper, we show that the $W_n$ are strong. The classes $[(W_n, \tau)]$ will be used in the proof of Theorem~\ref{thm:1.3}.

\begin{theorem}\label{thm:1.11}
For $n > 0$, let $W_n$ be the family of two-component link surgeries displayed on the left in Figure~\ref{fig:1.4}. Equip $W_n$ with the involution $\tau$ coming from the isotopy exchanging the link components. Then $(W_n, \tau)$ is a strong cork. In fact, we may modify each $W_n$ by introducing any number of symmetric pairs of negative full twists, as on the right in Figure~\ref{fig:1.4}, and this conclusion still holds.
\end{theorem}

\begin{figure}[h!]
\center
\includegraphics[scale=0.5]{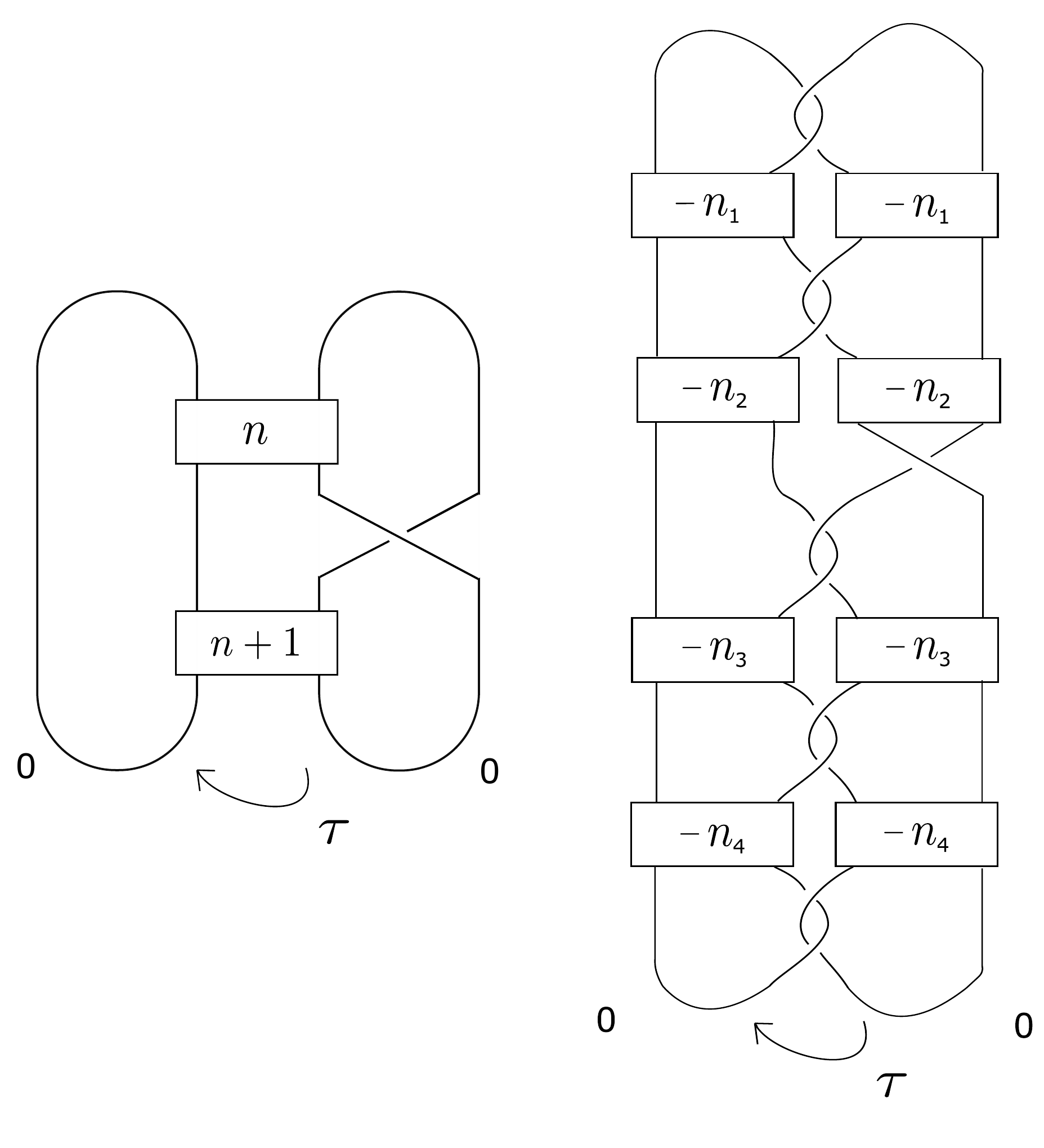}
\caption{Left: the manifold $W_n$ from \cite{AY}. (For a precise definition of $\tau$, see the discussion surrounding the proof of Theorem~\ref{thm:1.11} in Section~\ref{sec:5.4}.) Right: an example of adding symmetric pairs of negative full twists to $W_2$.}\label{fig:1.4}
\end{figure}

\noindent 
Note the above families bound Mazur manifolds \cite{Mazur}. In light of our computations, it is natural to ask:

\begin{question}\label{question:strong}
Is every cork a strong cork?
\end{question}

\noindent
Although one would expect the answer to this question to be ``no", the authors suspect that many (if not all) of the classical examples of corks currently recorded in the literature are strong. (Question~\ref{question:strong} was recently answered in the negative by Hayden and Piccirillo, who constructed several new examples of corks in \cite{HP}.) One particularly prominent class which is not fully addressed in this paper is the family of ``positron" corks introduced by Akbulut and Matveyev in \cite{AMconvex}. Here, we show (using a rather ad-hoc argument) that the first member of this family is a strong cork. We denote this example by $(P, \tau)$. Note that in the notation of \cite{AMconvex}, we have $P = \overline{W}_1$, where again we are conflating the boundary of a cork with the underlying 4-manifold. We stray from the notation of \cite{AMconvex}  due to possible confusion with the orientation reversal of $W_1$. The proof of Theorem~\ref{thm:P} may give some insight into the additional Floer-theoretic computations needed to establish further examples.

\begin{theorem}\label{thm:P}
Let $P$ be the two-component link surgery displayed in the left in Figure~\ref{fig:1.Pos}. Equip $P$ with the indicated involution $\tau$. Then $(P, \tau)$ is a strong cork. In fact, we may modify $P$ by introducing any number of symmetric pairs of negative full twists, as on the right in Figure~\ref{fig:1.Pos}, and this conclusion still holds.
\end{theorem}

\begin{figure}[h!]
\center
\includegraphics[scale=0.65]{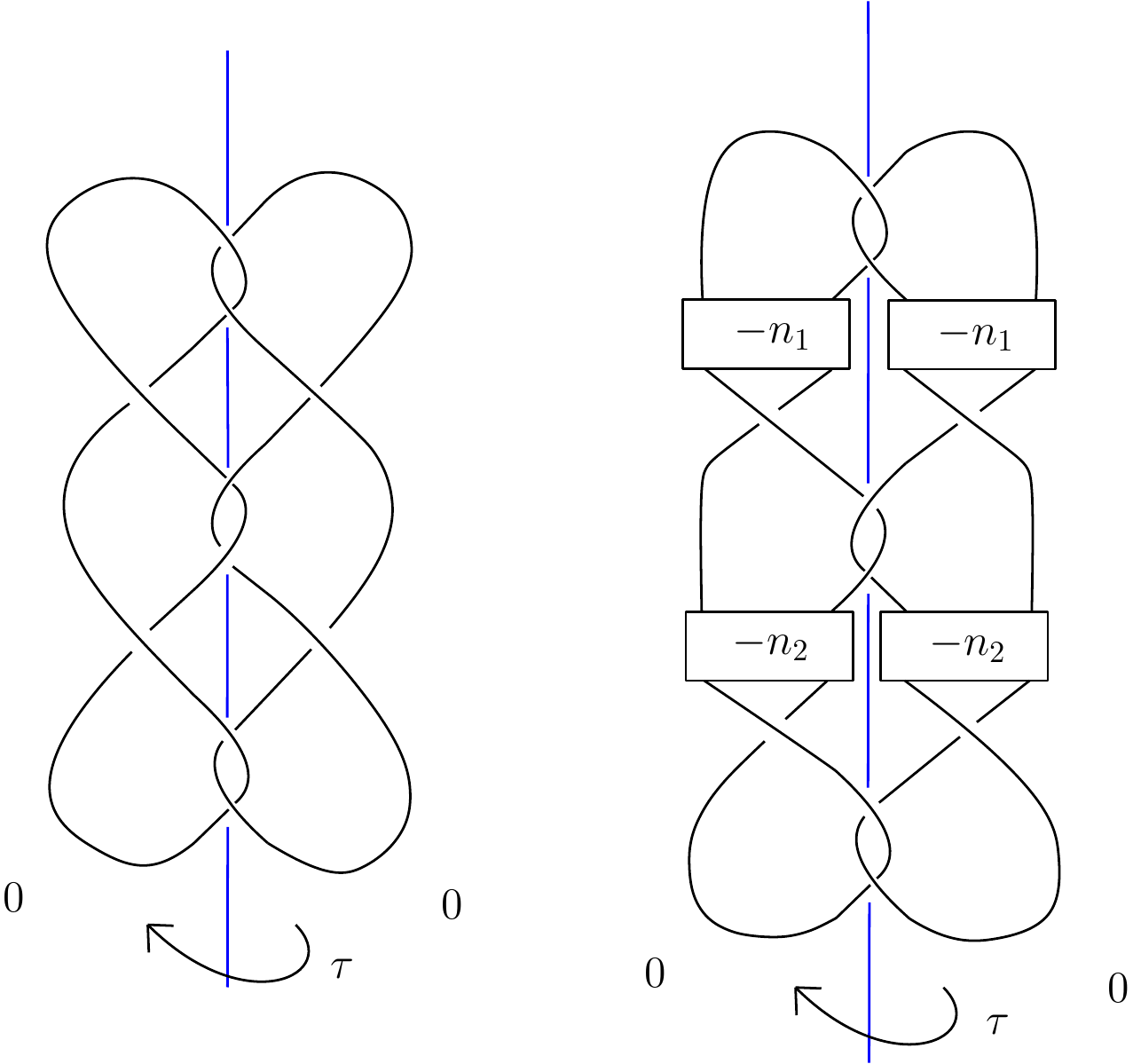}
\caption{Left: the ``positron" cork from \cite{AMconvex}. Right: adding symmetric pairs of negative full twists to $P$.}\label{fig:1.Pos}
\end{figure}

\subsection{Organization}
In the next section, we define $\G$ and discuss its topological context.  In Section~\ref{sec:2}, we then review the Floer theoretic framework underpinning our invariants, and the algebraic structures stemming from it. In Section~\ref{sec:3}, we use this framework to define our invariants $h_\tau$ and $h_{\ita}$, give some motivating examples, and prove Theorem~\ref{thm:1.1}. In Section~\ref{sec:4}, we discuss the behavior of $h_\tau$ and $h_{\ita}$ under cobordisms and prove Theorems~\ref{thm:1.4} and \ref{thm:1.5}. Section~\ref{sec:secZ} represents a somewhat technical interlude which utilizes the work of Zemke \cite{Zemkegraph} to address cobordism with disconnected ends. This leads to the proof of Theorem~\ref{thm:1.2}. Section~\ref{sec:5} represents the calculational heart of the paper, where we prove the remaining theorems by combining the algebraic formalism developed in the preceding sections with topological constructions of equivariant cobordisms.

\subsection*{Acknowledgements} The authors thank Danny Ruberman and Ian Zemke for several helpful conversations.   The authors are especially grateful to Ian  for explaining his work on ribbon graph cobordisms.  Irving Dai was supported by NSF DMS-1902746. Matthew Hedden was supported by NSF DMS-1709016.  Abhishek Mallick would like to thank Matthew Hedden and Kristen Hendricks for their support and helpfulness during the course of this project.



\section{Homology bordism group of involutions}\label{sec:homologycobordism}
In this section, we give a precise definition of $\G$ and discuss some motivation for its construction. This group naturally falls within the more general context of the bordism group of diffeomorphisms, which was popularized by Browder. Let $M_1$ and $M_2$ be two closed, oriented $n$-manifolds, each equipped with an orientation-preserving diffeomorphism $f_i$ ($i = 1, 2$). We say that $(M_1, f_1)$ and $(M_2, f_2)$ are \textit{bordant} if there exists a bordism $W$ between them which admits an orientation-preserving diffeomorphism restricting to $f_i$ on $M_i$. Here, neither the $M_i$ nor $W$ are assumed to be connected. Bordism is an equivalence relation, where transitivity follows from uniqueness of collar neighborhoods of boundary components.

\begin{definition}(\cite[pg.~22]{Browder}\label{def:1.12} or \cite[Definition 1.4]{Kreck}) The \textit{$n$-dimensional bordism group of orientation-preserving diffeomorphisms} $\Delta_n$ is the abelian group whose underlying  set consists of bordism classes of pairs $(M^n, f)$, endowed with the addition operation induced by disjoint union. The empty $n$-manifold serves as the identity, and inverses are given by orientation reversal.
\end{definition}

In analogy with $\Theta^3_{\Z}$, one would like to refine the three-dimensional group $\Delta_3$ by requiring $M$ to be a homology sphere and $W$ to be a homology cobordism. However, this presents certain technical difficulties due to the fact that the connected sum of $(M_1, f_1)$ and $(M_2, f_2)$ may not in general be well-defined. Indeed, note that in order to form $(M_1\#M_2,f_1\#f_2)$, one must first isotope each $f_i$ to fix a ball $B_i \subseteq M_i$. If $f_i$ and $f_i'$ are isotopic, then $(M_i, f_i)$ and $(M_i, f_i')$ are certainly bordant via the diffeomorphism of the cylinder $M_i \times I$ induced by the isotopy. However, it does \textit{not} follow that the homology cobordism class of $(M_1\#M_2,f_1\#f_2)$ is independent of the choice of isotopy. To see this, let $f_i$ and $f_i'$ ($i = 1, 2$) be two diffeomorphisms of $Y_i$ fixing $B_i$. Suppose that $f_i$ and $f_i'$ are isotopic, but that the intermediate stages of this isotopy do not fix any ball in $M_i$. Then it is not clear how to define a diffeomorphism on $(M_1 \# M_2) \times I$ restricting to $f_1 \# f_2$ and $f_1' \# f_2'$ at either end. We thus instead follow Definition~\ref{def:1.12} and take disjoint union to be our group operation.

In this context (i.e., the case of disconnected ends), the most natural generalization of the notion of a homology cobordism is a cobordism with the homology of an $n$-punctured $S^4$. Unfortunately, this class is not closed under composition (see Figure~\ref{fig:1.5}), so one is instead forced to consider the equivalence relation $\sim_p$ generated by all such cobordisms. Note that a composite cobordism constructed in this manner has $H_2(W) = 0$. In this paper, we will generalize this slightly and use the following (slightly coarser) equivalence relation:

\begin{figure}[h!]
\center
\includegraphics[scale=0.63]{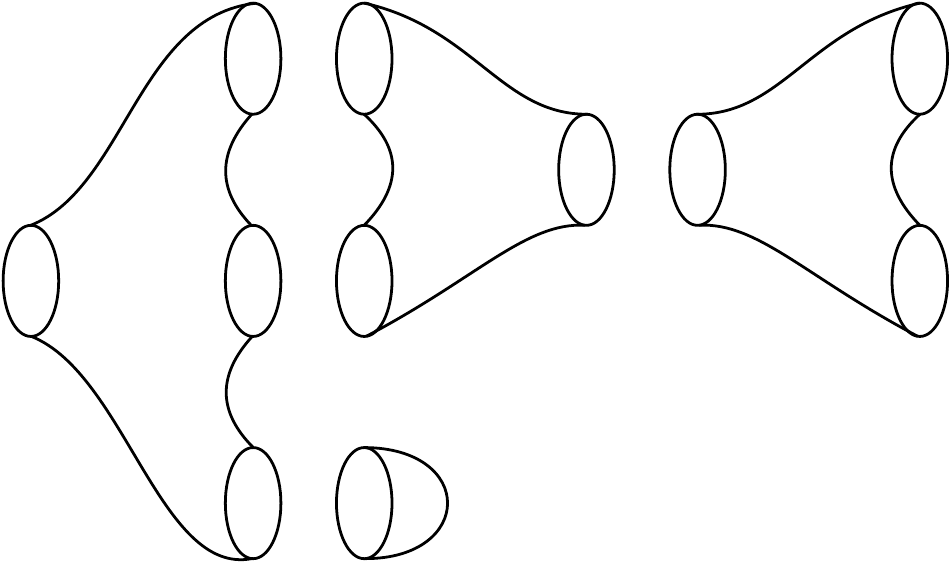}
\caption{A composition of cobordisms, each with the homology of an $n$-punctured $S^4$. Note that the composition need not have the homology of a punctured sphere.}\label{fig:1.5}
\end{figure}

\begin{definition}\label{def:1.13}
Consider the class of pairs $(Y, f)$, where:
\begin{enumerate}
\item $Y$ is a compact (possibly empty) disjoint union of oriented integer homology 3-spheres; and,
\item $f$ is an orientation-preserving diffeomorphism of $Y$ which fixes each component of $Y$ setwise.
\end{enumerate}
We say that two such pairs $(Y_1, f_1)$ and $(Y_2, f_2)$ are \textit{pseudo-homology bordant} if there exists a pair $(W, g)$ with the following properties:
\begin{enumerate}
\item $W$ is a compact, oriented cobordism between $Y_1$ and $Y_2$ with $H_2(W) = 0$; and,
\item $g$ is an orientation-preserving diffeomorphism of $W$ such that:
\begin{enumerate}
\item $g$ restricts to $f_i$ on each $Y_i$; and,
\item $g$ induces the identity map on $H_1(W, \partial W)$.
\end{enumerate}
\end{enumerate}
In this situation, we write $(Y_1, f_1) \sim (Y_2, f_2)$. It is clear that $\sim$ is an equivalence relation.
\end{definition}

\begin{remark}\label{rem:homological}
Note that $H^2(W) = H_2(W, \partial W) = H_2(W) = 0$. In particular, $W$ has only one $\spinc$-structure. The vanishing of the second homology and cohomology groups imply, by Lefschetz duality and the universal coefficients theorem, that all absolute and relative homology groups are free. An easy argument furthermore shows that $f$ must act as the identity on all homology groups $H_*(W)$ and all relative groups $H_*(W, \partial W)$.
\end{remark}

We will give some motivation for Definition~\ref{def:1.13} shortly. However, the reader should note that if $W$ has the homology of an $n$-punctured $S^4$, then $g$ automatically acts as the identity on $H_1(W, \partial W)$, using the fact that $f_i$ fixes each connected component of $Y_i$ setwise. Hence Definition~\ref{def:1.13} is a generalization of the equivalence relation $\sim_p$ discussed in the previous paragraph.\footnote{Unfortunately, the authors do not have an example showing that $\sim$ and $\sim_p$ are actually distinct.}

\begin{definition}\label{def:homologydiff}
The \textit{(three-dimensional) homology bordism group of orientation-preserving diffeomorphisms} $\Gdiff$ is the abelian group whose underlying set consists of pseudo-homology bordism classes of pairs $(Y, f)$ as in Definition \ref{def:1.13}, endowed with the addition operation induced by disjoint union. The empty 3-manifold serves as the identity, and inverses are given by orientation reversal.\end{definition}

For the reader more comfortable with the monoidal operation of connected sum, we make two remarks.  The first is that every diffeomorphism of $S^3$ extends over $B^4$, by Cerf \cite{Cerf}. Thus, instead of taking the empty set as the identity, we may take the class of $S^3$, equipped with any diffeomorphism.  Second, suppose we are given pairs $(Y_1,f_1)$ and $(Y_2,f_2)$ for which there exist $f_i$-equivariant balls $B_i\subset Y_i$ and an $f_i$-equivariant diffeomorphism from $B_1$ to $B_2$. Then we can form the connected sum $(Y_1 \# Y_2, f_1 \# f_2)$. It is evident that in this situation
\[
(Y_1, f_1) \sqcup (Y_2, f_2) \sim (Y_1 \# Y_2, f_1 \# f_2),
\]
using the cobordism formed by attaching a 1-handle to the outgoing end of $(Y_1 \sqcup Y_2) \times I$. (We refer to such a cobordism as a \textit{connected sum cobordism}. Note that this has the homology of a thrice-punctured $S^4$.) Thus, disjoint union agrees with the connected sum, although the latter is not always well-defined.

The reader may wonder as to the requirement that $g$ act as the identity on $H_1(W, \partial W)$.\footnote{Again, note that this is always satisfied if $W$ has the homology of a punctured $S^4$ or is a composition of such cobordisms.} It turns out that this condition will be crucial when applying various Floer-theoretic functoriality results of Zemke during the proof of Theorem~\ref{thm:1.2}. (See Section~\ref{sec:secZ}.) Indeed, we have chosen Definition~\ref{def:1.13} to be the coarsest possible equivalence relation under which $h_\tau$ and $h_{\ita}$ are invariant, in the sense that if $(Y_1, \tau_1)$ and  $(Y_2, \tau_2)$ are related by $\sim$, then $h_{\tau_1}(Y_1) = h_{\tau_2}(Y_2)$ and $h_{\ita_1}(Y_1) = h_{\ita_2}(Y_2)$. Note that two classes which are distinguished by our invariants up to $\sim$ are of course distinguished up to any finer equivalence relation.

It is easily checked that $\Gdiff$ admits a forgetful homomorphism to $\Theta^3_{\Z}$. This provides some additional motivation for the condition that $H_2(W) = 0$ in Definition~\ref{def:1.13}:

\begin{proposition}\label{prop:forget} 
There is a surjective homomorphism 
\[
F:\Gdiff\rightarrow \Theta^3_{\Z}
\]
obtained by forgetting the data of $f_i$ and $g$ in Definition~\ref{def:1.13}.  
\end{proposition}

\begin{proof} 
Forgetting all diffeomorphism data in Definition \ref{def:1.13} clearly still yields an equivalence relation. We denote the corresponding group by $\Theta^\sqcup_{\Z}$. This is generated by integer homology spheres equipped the operation of disjoint union, where the equivalence relation consists of cobordisms with $H_2(W) = 0$. There is an obvious surjective homomorphism
\[
F':\Gdiff\rightarrow \Theta^\sqcup_{\Z}.
\]
It is easy to check that $\Theta^\sqcup_{\Z}$ is naturally isomorphic to the usual homology cobordism group $\Theta^3_{\Z}$. The isomorphism in question is given by sending a disconnected 3-manifold $Y=\sqcup_i Y_i$ to the connected sum of its path components $\#_iY_i$. To see that this respects equivalences, suppose that $W$ is a cobordism as in Definition~\ref{def:1.13} with $H_2(W) = 0$. Then one can easily turn $W$ into a homology cobordism $W'$ whose two ends are the connected sums of all the incoming and outgoing boundary components of $W$, respectively. This is done by deleting neighborhoods of arcs connecting the incoming and outgoing components of $W$, respectively, and surgering out a set of closed curves that form a basis of $H_1(W)$. Note that $H_1(W)$ is free by Remark~\ref{rem:homological}.

Given that the map  $\Theta^\sqcup_{\Z}\rightarrow \Theta^3_{\Z}$ is well-defined, it is clearly a bijective homomorphism. Indeed, it maps disjoint unions to connected sums, verifying the homomorphism property; injectivity follows from the fact that  attaching 3-handles to a homology ball (to undo any connected sums enacted by our map) yields a punctured homology sphere, which we view as a pseudo-homology bordism to the empty set;  surjectivity is obvious. 
\end{proof}
\noindent
Note that in the above proof, we do \textit{not} obtain a diffeomorphism on the homology cobordism formed by surgering out arcs and curves on $W$, unless these can be chosen to be $g$-invariant. Indeed, the reader may view the requirement that $g$ act as the identity on $H_1(W, \partial W)$ as a homological (and hence less restrictive) version of this condition.


In this paper, we will specialize to the case in which $f$ is an involution on $Y$, due to the resulting connection with the theory of corks. This yields the following subgroup of $\Gdiff$:

\begin{definition}\label{def:1.17}
Let $\G$ be the subgroup of $\Gdiff$ generated by pseudo-homology bordism classes $[(Y, \tau)]$, where $\tau$ is an involution. We call $\G$ the \textit{(three-dimensional) homology bordism group of involutions}. 
\end{definition}
\noindent
We will furthermore be interested in classes $[(Y, \tau)] \in \ker F$ for which $Y$ bounds a contractible manifold, so that $\tau$ extends topologically. 

While Definition~\ref{def:1.13} might seem rather cumbersome, the reader will not lose much by considering only individual homology spheres and homology balls. Indeed, one can think of Definition~\ref{def:1.13} simply as a generalized situation in which the Floer-theoretic techniques of this paper also happen to apply. It is in fact possible to define equivariant connected sums for all of the families we consider here; one can then compare $(Y_1, \tau_1)$ and $(Y_2, \tau_2)$ by asking whether $- Y_1 \# Y_2$ bounds a homology ball over which $\tau_1 \# \tau_2$ extends. (See the discussion of boundary sums of corks in \cite[Section 1]{AKMR}.) Any two pairs which are equivalent via this relation are easily seen to be equivalent via $\sim$, although \textit{a priori} the converse need not hold. We have thus opted for Definition~\ref{def:1.13} out of generality and also to more closely parallel the construction of $\Delta_3$. In the examples of Section~\ref{sec:1.1}, connected sums can be taken by stacking one equivariant surgery diagram above the other, and inverses are given by mirroring and negating all surgery coefficients. We record a connected sum formula in Proposition~\ref{lem:secZ.4}.

\begin{remark}\label{rem:1.19}
Note that if $Y$ is a homology sphere (bounding a contractible manifold) with involution $\tau$, then $[(Y, \tau)]$ being nonzero in $\G$ is technically stronger than $Y$ being a strong cork. This is because the nontriviality of $[(Y, \tau)]$ actually obstructs the extension of $\tau$ over any null-bordism $W$ satisfying the conditions of Definition~\ref{def:1.13}. The authors do not have an example elucidating this distinction. 
\end{remark}

\section{Floer-theoretic overview}\label{sec:2}

In this section we give a non-technical overview of the arguments and Floer-theoretic results used in this paper. For convenience, we assume throughout that $Y$ is an integer homology sphere or, where appropriate, a disjoint union of integer homology spheres.

\subsection{The framework for the invariants}\label{sec:2.1}

We begin with a bird's eye perspective of the landscape where our invariants  reside. This is the realm of  Hendricks and Manolescu's involutive Heegaard Floer homology \cite{HM}. Their construction modifies the usual Heegaard Floer homology of Ozsv\'ath and Szab\'o \cite{OS3manifolds1}, \cite{OS3manifolds2}, taking into account the conjugation action on the Heegaard Floer complex coming from interchanging the $\alpha$- and $\beta$-curves. More precisely, Hendricks and Manolescu associate an algebraic object called an \textit{$\iota$-complex} to an integer homology sphere $Y$. This is a pair $(\CFm(Y), \iota)$, where $\CFm(Y)$ is the usual Heegaard Floer complex of $Y$, and 
\[
\iota : \CFm(Y) \rightarrow \CFm(Y)
\]
is a homotopy involution on $\CFm(Y)$ defined using the above-mentioned conjugation symmetry \cite[Section 2.2]{HM}. Up to the appropriate notion of homotopy equivalence, the pair $(\CFm(Y), \iota)$ is a well-defined 3-manifold invariant. 

In \cite{HMZ}, Hendricks, Manolescu, and Zemke define an equivalence relation on the set of $\iota$-complexes, called \textit{local equivalence}. This notion captures the algebraic relationship imposed on $\iota$-complexes by the presence of a homology cobordism between homology spheres. They then consider the set
\[
\Inv = \{\text{(abstract) } \iota \text{-complexes}\} \text{ }/\text{ local equivalence}
\]
consisting of all possible $\iota$-complexes modulo local equivalence. Taking the local equivalence class of the (grading-shifted\footnote{The grading shift convention is due to the fact that (as usually defined) $\CFm(S^3)$ has uppermost generator in grading $-2$, instead of zero.}) $\iota$-complex of $Y$ thus gives an element of $\Inv$, which we denote by $h(Y)$:
\[
Y \mapsto h(Y) = [(\CFm(Y)[-2], \iota)].
\]
In \cite[Section 8]{HMZ}, it is shown that $\Inv$ admits a group structure, with the group operation being given by tensor product. The identity element, denoted throughout by $0$, is the local equivalence class of $S^3$ or, more algebraically, the complex $\mathbb{F}[U]$ with trivial differential and identity involution. With this group structure, Hendricks, Manolescu, and Zemke  show that $h$ is a homomorphism
\[
h: \Theta^3_{\Z} \rightarrow \Inv.
\]
One can analyze algebraic properties of $\Inv$ and/or the image of $h$ as a means to better understand $\Theta^3_{\Z}$. Although $\Inv$ itself is not completely understood, this strategy has been used effectively by several authors. In particular, many auxiliary invariants of $\Inv$ have been defined (see \cite{HHL}, \cite{DHSTcobordism}) and various subgroups are well-understood \cite{DS}. Here, we will leverage these results about $\Inv$ to help us understand the output of $h_\tau$ and $h_{\ita}$.

In the involutive Heegaard Floer setting, $\iota$ is constructed using the conjugation symmetry on $\CFm$, but in general any homotopy involution on $\CFm$ defines an element of $\Inv$. In particular, by work of Juh\'asz, Thurston, and Zemke \cite{JTZ}, the mapping class group of $Y$ acts naturally on the Heegaard Floer complex of $Y$.\footnote{For subtleties addressing the basepoint, see Section~\ref{sec:3.1}. We are using here the fact that if $Y$ is a homology sphere, then the basepoint-moving map on $Y$ is trivial up to $U$-equivariant homotopy, by work of Zemke \cite{Zemkegraph}.} Thus, any involution $\tau$ on $Y$ induces a homotopy involution on $\CFm(Y)$. In this situation, we additionally have a third homotopy involution on $\CFm(Y)$ given by the composition $\ita$. (Here, we abuse notation slightly and use $\tau$ to also denote the homotopy equivalence class of the induced action on $\CFm(Y)$. When we write $\ita$, we similarly mean the composition of this action with $\iota$.) If $Y$ is equipped with an involution $\tau$, we can thus replace $\iota$ with the actions of $\tau$ and $\ita$, respectively, to obtain two new elements of $\Inv$. As in the involutive Floer case, the pairs $(\CFm(Y), \tau)$ and $(\CFm(Y), \ita)$ are well-defined invariants of $(Y, \tau)$ up to an appropriate notion of homotopy equivalence. Just as $[(\CFm(Y), \iota)]$ is an invariant of homology cobordism, we moreover show that 
\[
h_\tau(Y) = [(\CFm(Y)[-2], \tau)] \text{ and } h_{\ita}(Y) = [(\CFm(Y)[-2], \ita)]
\]
are pseudo-homology bordism invariants of $(Y, \tau)$ (in the sense of Definition~\ref{def:1.13}), and constitute homomorphisms from $\G$ to $\Inv$. 

While this alteration is straightforward, the problem is that understanding the induced action of $\tau$ on $\CFm(Y)$ is generally very difficult. In this paper, we will use the following special case of \cite[Theorem 5.3]{AKS}:

\begin{theorem}\cite[Theorem 5.3]{AKS}\label{thm:2.1}
Let $Y = \Sigma(p, q, r)$ be a Brieskorn sphere, and let $\tau$ be the involution on $Y$ given by viewing $Y$ as the double branched cover of the Montesinos knot $k(p, q, r)$.\footnote{Here, we use the notation of e.g.\ \cite{Savelievfloer} for Montesinos knots.} Then $\tau \simeq \iota$, where $\iota$ is the involutive Heegaard Floer homotopy involution of Hendricks and Manolescu.
\end{theorem}
\noindent
If $Y$ is a Brieskorn sphere which bounds a homology ball, then it follows from this that $h_\tau(Y) = h(Y) = 0$ (since $\tau \simeq \iota$) and $h_{\ita}(Y) = 0$ (since $\ita \simeq \id$). Thus, Theorem~\ref{thm:2.1} cannot be used directly to find new examples of corks. Indeed, to the best of the authors' knowledge, the following is open:

\begin{question}
Do there exist corks (strong or otherwise) with boundary a Brieskorn sphere?
\end{question}
\noindent
Our approach may be taken as mild evidence that no such examples exist.

Our strategy will instead be to find examples of pairs $(Y, \tau)$ that bound contractible manifolds, but admit cobordisms to other manifolds which we better understand. We will show in Theorem~\ref{thm:1.4} that $h_\tau$ and $h_{\ita}$ are monotonic (in an appropriate sense) under a very simple set of ``cobordism moves". These will allow us to use various elementary topological manipulations to establish a wide range of interesting examples. A schematic picture of this is given in Figure~\ref{fig:2.1}. \\

\begin{figure}[h!]
\center
\includegraphics[scale=0.6]{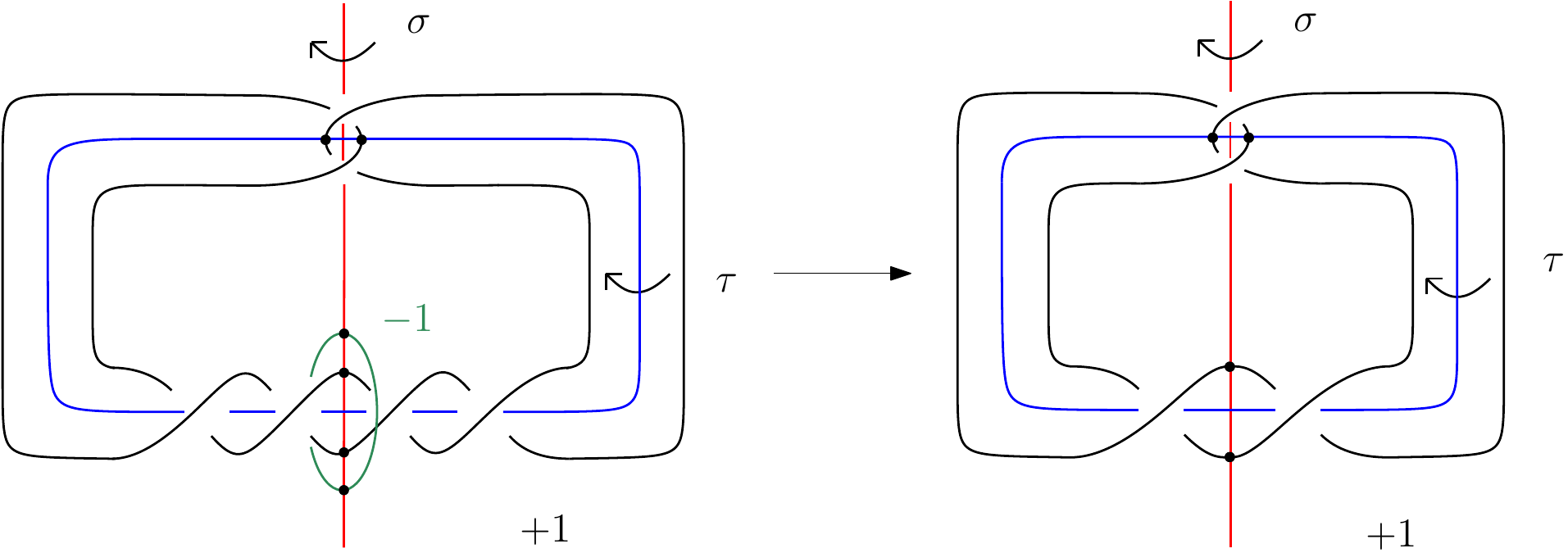}
\caption{A cobordism between two manifolds-with-involution(s). Note that on the left we have $(+1)$-surgery on the stevedore knot $6_1$ (which bounds a contractible manifold), while on the right we have $\Sigma(2, 3, 7)$. The cobordism in question is given by attaching a 2-handle along the $(-1)$-framed green curve.}\label{fig:2.1}
\end{figure}

\subsection{Local equivalence}\label{sec:2.2}
We now review the algebraic formalism of involutive Heegaard Floer homology. The reader who is already familiar with involutive Heegaard Floer homology as in \cite{HM} and \cite{HMZ} (and who has been convinced of the general setup by the preceding subsection) may wish to skip ahead to Section~\ref{sec:3}.

Throughout, we assume that $Y$ is an integer homology sphere. Let $\Hp = (H, J)$ be a Heegaard pair for $Y$, consisting of an embedded Heegaard diagram $H = (\Sigma, \alphas, \betas, z)$, together with a family of almost-complex structures $J$ on $\Sym^g(\Sigma)$. Associated to $\Hp$, we have the Heegaard Floer complex $\CFm(\Hp)$, which is generated by the intersection points $\Ta \cap \Tb$ in $\Sym^g(\Sigma)$. If $\Hp$ and $\Hp'$ are two Heegaard pairs for $Y$, then by work of Juh\'asz and Thurston \cite{JTZ}, any (basepoint-preserving) sequence of Heegaard moves relating $\Hp$ and $\Hp'$ defines a homotopy equivalence
\[
\Phi(\Hp, \Hp'): \CFm(\Hp) \rightarrow \CFm(\Hp').
\]
The map $\Phi$ is itself unique up to chain homotopy, in the sense that any two sequences of Heegaard moves define chain-homotopic maps $\Phi$. 

Now consider the conjugate Heegaard pair $\bHp = (\bh, \bJ)$. This is defined by reversing the orientation on $\Sigma$ and interchanging the $\alpha$ and $\beta$ curves to give the Heegaard splitting $\bh = (-\Sigma, \betas, \alphas, z)$. We also take the conjugate family $\bJ$ of almost-complex structures on $\Sym^g(-\Sigma)$. The points of $\Ta \cap \Tb$ are in obvious correspondence with the points of $\Tb \cap \Ta$, and $J$-holomorphic disks with boundary on $(\Ta, \Tb)$ are in bijection with $\bJ$-holomorphic disks with boundary on $(\Tb, \Ta)$. This yields a canonical isomorphism
\[
\eta: \CFm(\Hp) \rightarrow \CFm(\bHp).
\]
Note that this is \textit{not} the map $\Phi(\Hp, \bHp)$ defined in the previous paragraph. Instead, defining
\[
\iota = \Phi(\bHp, \Hp) \circ \eta,
\]
we obtain a chain map from $\CFm(\Hp)$ to itself. In \cite[Lemma 2.5]{HM}, it is shown that $\inv$ is a homotopy involution. We formalize this in the following:

\begin{definition}\cite[Definition 8.1]{HMZ}\label{def:2.2}
An {\em $\inv$-complex} is a pair $(C, \inv)$, where
\begin{enumerate}
\item $C$ is a (free, finitely generated, $\Z$-graded) chain complex over $\ff[U]$, with
\[
U^{-1}H_*(C) \cong \ff[U, U^{-1}].
\]
Here, $U$ has degree $-2$.
\item $\inv : C \to C$ is a ($\ff[U]$-equivariant, grading-preserving) homotopy involution; that is, $\inv^2$ is $U$-equivariantly chain homotopic to the identity.
\end{enumerate}
Two $\inv$-complexes $(C, \inv)$ and $(C', \inv')$ are called {\em homotopy equivalent} if there exist chain homotopy equivalences
\[
f : C \to C', \ \ g : C' \to C
\]
that are homotopy inverses to each other, and such that 
\[
f \circ \inv \simeq \inv' \circ f,  \ \ \ g \circ \inv' \simeq \inv \circ g,
\]
where $\simeq$ denotes $\ff[U]$-equivariant chain homotopy.
\end{definition}
\noindent
In \cite[Section 2.3]{HM}, it is shown that the homotopy equivalence class of $(\CFm(\Hp), \iota)$ is independent of $\Hp$. Thus, we may unambiguously talk about the homotopy type of the $\iota$-complex $(\CFm(Y), \iota)$.

In order to study homology cobordism, we introduce the following definition:

\begin{definition}\cite[Definition 8.5]{HMZ}\label{def:2.3}
Two $\inv$-complexes $(C, \inv)$ and $(C', \inv')$ are called {\em locally equivalent} if there exist ($U$-equivariant, grading-preserving) chain maps
\[
f : C \to C', \ \ g : C' \to C
\]
such that 
\[
f \circ \inv \simeq \inv' \circ f,  \ \ \ g \circ \inv' \simeq \inv \circ g,
\]
and $f$ and $g$ induce isomorphisms on homology after localizing with respect to $U$. We call a map $f$ as above a \textit{local map} from $(C, \inv)$ to $(C', \inv')$, and similarly we refer to $g$ as a local map in the other direction.
\end{definition}
\noindent
Note that this is a strictly weaker equivalence relation than homotopy equivalence. If $Y_1$ and $Y_2$ are homology cobordant, then their $\iota$-complexes are locally equivalent, with the maps $f$ and $g$ being the usual maps in Heegaard Floer theory induced by a cobordism. 

In \cite[Proposition 8.8]{HMZ}, it is shown that the set of all (abstract) $\iota$-complexes modulo local equivalence forms a group:

\begin{definition}\cite[Section 8.3]{HMZ}\label{def:2.4}
Let $\Inv$ be the set of $\inv$-complexes up to local equivalence. This has a multiplication given by tensor product, which sends (the local equivalence classes of) two $\inv$-complexes $(C_1, \inv_1)$ and $(C_2, \inv_2)$ to (the local equivalence class of) their tensor product complex $(C_1 \otimes C_2, \inv_1 \otimes \inv_2)$. The identity element of $\Inv$ is given by the trivial complex consisting of a single $\ff[U]$-tower starting in grading zero, together with the identity map on this complex. Inverses in $\Inv$ are given by dualizing.
\end{definition}
\noindent
In \cite[Theorem 1.8]{HMZ}, Hendricks, Manolescu, and Zemke additionally show that the map $h(Y) = [(\CFm(Y)[-2], \iota)]$ sending a pair $(Y, \s)$ to the local equivalence class of its (grading-shifted) $\inv$-complex is a homomorphism
\[
h: \Theta_{\Z}^3 \rightarrow \Inv.
\]
We stress that the local equivalence class of any $(C, \iota)$ obviously depends on the choice of $\iota$; in particular, if $\iota$ is homotopic to the identity then $(C, \iota)$ is locally equivalent to the trivial complex, up to overall grading shift.

Finally, one very important property of $\Inv$ is that it comes equipped with a partial order:

\begin{definition}\label{def:2.5}
Let $(C_1, \inv_1)$ and $(C_2, \inv_2)$ be two $\iota$-complexes. If there is a local map $f: C_1 \rightarrow C_2$, then we write $(C_1, \inv_1) \leq (C_2, \inv_2)$. If, in addition, there does \textit{not} exist any local map from $(C_2, \inv_2)$ to $(C_1, \inv_1)$, we write the strict inequality $(C_1, \inv_1) < (C_2, \inv_2)$. 
\end{definition}
\noindent
Since the composition of two local maps is local, it is clear that the above definition respects local equivalence. Because the tensor product of two local maps is also local, this partial order respects the group structure on $\Inv$. Note that it is \textit{not} always true that a given $\iota$-complex can be compared to the trivial complex. That is, Definition~\ref{def:2.5} does not define a \textit{total} order on $\Inv$. See \cite[Example 2.7]{DHSTcobordism} for further discussion.

It is often helpful to think of an $\iota$-complex in terms of its homology and the induced action of $\iota$. In general, of course, if $(C, \iota)$ is an $\iota$-complex, then the homological action of $\iota$ on $H_*(C)$ does not determine the chain homotopy type of $\iota$. However, in certain simple cases, it turns out that knowing $H_*(C)$ and the induced action $\iota_*$ suffices to recover the homotopy equivalence class of $(C, \iota)$. In particular, if $H_*(C)$ is concentrated in a single mod $2$ grading, then one can combinatorially write down a model for $(C, \iota)$ (called the \textit{standard complex}) which is correct up to homotopy equivalence. In such situations, we will thus sometimes blur the distinction between $(H_*(C), \iota_*)$ and $(C, \iota)$. See \cite[Section 4]{DM} for precise definitions and further discussion.

\begin{remark}\label{rem:2.H}
To rule out the existence of a local map from $(C_1, \iota_1)$ to $(C_2, \iota_2)$, it suffices to prove that there is no $\ff[U]$-module map $F$ from $H_*(C_1)$ to $H_*(C_2)$ such that:
\begin{enumerate}
\item $F$ maps $U$-nontorsion elements to $U$-nontorsion elements; and,
\item $F$ intertwines the actions of $(\iota_1)_*$ and $(\iota_2)_*$.
\end{enumerate}
However, in general Definition~\ref{def:2.5} is strictly stronger than the existence of such an $F$. For example, the complex $- X_1$ (see Example~\ref{ex:2.9} below) is strictly greater than zero, but this cannot be proven using only the action of $\iota_*$ on homology.
\end{remark}

\subsection{Connected homology} \label{sec:2.3}
In this subsection, we summarize the construction of connected Floer homology, defined by Hendricks, Hom, and Lidman \cite{HHL}. This associates to an $\iota$-complex $(C, \iota)$ an $\ff[U]$-module $H_\mathrm{conn}(C)$ whose isomorphism type is invariant under local equivalence. While $H_\mathrm{conn}(C)$ is \textit{a priori} strictly weaker than the local equivalence class of $(C, \iota)$, it is perhaps somewhat easier to understand.

\begin{definition}\cite[Definition 3.1]{HHL}\label{def:2.6}
Let $(C, \inv)$ be an $\inv$-complex. A \textit{self-local equivalence} is a local map from $(C, \inv)$ to itself; that is, a (grading-preserving) chain map $f: C \to C$ such that $f \circ \inv \simeq \inv \circ f$ and $f$ induces an isomorphism on homology after inverting the action of $U$.
\end{definition}
\noindent
Hendricks, Hom, and Lidman define a preorder on the set of self-local equivalences of $(C, \inv)$ by declaring $f \lesssim g$ whenever $\ker f \subseteq \ker g$. A self-local equivalence $f$ is then said to be \textit{maximal} if $g \gtrsim f$ implies $g \lesssim f$ for any self-local equivalence $g$.

Maximal self-local equivalences should heuristically be thought of as producing a local equivalence between $C$ and some very small subcomplex of $C$ given by $\im f$. (This subcomplex is small because $\ker f$ is large.) Note, however, that since $\inv$ is a homotopy involution and $f$ commutes with $\inv$ only up to homotopy, the action of $\inv$ need not preserve $\im f$, and similarly for the relevant homotopy maps. However, according to \cite[Lemma 3.7]{HHL}, we can modify $\inv$ to produce an actual homotopy involution $\inv_f$ on $\im f$ such that $(C, \inv)$ and $(\im f, \inv_f)$ are locally equivalent. 

Hendricks, Hom, and Lidman further show that a maximal self-local equivalence always exists \cite[Lemma 3.3]{HHL}, and that any two maximal self-local equivalences give homotopy equivalent $\inv$-complexes $(\im f, \inv_f)$ and $(\im g, \inv_g)$ \cite[Lemma 3.8]{HHL}. We can thus make the unambiguous definition:

\begin{definition}\cite[Definition 3.9]{HHL}\label{def:2.7}
Let $(C, \inv)$ be an $\inv$-complex. The \textit{connected complex of} $(C, \inv)$, which we denote by $(C_{\mathrm{conn}}, \inv_{\mathrm{conn}})$, is defined to be $(\im f, \inv_f)$, where $f$ is any maximal self-local equivalence. This is well-defined up to homotopy equivalence of $\inv$-complexes.
\end{definition}
\noindent
Note that since $(C, \inv)$ and $(C_{\mathrm{conn}}, \inv_{\mathrm{conn}})$ are locally equivalent, the connected complex is in fact an invariant of the local equivalence class of $(C, \inv)$. Indeed, the connected complex should be thought of as the simplest possible local representative of $(C, \inv)$. The connected homology is then the torsion submodule of the homology of this complex, shifted up in grading by one:

\begin{definition}\cite[Definition 3.13]{HHL}\label{def:2.8}
Let $(C, \inv)$ be an $\inv$-complex. The \textit{connected homology of} $(C, \inv)$, denoted by $H_\mathrm{conn}(C)$, is defined to be the $U$-torsion submodule of $H_*(C_{\mathrm{conn}})$, shifted up in grading by one. Here, $H_*(C_{\mathrm{conn}})$ is the usual homology of $C_{\mathrm{conn}}$ as a $\ff[U]$-complex.
\end{definition}
\noindent
The \textit{connected Heegaard Floer homology of $Y$}, denoted  $\HF_\mathrm{conn}(Y)$, is then defined to be the connected homology of $(\CFm(Y), \inv)$. The grading shift by one is enacted so that $\HF_\mathrm{conn}$ is a summand of $\HF_\mathrm{red}$, viewed as a quotient of $\HFp$. 

\subsection{Simple families}\label{sec:2.4}
We conclude our overview by discussing a  simple family of $\iota$-complexes that will serve to illustrate the formalism at hand. As these examples will recur throughout the paper, we encourage the reader to develop some familiarity with them.

\begin{example}\label{ex:2.9}
For $i > 0$, consider the chain complex spanned by the generators $v, \iota v,$ and $\alpha$, with
\[
\partial \alpha = U^i(v + \iota v).
\]
Here, $v$ and $\iota v$ lie in Maslov grading zero, while $\alpha$ has grading $-2i + 1$. The action of $\iota$ interchanges $v$ and $\iota v$ and fixes $\alpha$. We denote this $\iota$-complex (or sometimes its local equivalence class) by $X_i$. The homology of $X_i$ is displayed in Figure~\ref{fig:2.2}; note that the induced action of $\iota$ is given by the obvious involution reflection through the vertical axis. 

\begin{figure}[h!]
\center
\includegraphics[scale=1.1]{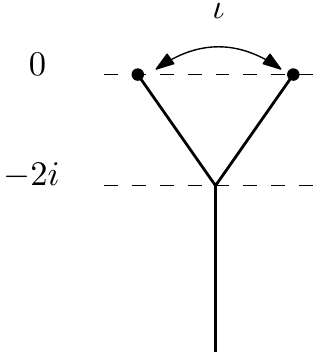}
\caption{Homology of $X_i$, expressed as a graded root with involution. Vertices of the graph correspond to $\mathbb{F}$-basis elements supported in grading given by the height (shown on the left). Edges between vertices indicate the action of $U$, and we suppress all vertices forced by this relation. Thus, for instance, the two upper legs of the graded root contain $i$ vertices (excluding the symmetric vertex lying in grading $-2i$). See for example \cite[Definition 2.11]{DM}.}\label{fig:2.2}
\end{figure}

The reader should verify that the only self-local equivalences of $X_i$ are isomorphisms. The connected homology of $X_i$ is thus simply (the $U$-torsion part of) its usual homology, so that $H_\mathrm{conn}(X_i)$ is just $\ff[U]/(U^i\ff[U])$. In particular, this shows that the local equivalence classes of the $X_i$ are nonzero and mutually distinct. We can refine their distinction by considering  the partial order on $\Inv$. It is easily checked that
\[
\cdots < X_3 < X_2 < X_1 < 0,
\]
where $0$ denotes the trivial $\iota$-complex. Indeed, there is evidently a local map showing that $X_1 \leq 0$, by mapping both $v$ and $\iota v$ to $x$ and $\alpha$ to zero. However, the only $\iota$-equivariant map in the other direction sends $x$ to $v + \iota v$, which is $U$-torsion in homology (See Figure~\ref{fig:2.3}.) Thus, the inequality is strict. The proof that $X_{i+1} < X_i$ is similar.
\end{example}

\begin{figure}[h!]
\center
\includegraphics[scale=1]{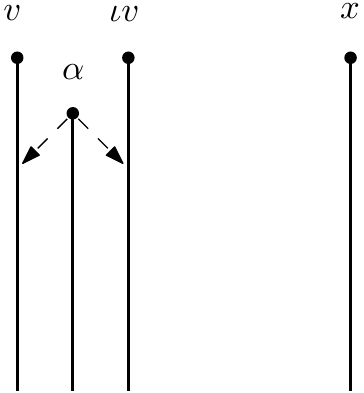}
\caption{Left: the complex $X_1$. Right: the trivial complex $0$.}\label{fig:2.3}
\end{figure}

The classes $X_i$ actually play quite an important role in the study of $\Theta^3_{\Z}$ and $\Inv$. In \cite[Theorem 1.7]{DM}, it is shown that the $X_i$ are linearly independent in $\Inv$, and in fact they span a $\Z^\infty$-summand of $\Inv$ by \cite[Theorem 1.1]{DHSTcobordism}. Here, we will use the fact that $(-1)$-surgery on the right-handed $(2, 2n+1)$-torus knots realize the $X_i$:
\[
h(S_{-1}(T_{2, 2n+1})) = X_{\lfloor (n+1)/2 \rfloor}
\]
See the proof of \cite[Theorem 1.4]{HHL}. Note that $S_{-1}(T_{2, 2n+1})$  can be identified with the  Brieskorn sphere $\Sigma(2, 2n+1, 4n+3)$. 


\section{$\tau$- and $\ita$-Local Equivalence}\label{sec:3}

\subsection{$\tau$- and $\ita$-complexes}\label{sec:3.1}
We now adapt the material of the previous section to the situation at hand. Let $Y$ be a homology sphere equipped with an involution $\tau : Y \rightarrow Y$. As discussed in Section~\ref{sec:2.1}, the idea will be to repeat the algebraic construction of involutive Heegaard Floer homology,  with the role of $\iota$ played by the induced action of $\tau$ on $\CFm(Y)$.  This alteration is fairly straightforward, but a key new feature which arises is the presence of an additional homotopy involution on $\CFm(Y)$ given by the composition $\iota \circ \tau$. Taken together, these two involutions will provide a powerful tool for establishing the non-extendability of various involutions on $Y$.

We begin by considering the induced action of $\tau$ on $\CFm(Y)$. We denote this chain homotopy involution by $\tau$ as well, and rely on the context to make clear whether $\tau$ refers to a diffeomorphism of $Y$ or a homotopy class of chain maps. The fact that an involution on $Y$ induces (the homotopy class of) a homotopy involution $\tau:\CFm(Y)\rightarrow \CFm(Y)$ follows from the work of Juh\'asz, Thurston, and Zemke \cite{JTZ}, who showed that the (based) mapping class group acts naturally on Heegaard Floer homology. However, for readers less familiar with \cite{JTZ}, we give a brief overview emphasizing the connection with Section~\ref{sec:2.2}. Let $\H$ be a choice of Heegaard data for $Y$, and suppose that $\tau$ fixes the basepoint $z$ of $\H$. Applying $\tau$ to $\H$, we obtain a ``pushforward" set of Heegaard data which we denote by $t\H$. Explicitly, we think of $\Sigma$ as embedded in $Y$, so that $\tau$ maps $\Sigma$ to another embedded surface $\tau(\Sigma)$ in $Y$ with the obvious pushforward $\alpha$- and $\beta$-curves. We similarly pushforward the family of almost complex structures $J$ on $\Sym^g(\Sigma)$ using the diffeomorphism between $\Sigma$ and $\tau(\Sigma)$ effected by $\tau$. There is a tautological chain isomorphism
\[
t: \CFm(\H) \rightarrow \CFm(t\H)
\]
given by the map sending an intersection point in $\Ta \cap \Tb$ to its corresponding pushforward intersection point. The action of $\tau$ is then defined to be the homotopy class of the chain map
\[
\tau = \Phi(t\H, \H) \circ t : \CFm(\H) \rightarrow \CFm(\H),
\]
where $\Phi(t\H, \H)$ is the Juh\'asz-Thurston-Zemke homotopy equivalence from $\CFm(t\H)$ to $\CFm(\H)$. Theorem 1.5 of \cite{JTZ} shows that induced map $\tau_*$ on homology is well-defined, and an invariant of the pointed mapping class represented by $\tau$.  The proof of their result, however, shows that the homotopy class of $\tau$ is also invariant. (See \cite[Proposition 2.3]{HM}.)

In the case that $\tau$ does not fix a point on $Y$, we first consider an isotopy $h_s: Y \rightarrow Y$ that moves $\tau z$ back to $z$ along some arc $\gamma$. Composing $\tau$ with the result of this isotopy gives an isotoped diffeomorphism $\tau_\gamma = h_1 \circ \tau$, which now fixes the basepoint $z$. We then  define the action of $\tau$ to be the mapping class group action of $\tau_\gamma$:
\[
\tau = \Phi(t_{\gamma}\H, \H) \circ t_{\gamma} : \CFm(\H) \rightarrow \CFm(\H),
\]
where $t_\gamma$ is the tautological pushforward associated to $\tau_\gamma$. The fact that this is independent of $\gamma$ follows from work of Zemke \cite{Zemkegraph}, who showed that for a homology sphere $Y$, the $\pi_1$-action on $\CFm(Y)$ is trivial up to $U$-equivariant homotopy. Explicitly, let
\[
f_{\gamma}: \CFm(t\H) \rightarrow \CFm(t_{\gamma}\H)
\]
be the pushforward map associated to isotopy along $\gamma$, so that $t_{\gamma} = f_{\gamma} \circ t$. Let $\gamma'$ be a different arc connecting $\tau z$ to $z$. Then $t_{\gamma'}\H$ is obtained from from $t_{\gamma}\H$ by an isotopy which pushes $z$ around the closed loop $\gamma^{-1} * \gamma'$. The basepoint-moving action of $\gamma^{-1} * \gamma'$ on $\CFm(t_{\gamma} \H)$ is equal to
\[
(\gamma^{-1} * \gamma')_* \simeq \Phi(t_{\gamma'}\H, t_{\gamma} \H) \circ f_{\gamma'} \circ f_{\gamma}^{-1}.
\]
Since $Y$ is a homology sphere, this is $U$-equivariantly homotopic to the identity by \cite[Theorem D]{Zemkegraph}. We thus have
\[
\Phi(t_{\gamma'}\H, \H) \circ f_{\gamma'} \simeq \Phi(t_{\gamma} \H, \H) \circ \Phi(t_{\gamma'} \H, t_{\gamma}\H) \circ f_{\gamma'} \simeq \Phi(t_{\gamma}\H, \H) \circ f_{\gamma}.
\]
Composing both sides of this with $t$ shows that $\Phi(t_{\gamma'}\H, \H) \circ t_{\gamma'} \simeq \Phi(t_{\gamma}\H, \H) \circ t_{\gamma}$, as desired. For the purposes of Floer theory, we will thus generally think of $\tau$ as having been isotoped to fix a basepoint of $Y$, and in such situations we will blur the distinction between $\tau$ and $\tau_\gamma$.

\begin{lemma}
Let $Y$ be a homology sphere equipped with an involution $\tau : Y \rightarrow Y$. Then the map $\tau : \CFm(Y) \rightarrow \CFm(Y)$ constructed above is a well-defined homotopy involution.
\end{lemma}
\begin{proof}
A similar argument as in \cite[Section 2]{HM} shows that $\tau$ is well-defined up to homotopy equivalence (upon changing the choice of Heegaard data $\H$). If $\tau : Y \rightarrow Y$ fixes the basepoint of $Y$, then the action of $\tau$ on $\CFm$ is simply defined to be the usual mapping class group action of $\tau$. In this case, the rest of the claim follows from the fact that the action of the (based) mapping class group satisfies $(f \circ g)_* \simeq f_* \circ g_*$. If $\tau$ does not have a fixed point, then the action of $\tau$ is instead defined to be the mapping class group action of $\tau_\gamma = h_1 \circ \tau$. Now, $\tau_\gamma^2$ is evidently isotopic to the identity via
\[
H_s = (h_s \circ \tau) \circ (h_s \circ \tau) : Y \rightarrow Y.
\]
However, this isotopy does \textit{not} necessarily fix the basepoint $z$, so some care is needed. Define a modified isotopy $H_s'$ as follows. For each $s$, let $a_s$ be the arc traced out by $H_r(z)$ as $r$ ranges from $s$ back to zero. At time $s$, let $H_s'$ be equal to $H_s$, followed by the result of an isotopy pushing $H_s(z)$ back to $z$ along $a_s$. Then $H_s'$ fixes $z$ for all $s$. Clearly, $H_1'$ is equal to $H_1$ composed with an isotopy pushing $z$ around the closed curve $a_1$. Since the $\pi_1$-action on $\CFm(Y)$ is trivial, it follows that the induced actions of $H_1'$ and $\tau_\gamma^2$ coincide (up to $U$-equivariant homotopy). However, the former action is homotopy equivalent to the identity, since $H_1'$ is isotopic to the identity via a basepoint-preserving isotopy.
\end{proof}

We thus obtain:
\begin{definition}\label{def:3.1}
Let $Y$ be a homology sphere with an involution $\tau$. We define the \textit{$\tau$-complex} of $(Y, \tau)$ to be the pair $(\CFm(Y), \tau)$, where $\tau : \CFm(Y) \rightarrow \CFm(Y)$ is the homotopy involution defined above. This is well-defined up to the notion of homotopy equivalence in Definition~\ref{def:2.2}. We denote the local equivalence class of this complex by
\[
h_{\tau}(Y) = [(\CFm(Y)[-2], \tau)].
\]
Applying the construction of Section~\ref{sec:2.3}, we obtain the \textit{$\tau$-connected homology} $\Hc^\tau(Y, \tau)$. 
\end{definition}

\noindent
As we will see in Lemma~\ref{lem:3.3}, $\iota$ and $\tau$ homotopy commute. Hence their composition is another well-defined homotopy involution. We thus have:

\begin{definition}\label{def:3.2}
Let $Y$ be a homology sphere with an involution $\tau$. We define the \textit{$\ita$-complex} of $(Y, \tau)$ to be the pair $(\CFm(Y), \ita)$. This is well-defined up to the notion of homotopy equivalence in Definition~\ref{def:2.2}. We denote the local equivalence class of this complex by
\[
h_{\ita}(Y) = [(\CFm(Y)[-2], \ita)].
\]
Applying the construction of Section~\ref{sec:2.3}, we obtain the \textit{$\ita$-connected homology} $\Hc^{\ita}(Y)$. 
\end{definition}

Note that $\iota$ and $\tau$ homotopy commute, so nothing is gained by considering the homotopy involution $\tau \circ \iota$. This is just a re-phrasing of the fact that $\iota$ is well-defined up to homotopy (so that conjugating by any diffeomorphism replaces $\iota$ with a homotopy equivalent map), but for completeness we give the formal argument below:

\begin{lemma}\label{lem:3.3}
Let $Y$ be a homology sphere with an involution $\tau$. Then $\ita \simeq \tau \circ \iota$.
\end{lemma}
\begin{proof}
Let $\H$ be a choice of Heegaard data for $Y$, and let $t \H$ be as above. For notational convenience, let $\eta \H$ denote the conjugate Heegaard data $\smash{\bH}$ defined in Section~\ref{sec:2.2}. Note that we also have the Heegaard data $\eta t \H$, which consists of first pushing forward via $\tau$ and then interchanging the $\alpha$- and $\beta$-curves (and conjugating the almost complex structure). Similarly, we have the Heegaard data $t \eta \H$ which is formed by first conjugating and then pushing forward. However, it is evident that $\eta t\H = t \eta \H$, and moreover that $t$ and $\eta$ commute as isomorphisms of the relevant Floer complexes. Now choose any sequence of Heegaard moves from $\eta \H$ to $\H$. Taking the pushforward sequence of Heegaard moves gives the commutative diagram on the left in Figure~\ref{fig:3.1}. Similarly, choosing any sequence of Heegaard moves from $t\H$ to $\H$ and then applying $\eta$ gives the commutative diagram on the right. 

\begin{figure}[h!]
\center
\includegraphics[scale=0.91]{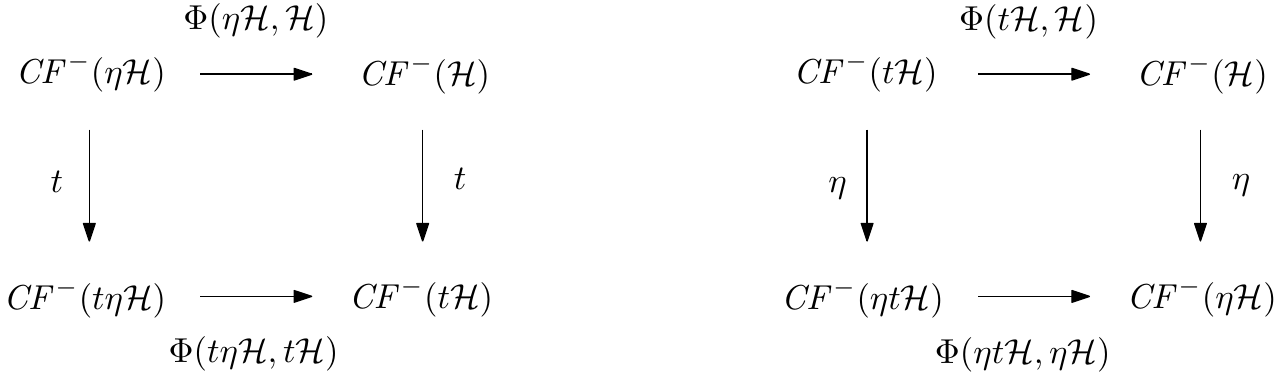}
\caption{Commutative diagrams for $t$ and $\eta$.}\label{fig:3.1}
\end{figure}

We thus have
\begin{align*}
\tau \circ \iota &= \Phi(t\H, \H) \circ t \circ \Phi(\eta\H, \H) \circ \eta \\
&= \Phi(t\H, \H) \circ \Phi(t\eta\H, t\H) \circ t \circ \eta \\
&= \Phi(t\H, \H) \circ \Phi(\eta t \H, t\H) \circ \eta \circ t \\
&\simeq \Phi(\eta\H, \H) \circ \Phi(\eta t \H, \eta\H) \circ \eta \circ t \\
&= \Phi(\eta\H, \H) \circ \eta \circ \Phi(t \H, \H) \circ t \\
&= \ita.
\end{align*}
Here, in the fourth line we have used the fact that the maps $\Phi(t\H, \H) \circ \Phi(\eta t \H, t\H)$ and $\Phi(\eta\H, \H) \circ \Phi(\eta t \H, \eta\H)$ are chain homotopic, since they are both induced by sequences of Heegaard moves from $\eta t \H$ to $\H$.
\end{proof}

\begin{remark}
In addition to the connected homology, there are several other algebraic constructions which can be associated to an $\iota$-complex $(C, \iota)$. For example, one can form the \textit{involutive complex} $I^-(C)$, which is given by the mapping cone
\[
I^-(C) = \text{Cone}(C \xrightarrow{Q(1 + \iota)} Q \cdot C[-1]).
\]
The homology $\HIm(C)$ of this is a $\ff[U, Q]$-module, where $Q$ is a formal variable of degree $-1$. One can similarly form the involutive complex associated to the connected complex $C_\textit{conn}$. Using $\HIm(C)$, we can define two numerical \textit{involutive correction terms}, denoted by $\du$ and $\dl$ (see \cite[Section 5]{HM}). These constructions are all invariants of the local equivalence class of $(C, \iota)$. We leave it to the reader to formalize the corresponding notions for $\tau$ and $\ita$.
\end{remark}

\subsection{Motivating examples} \label{sec:3.2}
We now introduce two  important examples in which we display the actions of various involutions. The proofs of these computations will be given in Section~\ref{sec:5.1}, but we state them here to so as to give some familiarity with the algebra of $\tau$- and $\ita$-complexes.

\begin{example}\label{ex:3.4}
Let $Y_1 = \Sigma(2, 3, 7)$ be given by $(+1)$-surgery on the figure-eight knot. In Figure~\ref{fig:3.2}, we have displayed two involutions (denoted by $\tau$ and $\sigma$) on $Y_1$, given by rotating the surgery diagram $180^{\circ}$ about the appropriate axes.
\begin{figure}[h!]
\center
\includegraphics[scale=0.65]{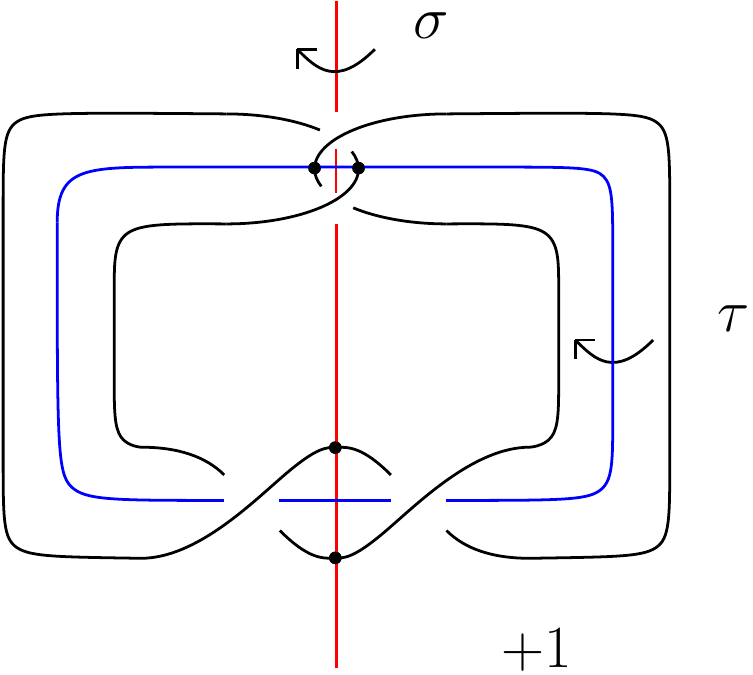}
\caption{$(+1)$-surgery on the figure eight knot, with involutions $\tau$ and $\sigma$.}\label{fig:3.2}
\end{figure}

The Floer homology $\HFm(Y_1)$ is displayed in Figure~\ref{fig:3.3}, together with the homology actions of $\tau$, $\ita$, $\sigma$, and $\iota \circ \sigma$. (Each of these is induced by an obvious action on the standard complex of Example~\ref{ex:2.9}.) Note that each map is either the identity or equal to $\iota$, as indeed these are the only two possible involutions on the Heegaard Floer complex. Thus the involutive invariants corresponding to $\ita$ and $\sigma$ are trivial, while those for $\tau$ and $\iota \circ \sigma$ are identical to those for $\iota$ in Example~\ref{ex:2.9}. In particular, as discussed in the introduction, the fact that $h_{\tau}(Y_1) \neq 0$ means that $\tau$ is non-extendable, while the fact that $h_{\iota \circ \sigma}(Y_1) \neq 0$ implies that $\sigma$ is non-extendable. Of course, since $h(Y_1) \neq 0$, the manifold in question does not bound \textit{any} homology ball, so the result in this case is trivial.

\begin{figure}[h!]
\center
\includegraphics[scale=0.9]{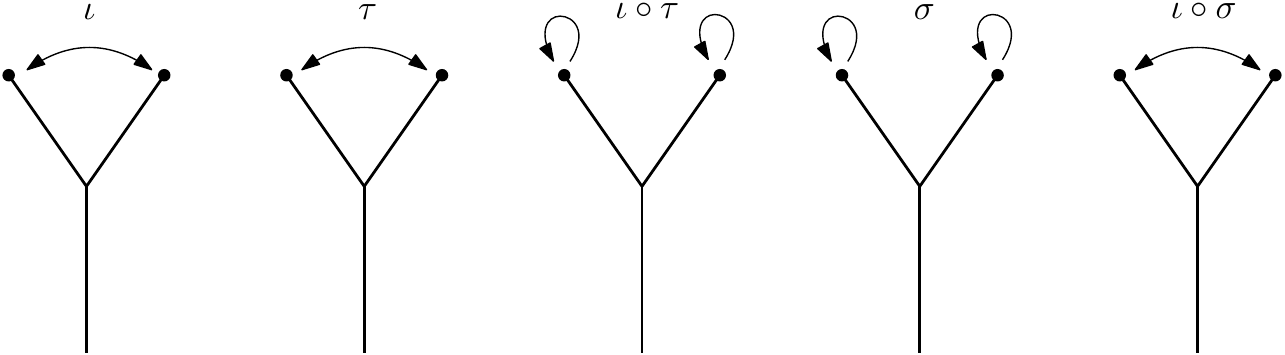}
\caption{Floer homology (and involution actions) for $Y_1 = S_{+1}(4_1)$.}\label{fig:3.3}
\end{figure}

\end{example}

\begin{example}\label{ex:3.5}
Let $Y_2$ be given by $(+1)$-surgery on the the stevedore knot $6_1$. In Figure~\ref{fig:3.4}, we have displayed two involutions (denoted by $\tau$ and $\sigma$) on $Y_2$, given by rotating the surgery diagram $180^{\circ}$ about the appropriate axes. The Floer homology $\HFm(Y_2)$ is displayed in Figure~\ref{fig:3.4}, together with the homology actions of $\tau$, $\ita$, $\sigma$, and $\iota \circ \sigma$. Note that (as in the previous example) $\tau = \iota \circ \sigma$ and $\sigma = \ita$. Each of these actions on $\HFm(Y_2)$ determines a chain-level action on the standard complex representative of $\CFm(Y_2)$. We have displayed these on the lower-left in Figure~\ref{fig:3.5}. The reader should check that each of these are chain maps and that they induce the claimed actions. 
\begin{figure}[h!]
\center
\includegraphics[scale=0.65]{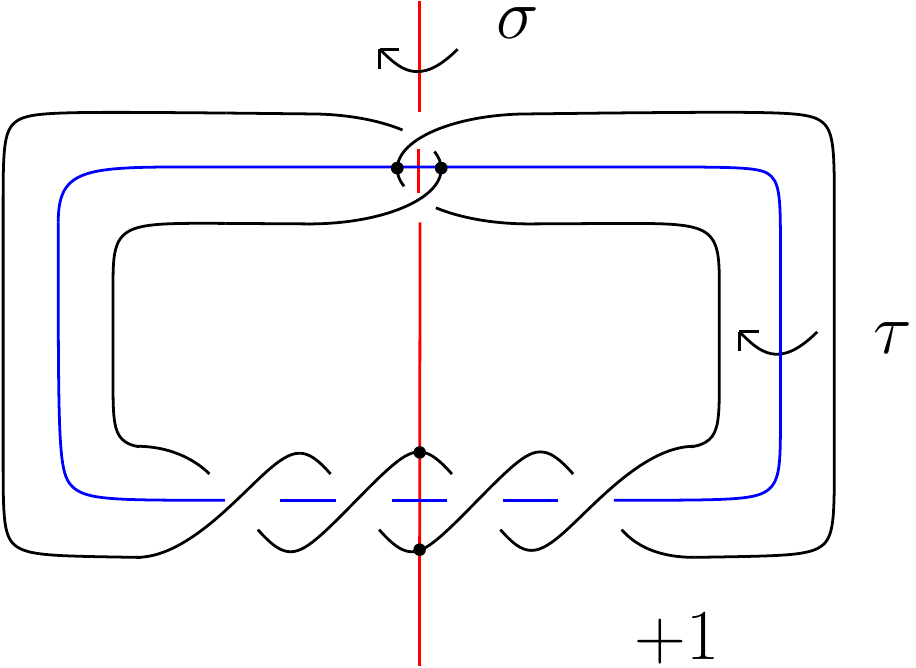}
\caption{$(+1)$-surgery on the stevedore knot, with involutions $\tau$ and $\sigma$.}\label{fig:3.4}
\end{figure}

\begin{figure}[h!]
\center
\includegraphics[scale=0.9]{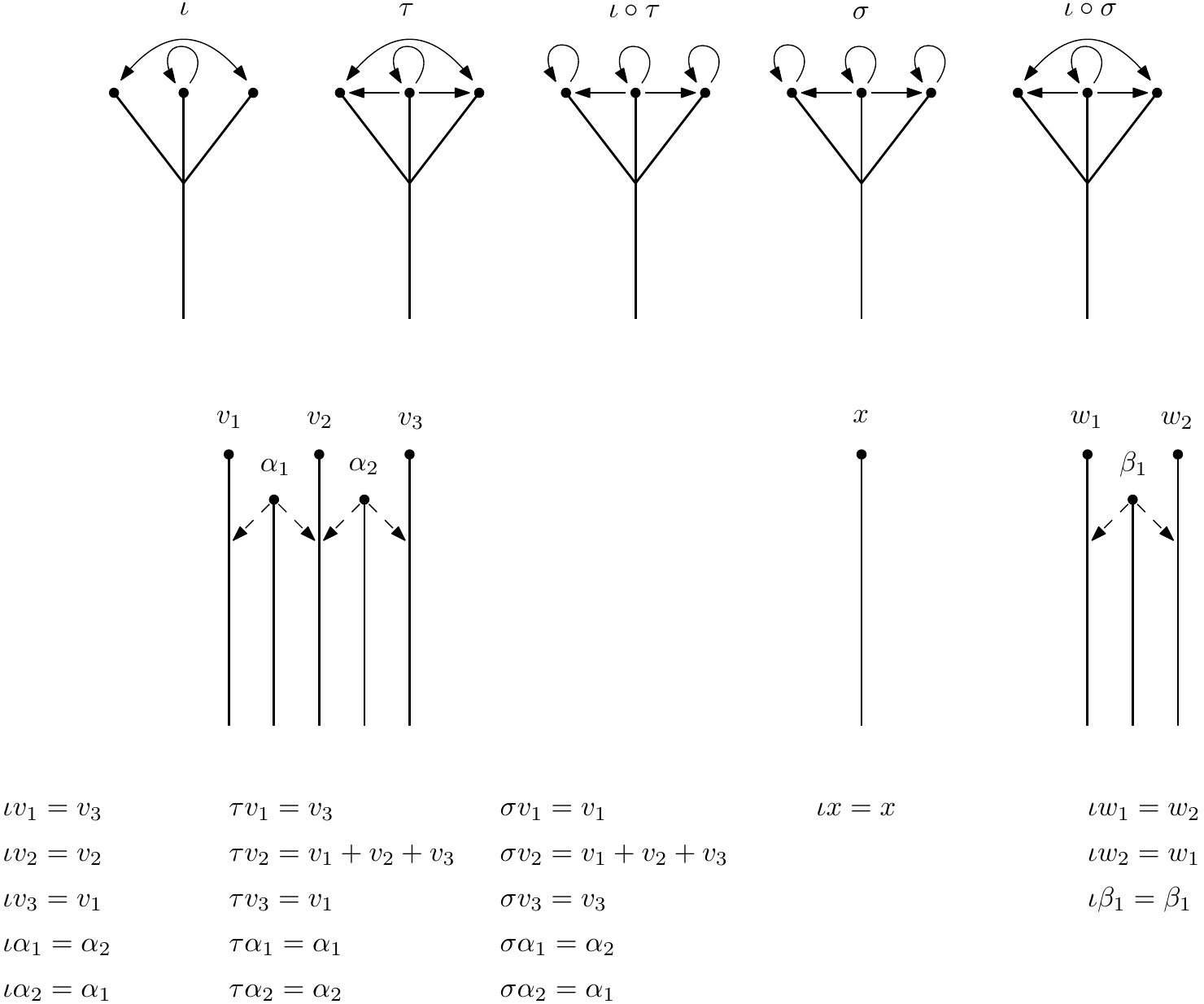}
\caption{Top: Floer homology (and involution actions) for $Y_2 = S_{+1}(6_1)$. Middle: standard complexes representatives of $\CFm(Y_2)$, $\CFm(Y_1)$ and $\CFm(S^3)$. Bottom: chain complex actions of various involutions.} \label{fig:3.5}
\end{figure}

We now claim that $h(Y_2)$ and $h_{\sigma}(Y_2)$ are both trivial. For $h(Y_2)$, this follows from the monotone root algorithm established in \cite[Section 6]{DM}, but it is easy to write down the local equivalence explicitly. Indeed, in the notation of Figure~\ref{fig:3.5}, setting
\begin{align*}
&f(v_1) = f(v_2) = f(v_3) = x & &g(x) = v_2 \\
&f(\alpha_1) = f(\alpha_2) = 0
\end{align*}
gives the local equivalence. For $h_{\sigma}(Y_2)$, we set
\begin{align*}
&f(v_1) = f(v_2) = f(v_3) = x & &g(x) = v_1 \\
&f(\alpha_1) = f(\alpha_2) = 0.
\end{align*}
The reader should carefully verify that $f$ and $g$ intertwine $\sigma$ on $\CFm(Y_2)$ with the (trivial) action of $\iota$ on $\CFm(S^3)$. Note that $g$ does \textit{not} intertwine the actions of $\iota$, even up to homotopy. It follows that the involutive invariants for $\sigma$ (and hence $\ita$) are trivial.

We furthermore claim that $h_\tau(Y_2)$ is locally equivalent to $h_\tau(Y_1) = h(Y_1)$. To effect the local equivalence, set
\begin{align*}
&f(v_1) = f(v_2) = w_1 \text{ and } f(v_3) = w_2 & &g(w_1) = v_1 \text{ and } g(w_2) = v_3 \\
&f(\alpha_1) = 0 \text{ and } f(\alpha_2) = \beta_1 & &g(\beta_1) = \alpha_1 + \alpha_2.
\end{align*}
The reader should again carefully verify that $f$ and $g$ intertwine $\tau$ on $\CFm(Y_2)$ with $\tau = \iota$ on $\CFm(Y_1)$. Note that $f$ does \textit{not} intertwine the actions of $\iota$, even up to homotopy. It follows that the involutive invariants for $\tau$ (and hence $\iota \circ \sigma$) are the same as those for $h_\tau(Y_1) = h(Y_1)$.

We now take a step back to re-iterate the salient points of Example~\ref{ex:3.5}. Firstly, note that in contrast to Example~\ref{ex:3.4}, the (usual) local equivalence class $h(Y_2)$ is zero. Hence $Y_2$ is not obstructed from bounding a homology ball; indeed, $Y_2$ bounds the Mazur manifold $W(0, 1)^-$ \cite{AKir}. Secondly, the nontriviality of $h_{\tau}(Y_2)$ obstructs $\tau$ from extending over any homology ball. Here, we have described this nontriviality by algebraically establishing a local equivalence between the claimed complexes of $h_\tau(Y_2)$ and $h_\tau(Y_1)$. The main idea of this paper will be to turn this procedure around and show that \textit{a priori} such a relationship must exist due to the existence a certain cobordism from $Y_2$ to $Y_1$. (In fact, this is precisely how we will show that the actions of $\tau$ and $\sigma$ on $\CFm(Y_2)$ are the ones claimed in Figure~\ref{fig:3.5}. See Lemma~\ref{lem:computation}.) Note, however, that $h(Y_1) \neq h(Y_2)$, so there is certainly no homology cobordism between the two.

Finally, observe that the nontriviality of $h_{\iota \circ \sigma}(Y_2)$ obstructs $\sigma$ from extending over any homology ball. This holds even though the local equivalence classes of both the $\iota$-complex and $\sigma$-complex of $Y_2$ are trivial. The point is that the local map $g$ in the $\sigma$-case is not $\iota$-equivariant (and similarly, the local map $g$ in the $\iota$-case is not $\sigma$-equivariant). Hence even though $h(Y_2)$ and $h_{\sigma}(Y_2)$ are both locally trivial, these local equivalences are not induced by the same map(s). Thus, as happens in this example, it is still possible for the class $h_{\iota \circ \sigma}(Y_2)$ of their composition to be nontrivial in local equivalence. 
\end{example}

\begin{figure}[h!]
\center
\includegraphics[scale=0.6]{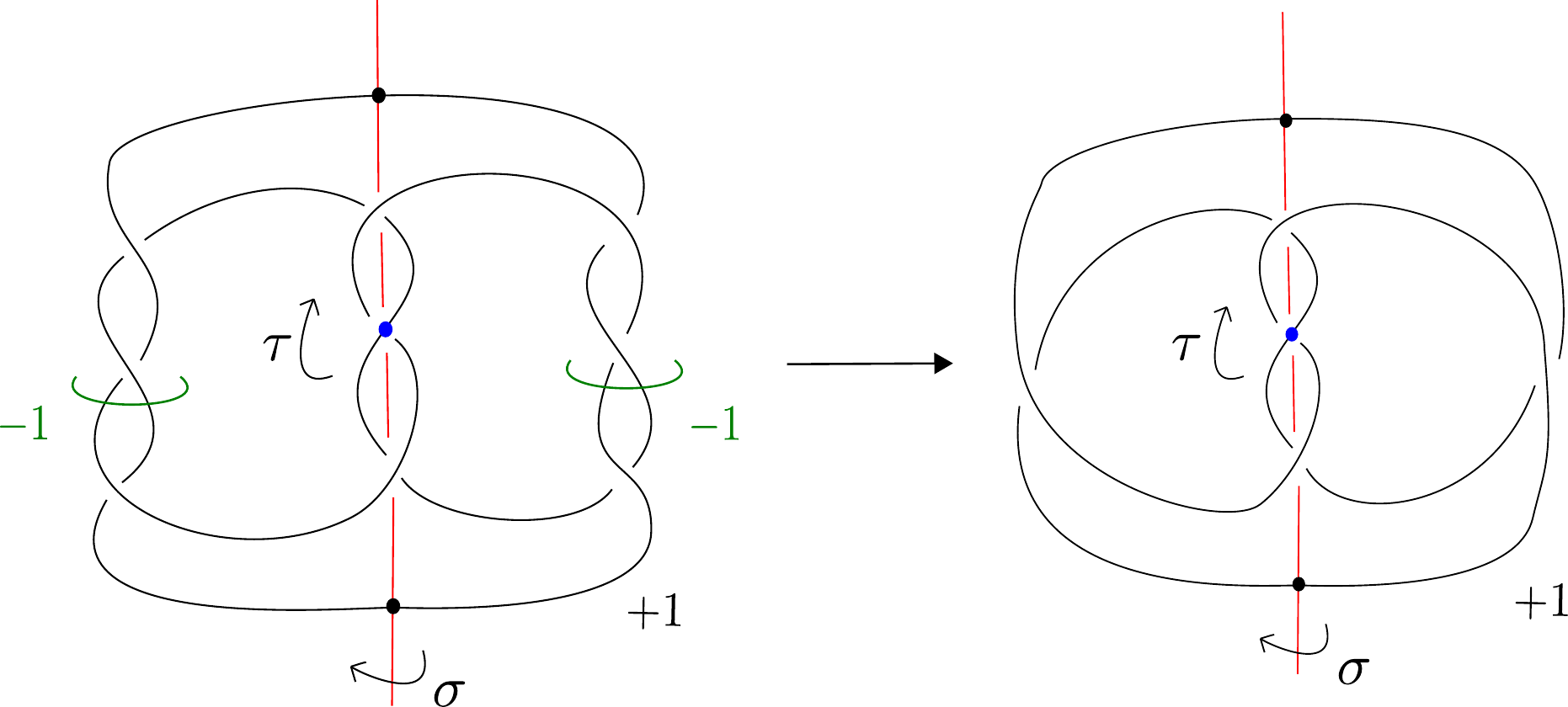}
\caption{A cobordism between two manifolds-with-involutions(s). Note that on the left we have $(+1)$-surgery on $P(-3, 3, -3)$ (which bounds a contractible manifold), while on the right we have $\Sigma(2, 3, 7)$. The cobordism in question is given by attaching $(-1)$-handles along the indicated green curves.}\label{fig:3.6}
\end{figure}

\begin{example}[Akbulut cork]\label{ex:akbulut}
As we will see, the methods of this paper can also easily be applied to the original Akbulut cork $M_1$. For our purposes, the simplest depiction of $M_1$ will be as $(+1)$-surgery on the pretzel knot $P(-3, 3, -3)$. In Figure~\ref{fig:3.6}, we have displayed an equivariant cobordism from $M_1$ (equipped with the indicated symmetries $\tau$ and $\sigma$) to $\Sigma(2, 3, 7)$. In Section~\ref{sec:5.1}, we leverage this cobordism to compute $h_\tau(M_1)$ and $h_{\ita}(M_1)$. It turns out that $\HFm(M_1)$ is isomorphic to $\HFm(Y_2)$, and that $\tau$ and $\sigma$ act on $\HFm(M_1)$ exactly as they do in Figure~\ref{fig:3.5}. Hence the same analysis as in Example~\ref{ex:3.5} shows that $(M_1, \tau)$ and $(M_1, \sigma)$ are strong corks. Note that the first computation of the action of $\tau$ on $\HFm(M_1)$ was carried out by Lin, Ruberman, and Saveliev \cite[Section 8.1]{LRS}, following work of Akbulut and Durusoy \cite{AD}. Here, our computation will be a straightforward consequence of Figure~\ref{fig:3.6}. See Lemma~\ref{lem:akbulut}.
\end{example}

\subsection{Cobordism maps}
We now investigate the behavior of $h_{\tau}$ and $h_{\ita}$ under negative-definite cobordisms. We begin with the simplest case, in which $Y_1$ and $Y_2$ are two homology spheres. Let $W$ be a cobordism from $Y_1$ to $Y_2$ and let $\s$ be a $\spinc$-structure on $W$. Recall that the associated Heegaard Floer grading shift is given by
\[
\Delta(W, \s) = \dfrac{c_1(\s)^2 - 2\chi(W) - 3\sigma(W)}{4}.
\]
In what follows, we will be concerned with negative-definite cobordisms admitting $\s$ for which $\Delta(W, \s) = 0$. 

\begin{remark}\label{rem:3.6}
Suppose that $W$ is definite. By a well-known result of Elkies \cite{Elkies}, $\Delta(W, \s) = 0$ if and only if the intersection form of $W$ is diagonalizable (over $\Z$) and $c_1(\s)$ has all coefficients equal to $\pm 1$ in the diagonal basis.
\end{remark}

\begin{proposition}\label{lem:3.7}
Let $Y_1$ and $Y_2$ be two homology spheres equipped with involutions $\tau_1$ and $\tau_2$, respectively. Let $(W, f, \s)$ be a negative-definite, $\spinc$-cobordism from $(Y_1, \tau_1)$ to $(Y_2, \tau_2)$ with $b_1(W) = 0$ and $\Delta(W, \s) = 0$. Then:
\begin{enumerate}
\item If $f_* \s = \s$, then $h_{\tau_1}(Y_1) \leq h_{\tau_2}(Y_2)$.
\item If $f_* \s = \bs$, then $h_{\iota \circ \tau_1}(Y_1) \leq h_{\iota \circ \tau_2}(Y_2)$.
\end{enumerate}
\end{proposition}
\begin{proof}
The proposition is a straightforward consequence of the functorial properties of Heegaard Floer homology under cobordisms. By the proof of \cite[Theorem 9.1]{OSabsgr}, the cobordism map
\[
F_{W, \s} : \CFm(Y_1) \rightarrow \CFm(Y_2)
\]
sends $U$-nontorsion elements to $U$-nontorsion elements in homology. By \cite[Proposition 4.9]{HM}, we have
\[
F_{W, \bs} \circ \iota_1 \simeq \iota_2 \circ F_{W, \s}.
\]
The analogous commutation relation for $\tau$ is given by
\[
F_{W, f_*\s} \circ \tau_1 \simeq \tau_2 \circ F_{W, \s}.
\]
We defer the proof of this statement to Lemma~\ref{lem:secZ.1} below, although it is certainly well-known to experts in different forms (see for example \cite[Theorem 3.1]{OSsmooth4}, \cite[Theorem A]{Zemkegraph}). Note that implicitly, $F_{W, \s}$ depends on a choice of path $\gamma$ from $Y_1$ to $Y_2$. The two cobordism maps above should thus be taken with respect to different paths, with the map on the left being taken with respect to $f(\gamma)$. However, since $b_1(W) = 0$, one can show that $F_{W, \s}$ is independent of the choice of path (up to $U$-equivariant homotopy). For this and a discussion of several other subtleties (including the fact that $\tau_i$ need not fix the basepoint of $Y_i$), see Section~\ref{sec:secZ.1}.

If $f_*\s = \s$, then the commutation relation for $\tau$ immediately exhibits $F_{W, \s}$ as the desired local map for the first claim. If $f_*\s = \bs$, we instead observe that
\[
F_{W, \s} \circ (\iota_1 \circ \tau_1) \simeq \iota_2 \circ F_{W, \bs} \circ \tau_1 \simeq (\iota_2 \circ \tau_2) \circ F_{W, f_*\bs}.
\]
Noting that $f_*$ commutes with conjugation, we thus see that $F_{W, \s}$  effects the desired local map for the second claim.
\end{proof}
\noindent
This immediately yields a proof of Theorem~\ref{thm:1.1}:

\begin{proof}[Proof of Theorem~\ref{thm:1.1}]
We may isotope the extension of $\tau$ in the interior of $W$ to fix a ball (pointwise). Cutting out this ball gives a homology cobordism from $(Y, \tau)$ to $(S^3, \id)$. This constitutes a negative-definite cobordism in both directions, on which there is only one $\spinc$-structure. The claim then follows from Proposition~\ref{lem:3.7}.
\end{proof}

For ease of terminology, we define:
\begin{definition}\label{def:3.8}
Let $(W, f, \s)$ be a negative-definite, $\spinc$-cobordism between two homology spheres with $b_1(W) = 0$ and $\Delta(W, \s) = 0$. If $f_*\s = \s$, then we say that $(W, f, \s)$ is \textit{$\spinc$-fixing} with respect to $\s$. If $f_* \s = \bs$, then we say that $(W, f, \s)$ is \textit{$\spinc$-conjugating} with respect to $\s$. When there is no possible confusion, we will suppress writing $\s$. Note that if $\s = \bs = f_*\s$, then $(W, f, \s)$ is both $\spinc$-fixing and $\spinc$-conjugating.
\end{definition}


\section{Equivariant Cobordisms}\label{sec:4}
In the preceding section, we showed that $h_\tau$ and $h_{\ita}$ are monotonic (with respect to the partial order on $\Inv$) under an appropriate class of equivariant negative-definite cobordisms. (See Lemma \ref{lem:3.7}.) In this section, we establish some simple topological methods for explicitly constructing equivariant cobordisms to which we can apply the aforementioned monotonicity. These cobordisms will serve as the key topological tool for our results.

\subsection{Equivariant surgery and handle attachments}\label{sec:4.1}


Let $K$ be a knot in a 3-manifold $Y$, and let $\tau$ be an orientation-preserving involution on $Y$ that fixes $K$ setwise. In the case that $Y = S^3$, we will often draw $\tau$ as $180^{\circ}$ rotation though some axis of symmetry. (By work of Waldhausen, any orientation-preserving involution of $S^3$ is conjugate to one of this form \cite{Waldhausen}.) Usually, we draw this axis as a line in $\R^3$, but sometimes it will be more convenient to draw the axis of rotation as an unknot, as in Figure~\ref{fig:1.1}. 

In this subsection, we verify that $\tau$ induces an involution on any manifold obtained by surgery on $K$ and, similarly, on any cobordism formed from handle attachment along $K$. This is well-known and implicit in many sources, e.g.\ \cite{Montesinos}, but we include the proofs here for completeness.

\begin{definition}\label{def:4.1}
An involution $\tau$ of $(Y,K)$ is said to be a \textit{strong involution} (or \textit{strong inversion}) if $\tau$ fixes two points on $K$. If instead the action is free on $K$, we say that $\tau$ is a \textit{periodic involution}. Note that a strong involution reverses orientation on $K$, while a periodic involution preserves orientation. We will sometimes refer to such a $K$ as an \textit{equivariant knot}.
\end{definition}

Now let $K$ be an equivariant knot in $Y$. It is easily checked that there exists an equivariant framing of $K$, as follows.\footnote{In fact, if $K$ is a knot in $S^3$, then the reader can check that the Seifert framing can be made equivariant.} By averaging an arbitrary Riemannian metric with its pullback under $\tau$, we may assume that $\tau$ acts as an isometry on $Y$, and hence also on the normal bundle to $K$. If we fix an arbitrary framing of $K$, we can choose coordinates 
\[
\nu(K) \cong S^1\times D^2=\{(z, w) : |z| = 1, |w| \leq 1\},
\]
such that:
\begin{enumerate}
\item If $\tau$ is strong, then the action of $\tau$ on $\nu(K)$ is $\tau(z, w) = (\bar{z}, A_z\bar{w})$.
\item If $\tau$ is periodic, then the action of $\tau$ on $\nu(K)$ is $\tau(z, w) = (-z, A_zw)$.
\end{enumerate}
In both cases, $A_z$ denotes a continuous family of matrices parametrized by $S^1$. In the strong case, we have $A_z\in O(2)$ and $\text{det }A_z=-1$, while in the periodic case, we have $A_z\in SO(2)$. If $\tau$ is strong, then $\tau$ fixes the two discs $\{1\} \times D^2$ and $\{-1\} \times D^2$ setwise, and has two fixed points on the boundary of each. Take any arc $\gamma$ on $\partial \nu(K)$ running from a fixed point of $\tau$ on $\{1\} \times S^1$ to a fixed point of $\tau$ on $\{-1\} \times S^1$. (We may also assume that $\gamma$ projects as a diffeomorphism onto a subarc of $K$.) Then $\gamma \cup \tau \gamma$ constitutes an equivariant framing of $K$ (and in fact any framing can be realized). If $\tau$ is periodic, then we instead take $\gamma$ to be a similar arc joining an arbitrary point $p$ in $\{1\} \times S^1$ to its image $\tau p$ in $\{-1\} \times S^1$. Clearly, $\gamma \cup \tau \gamma$ is again an equivariant framing of $K$.

It follows that we can re-parameterize our neighborhood of $K$ so that the equivariant framing constructed above is given by $S^1 \times \{1\}$. Then:
\begin{enumerate}
\item If $\tau$ is strong, then the action of $\tau$ on $\nu(K)$ is $\tau(z, w) = (\bar{z}, \bar{w})$.
\item If $\tau$ is periodic, then the action of $\tau$ on $\nu(K)$ is $\tau(z, w) = (-z, w)$.
\end{enumerate}
\noindent

\begin{lemma}\label{lem:4.3}
Let $K$ be an equivariant knot in $Y$ with symmetry $\tau$. Fix any framing $K'$ of $K$, and let $Y_{p/q}(K)$ be $(p/q)$-surgery on $K$ with respect to this framing. Then $\tau$ extends to an involution on $Y_{p/q}(K)$. This extension is unique up to isotopy.
\end{lemma}

\begin{proof}
It suffices to prove the claim under the additional assumption that $K'$ is equivariant. Indeed, since the claim of the lemma holds for all surgeries, proving the desired statement for a single framing establishes it for all framings.

On the complement of $\nu(K)$, we define our involution to be equal to $\tau$. Parameterize the boundary of $\nu(K)$ by $z$ and $w$, as above. The surgered manifold $Y_{p/q}(K)$ is obtained from the complement of $K$ by gluing in the solid torus
\[
S^1 \times D^2 = \{(z', w') : |z'| = 1, |w'| \leq 1\}
\]
via the boundary diffeomorphism
\[
f(z', w') = (z = (z')^s (w')^q, w = (z')^r (w')^p)
\]
where $r$ and $s$ are integers such that $ps - qr = 1$. If $\tau$ is strong, then we have the obvious extension by complex conjugation
\[
\tau(z', w') = (\bar{z}', \bar{w}').
\]
If $\tau$ is periodic, then we have the extension:
\[
\tau(z', w') = 
\begin{cases}
(-z', -w') &\text{if } (p, q, r, s) = (1, 0, 1, 1) \text{ or } (1, 1, 1, 0) \mod 2\\
(-z', w') &\text{if } (p, q, r, s) = (1, 0, 0, 1) \text{ or } (1, 1, 0, 1) \mod 2\\
(z', -w') &\text{if } (p, q, r, s) = (0, 1, 1, 0) \text{ or } (0, 1, 1, 1) \mod 2.
\end{cases}
\]
Note that the diffeomorphism between the gluings corresponding to $(p, q, r, s)$ and $(p, q, r+p, s+q)$ sends $(z', w')$ to $(z', z'w')$. This intertwines $\tau$, so (up to re-parameterization) our extension of $\tau$ does not depend on $(r, s)$.

It is easy to check that any two extensions of $\tau$ must be isotopic to each other. For example, in the case of a strong involution, $\tau$ fixes a meridional curve on the torus boundary setwise. Hence any extension of $\tau$ maps the disk $D$ bounded by this curve to some other disk $D'$ with $\partial D = \partial D'$. It is then clear that we can isotope $\tau$ (rel boundary) so that it fixes $D$. Cutting out $D$, we then use the fact that every diffeomorphism of $S^2$ extends uniquely over $B^3$ (up to isotopy).
\end{proof}
\noindent
Given an equivariant knot, we will thus freely view its symmetry as defining a symmetry on any surgered manifold.

\begin{lemma}\label{lem:4.4}
Let $K$ be an equivariant knot in $Y$ with symmetry $\tau$. Fix any framing $K'$ of $K$. Then $\tau$ extends over the 2-handle cobordism given by attaching a 2-handle along $K$ with framing $n$ (relative to $K'$). The involution on the boundary is the extension of $\tau$ to $Y_n(K)$ afforded by Lemma~\ref{lem:4.3}.
\end{lemma}
\begin{proof}
It suffices to prove the claim under the additional assumption that $K'$ is equivariant. Indeed, since the claim of the lemma holds for all $n$, proving the desired statement for a single framing establishes it for all framings.

Parameterize the 2-handle by
\[
D^2 \times D^2 = \{(z', w') : |z'| \leq 1, |w'| \leq 1\}.
\]
The boundary subset $S^1 \times D^2 \subset D^2 \times D^2$ is identified with $\nu(K)$ via the map sending
\[
(z', w') \mapsto (z = z', w = (z')^n w').
\]
If $\tau$ is strong, then the extension is given by 
\[
\tau(z', w') = (\bar{z}', \bar{w}').
\]
If $\tau$ is periodic, then the extension is given by
\[
\tau(z', w') = 
\begin{cases}
(-z', -w') &\text{if } n \text{ is odd}\\
(-z', w') &\text{if } n \text{ is even},
\end{cases}
\]
as desired.
\end{proof}

\noindent
This shows that if $K$ is an equivariant knot, then equivariant handle attachment along $K$ is well-defined, and that $\tau$ moreover extends over the handle attachment cobordism.

We will also consider surgeries on links in which $\tau$ exchanges some pairs of link components (with the same framing), in addition to possibly fixing some components. Given the above treatment of the fixed link components, it is  clear that such  $\tau$ extend to involutions on the surgered manifolds and over the handle attachment cobordisms (whenever the surgery coefficients are integral). 

\subsection{Actions on $\spinc$-structures}\label{sec:4.2}
We now specialize to the case where $Y$ is a homology sphere. Let $K$ be an equivariant knot in $Y$ with symmetry $\tau$. Let $W$ be the cobordism formed by $(-1)$-handle attachment along $K$, relative to the Seifert framing. This is a negative-definite cobordism whose second cohomology $H^2(W)$ is generated by a single element $x$. Note that the $\spinc$-structures on $W$ with $c_1(\s) = \pm x$ have $\Delta(W, \s) = 0$. We claim that if the involution $\tau$ is periodic, then $W$ is $\spinc$-fixing, while if $\tau$ is strong, then $W$ is $\spinc$-conjugating. To see this, it suffices to understand the action of $\tau$ on $H^2(W)$. Under the isomorphism $H^2(W) \cong H_2(W, \partial W)$, the generator $x$ corresponds to the cocore of the attaching 2-handle. In the notation of Lemma~\ref{lem:4.4}, this is given by
\[
\{0\} \times D^2 = \{(z', w') : z' = 0, |w'| \leq 1 \}.
\]
An examination of the extension of $\tau$ over $W$ shows that $\tau$ reverses orientation on the cocore if $\tau$ is strong and preserves orientation if it is periodic. Hence if $\tau$ is strong, it acts via multiplication by $-1$ on $H^2(W)$, and otherwise fixes $H^2(W)$. We thus define:

\begin{definition}
Let $Y_1$ be a homology sphere with involution $\tau$. Let $K$ be an equivariant knot in $Y_1$. Suppose that $Y_2$ is obtained from $Y_1$ by doing $(-1)$-surgery on $K$, relative to the Seifert framing. Then the corresponding handle attachment cobordism constitutes an equivariant cobordism from $Y_1$ to $Y_2$, where the latter is equipped with the usual extension of $\tau$. This is $\spinc$-fixing if $\tau$ is periodic and $\spinc$-conjugating if $\tau$ is strong. We refer to these as \textit{$\spinc$-fixing $(-1)$-cobordisms} and \textit{$\spinc$-conjugating $(-1)$-cobordisms}, respectively. 
\end{definition}

Similarly, we may consider attaching a pair of handles to $Y_1$ along a two-component link with algebraic linking number zero whose components are interchanged by $\tau$. In this situation, $H^2(W)$ is generated by two elements $x$ and $y$, where $\tau x = y$ and $\tau y = x$ (with appropriately chosen orientations). Choosing the $\spinc$-structure $\s$ with $c_1(\s) = x + y$ then yields a $\spinc$-fixing cobordism, while choosing the $\spinc$-structure with $c_1(\s) = x - y$ yields a $\spinc$-conjugating cobordism. 

\begin{definition}
Let $Y_1$ be a homology sphere with involution $\tau$. Let $L$ be a two-component link in $Y_1$ with algebraic linking number zero whose components are interchanged by $\tau$. Let $Y_2$ be obtained from $Y_1$ by doing an additional $(-1)$-surgery on each component of $L$, relative to the Seifert framing. Then the corresponding handle attachment cobordism constitutes an equivariant cobordism from $Y_1$ to $Y_2$, where the latter is equipped with the usual extension of $\tau$. This is both $\spinc$-fixing and $\spinc$-conjugating (with respect to different $\spinc$-structures). We refer to such a cobordism as an \textit{interchanging $(-1, -1)$-cobordism}.
\end{definition}

In light of Proposition~\ref{lem:3.7}, we thus immediately obtain a proof of Theorem~\ref{thm:1.4}:

\begin{proof}[Proof of Theorem~\ref{thm:1.4}]
Using the conditions on the linking numbers, it is easily checked that the cobordisms in question satisfy the hypotheses of Proposition~\ref{lem:3.7}. Note that in the third claim of Theorem~\ref{thm:1.4}, we use two separate $\spinc$-structures to establish the two different inequalities.
\end{proof}

\noindent
Of course, we have the analogous notion of $(+1)$- and $(+1, +1)$-cobordisms. We obtain a similar set of inequalities (going in the opposite direction) by turning these cobordisms around.

\subsection{Further operations}\label{sec:4.3}
We will occasionally need to compare symmetries in two different surgery descriptions of the same 3-manifold. Although we will not belabor the point, the reader should check that the blow-up and blow-down operations displayed in Figure~\ref{fig:4.2} can be performed equivariantly. Note that if $u$ is an equivariant $(1/k)$-framed unknot which is split off from the rest of a surgery diagram, then $u$ can be deleted. Indeed, let $u$ be contained in a ball $B^3$. Then $(1/k)$-surgery on $u$ is again a ball, equipped with a slightly different extension of the 180-degree-rotation on $S^2 = \partial B^3$. However, every diffeomorphism of $\partial B^3$ extends uniquely over $B^3$ up to isotopy rel boundary. 

\begin{figure}[h!]
\center
\includegraphics[scale=0.83]{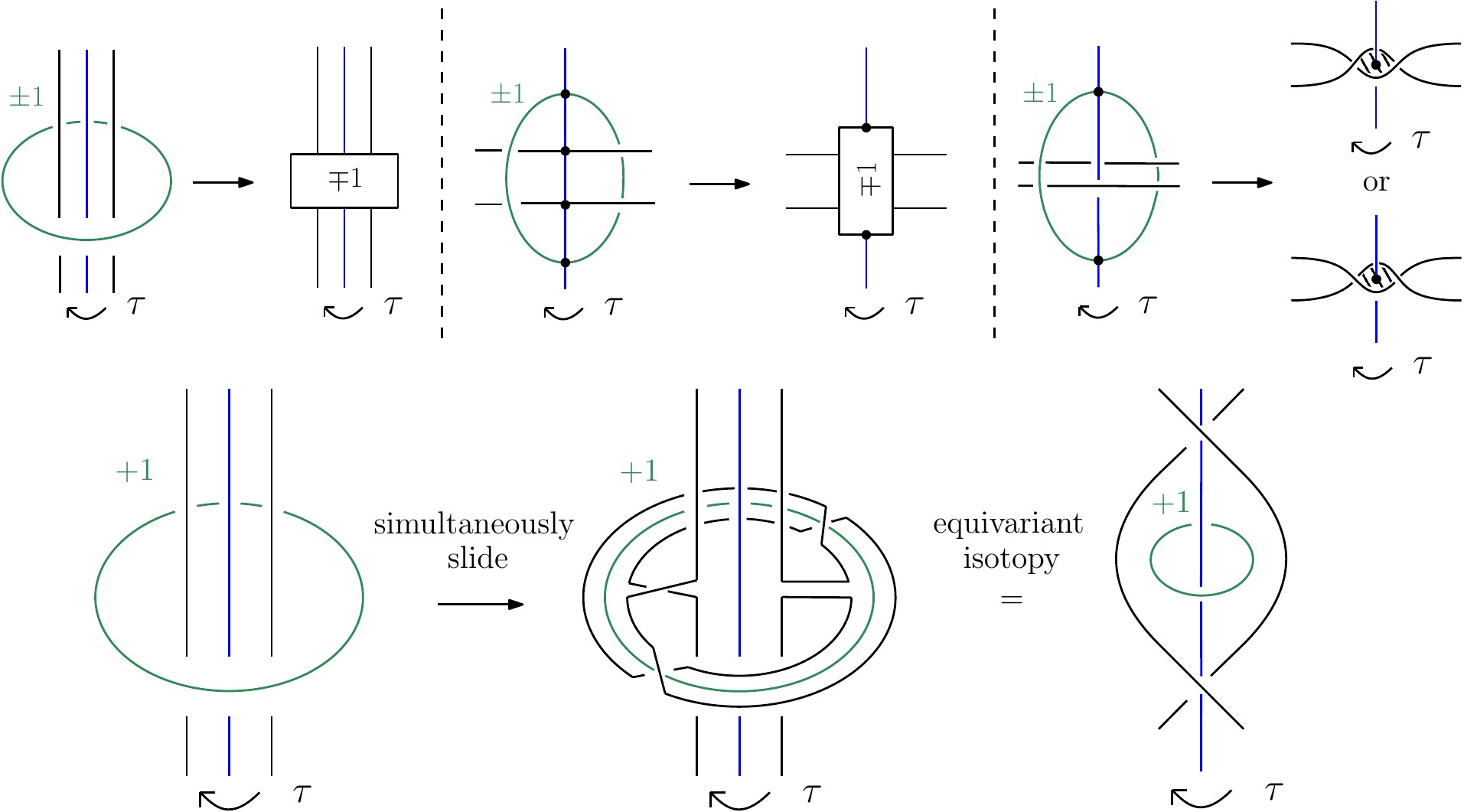}
\caption{Top: various equivariant blow-up/blow-down operations. Bottom: an equivariant (simultaneous) slide followed by an equivariant isotopy.}\label{fig:4.2}
\end{figure}

\begin{proof}[Proof of Theorem~\ref{thm:1.5}]
Let $K$ be a knot in $S^3$ with a strong involution $\tau$. Then $(1/k)$-surgery on $K$ is equivariantly diffeomorphic to the two-component link surgery consisting of $0$-surgery on $K$, together with $(-k)$-surgery on a meridian $\mu$ of $K$. Choosing $\mu$ to be an equivariant unknot near one of the fixed points on $K$ makes this diffeomorphism $\tau$-equivariant. Let $u$ and $\tau u$ be an additional pair of $(-1)$-framed unknots which each link $\mu$, as in Figure~\ref{fig:4.3}. Blowing down, the resulting manifold is equivariantly diffeomorphic to surgery on $K$ with coefficient $1/(k-2)$. We claim that handle attachment along $u$ and $\tau u$ constitutes an interchanging $(-1, -1)$-cobordism from $S_{1/k}(K)$ to $S_{1/(k-2)}(K)$. To see this, we equivariantly slide $u$ and $\tau u$ over $K$, which algebraically unlinks them from the rest of the diagram (see Figure~\ref{fig:4.3}). The claim then follows from Theorem~\ref{thm:1.4}.
\end{proof}

\begin{figure}[h!]
\center
\includegraphics[scale=0.95]{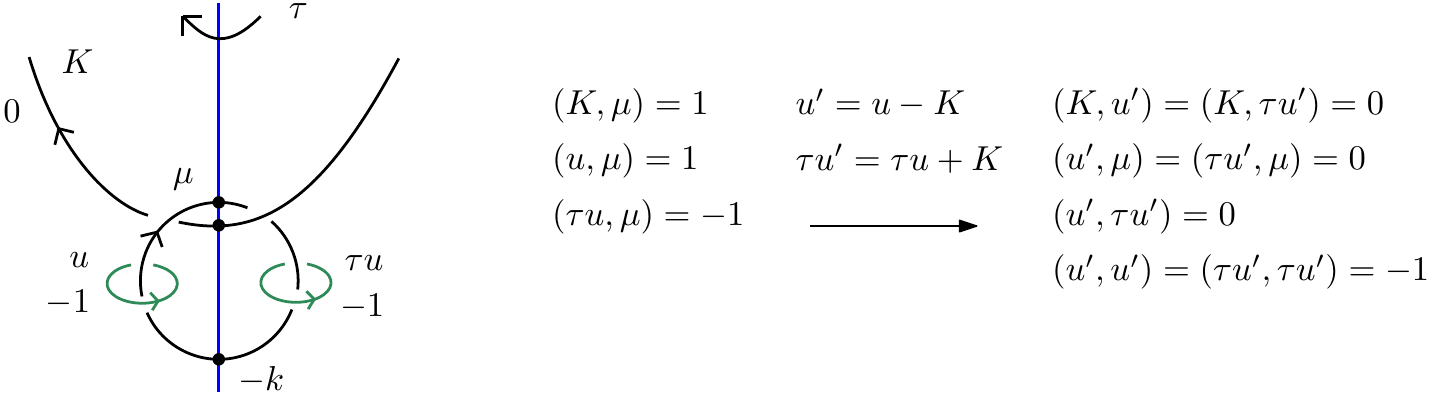}
\caption{Left: the equivariant cobordism used in the proof of Theorem~\ref{thm:1.5}. Right: handleslides establishing that this is an interchanging $(-1, -1)$-cobordism. Since $\tau$ reverses orientation on $K$, the indicated handleslides are $\tau$-equivariant.}\label{fig:4.3}
\end{figure}


\pagebreak
\section{Bordisms and Local Equivalence}\label{sec:secZ}

\subsection{Graph cobordisms and equivariance}\label{sec:secZ.1}
We now turn to a proof of Theorem~\ref{thm:1.2}. For this, we will need to consider the more general situation of a cobordism with disconnected ends. We thus briefly review the functoriality package for Heegaard Floer homology developed by Zemke in \cite{Zemkegraph}, building on previous work of Ozsv\'ath-Szab\'o \cite{OSsmooth4} and Juh\'asz \cite{Juhasz}, \cite{JuhaszTQFT}. In what follows, we allow each manifold $Y$ to have a collection of basepoints $\w$. Usually, one introduces different $U$-variables to keep track of the different basepoints, but here we will identify all of these into a single $U$-variable. In the terminology of \cite{Zemkegraph}, this is called the \textit{trivial coloring}.

Let $W$ be a cobordism between two (possibly disconnected) 3-manifolds $(Y_1, \w_1)$ and $(Y_2, \w_2)$. A \textit{ribbon graph} in $W$ is an embedded graph $\Gamma$ whose intersection with each $Y_i$ is precisely $\w_i$. We also require that $\Gamma$ be given a \textit{formal ribbon structure}, which is a choice of cyclic ordering at every internal vertex of $\Gamma$. We refer to the pair $(W, \Gamma)$ as a \textit{ribbon graph cobordism}. Associated to any such $(W, \Gamma)$, Zemke constructs two chain maps
\[
F^A_{W, \Gamma, \s}, F^B_{W, \Gamma, \s} : \CFm(Y_1, \w_1, \s|_{Y_1}) \rightarrow \CFm(Y_2, \w_2, \s|_{Y_2}).
\]
These are well-defined up to $U$-equivariant homotopy and are an invariant of the smooth isotopy class of $\Gamma$ in $W$ \cite[Definition 3.4]{Zemkegraph}. In fact, $F^A$ and $F^B$ are invariant under a weaker notion of equivalence called \textit{ribbon equivalence}; see \cite[Corollary D]{Zemkelinkcobord}. Although $F^A$ and $F^B$ satisfy certain symmetries with respect to each other, they are not always equal. However, they have the same formal properties, so for convenience we will focus on $F^A$. See \cite[Section 3.1]{Zemkegraph} for precise definitions.

Now let $Y_1$ and $Y_2$ be disjoint unions of homology spheres, and equip each connected component of $Y_1$ and $Y_2$ with a single basepoint. Let $f$ be a diffeomorphism of $W$ restricting to $\tau_i$ on each $Y_i$. If $\tau_i$ fixes the basepoints of $Y_i$, then it follows from \cite[Theorem A]{Zemkegraph} (together with the well-definedness of graph cobordism maps up to $U$-equivariant homotopy) that
\begin{equation}\label{eq:1}
\tau_2 \circ F^A_{W, \Gamma, \s} \simeq F^A_{W, f(\Gamma), f_*(\s)} \circ \tau_1.
\end{equation}
See \cite[Equation 1.2]{Zemkegraph}. If $\tau_i$ does not fix the basepoints of $Y_i$, then (\ref{eq:1}) is not quite correct, since in this case we have defined the action of $\tau_i$ on $\CFm$ using an isotoped version of $\tau_i$ instead. Clearly, however, we can isotope $f$ so that it restricts to the isotoped versions of $\tau_i$ at either end. Thus, (\ref{eq:1}) holds after replacing $f(\Gamma)$ with a slightly altered graph $f(\Gamma)$ which has the same endpoints as $\Gamma$. (Usually, we will be sloppy and continue to write $f(\Gamma)$ despite this difference.) 

In order to define the $F^A$- and $F^B$-maps, Zemke first defines graph cobordism maps in the case of a product cobordism $Y \times I$. In this situation, we can use the projection map to view $\Gamma$ as being embedded in $Y$ (after perturbing slightly, if necessary). In \cite{Zemkegraph}, Zemke introduces a set of auxiliary maps on $\CFm(Y)$ which can be used to associate to any such graph an endomorphism $\fA_{\cG_\Gamma}$ of $\CFm(Y)$. These auxiliary maps include the \textit{free stabilization maps} $S^{\pm}_w$, as well as the \textit{relative homology maps} $A_\lambda$. We will assume some familiarity with these constructions; the reader is referred to \cite[Section 3]{HMZ} for a concise and helpful summary. 

In order to understand $F^A$ for a general cobordism $W$, it is helpful to keep in mind the desired composition law. Let $(W, \Gamma) = (W_2, \Gamma_2) \cup (W_1, \Gamma_1)$. If $\s_1$ and $\s_2$ are $\spinc$-structures on $W_1$ and $W_2$, then the obvious generalization of the usual composition law of Ozsv\'ath and Szab\'o yields:
\begin{equation}\label{eq:2}
F^A_{W_2, \Gamma_2, \s_2} \circ F^A_{W_1, \Gamma_1, \s_1} \simeq \sum_{\substack{\s \ \in \ \spinc(W) \\ \s|_{W_2} = \ \s_2 \\ \s|_{W_1} = \ \s_1} }F^A_{W, \Gamma, \s}.
\end{equation}
To this end, consider a parameterized Kirby decomposition for $W$, and split
\[
W = W_2 \circ W_1,
\]
where $W_1$ is the subcobordism consisting of all 0- and 1-handles. We denote the outgoing boundary of $W_1$ by $Y$. Note that for such a splitting, a $\spinc$-structure $\s$ on $W$ is uniquely determined by its restrictions $\s_i$ to each $W_i$.

The underlying Morse function on $W$ provides a gradient-like vector field $\vec{v}$ on $W$. After a small perturbation, we can assume that $\Gamma$ is disjoint from the descending manifolds of the index-one critical points, the ascending manifolds of the index-three critical points, and both the ascending and descending manifolds of the index-two critical points. Using $\vec{v}$, we flow each point of $\Gamma$ backwards or forwards so that it hits $Y$. This gives (possibly after a small perturbation) an embedded graph in $Y$, which we may think of as a ribbon graph in $Y \times (-\epsilon, \epsilon)$. We connect this to the basepoints of the $Y_i$ via arcs going along the flow lines of $\vec{v}$. Denote these collections of arcs by $\Gamma_1$ and $\Gamma_2$. The map $\smash{F^A_{W, \Gamma, \s}}$ is then equal to the composition
\begin{equation}\label{eq:3}
F^A_{W, \Gamma, \s} \simeq F^A_{W_2, \Gamma_2, \s_2} \circ \mathfrak{A}_{\mathcal{G}_\Gamma} \circ F^A_{W_1, \Gamma_1, \s_1}.
\end{equation}
Here, $\mathfrak{A}_{\mathcal{G}_\Gamma}: \CFm(Y) \rightarrow \CFm(Y)$ is the graph action map associated to the flowed image of $\Gamma$ in $Y$, and should be thought of as defining the cobordism map in the case where $W = Y \times I$. When no confusion is possible, we will sometimes suppress notation and write the outer two maps as $\smash{F^A_{W_1, \s_1}}$ and $\smash{F^A_{W_2, \s_2}}$. See Figure~\ref{fig:secZ.1}.

\begin{figure}[h!]
\center
\includegraphics[scale=0.8]{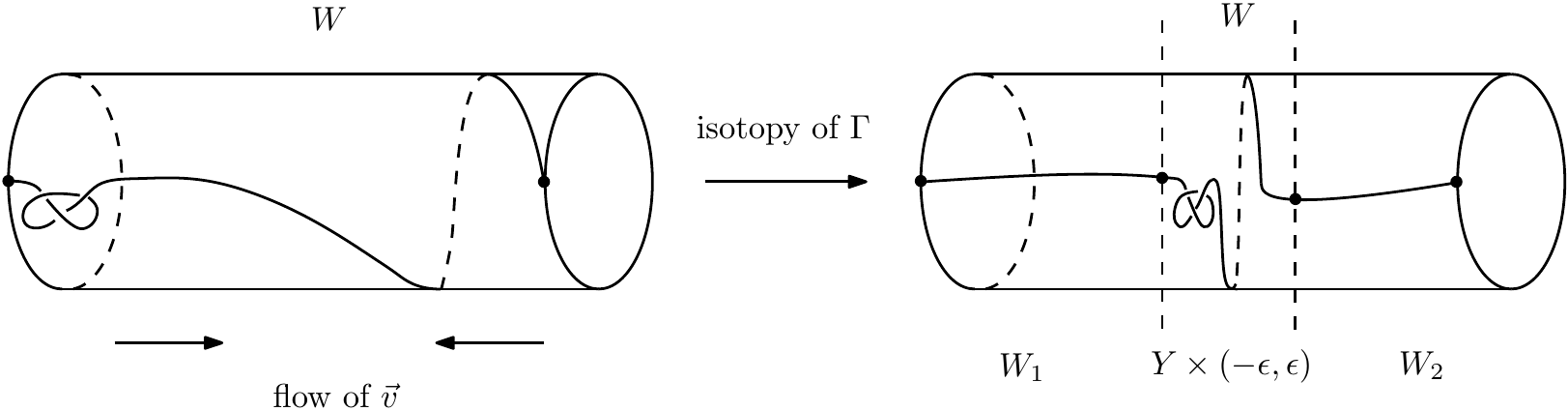}
\caption{Schematic depiction of flowing $\Gamma$ into $Y$. In actuality, $Y$ will have some topology and $\Gamma$ need not be a path.}\label{fig:secZ.1}
\end{figure}  

Roughly speaking, we think of the whole procedure as isotoping $\Gamma$ so that it is uninteresting outside of $Y$; the maps associated to $(W_1, \Gamma_1)$ and $(W_2, \Gamma_2)$ can then be defined using only a slight modification of the construction of Ozsv\'ath and Szab\'o. In what follows, we similarly use the technique of flowing $\Gamma$ so that it is ``concentrated" in a convenient slice. In particular, note that if $\Gamma$ and $\Gamma'$ are two ribbon graphs in $W$, then their flowed versions agree outside of $Y$.

For convenience, we also record the grading shift formula established in \cite[Proposition 4.1]{HMZ}. Let $(W, \Gamma)$ be a ribbon graph cobordism from $(Y_1, \w_1)$ to $(Y_2, \w_2)$ and let $\s$ be a $\spinc$-structure on $W$. Define the \textit{reduced Euler characteristic of $\Gamma$} to be
\[
\widetilde{\chi}(\Gamma) = \chi(\Gamma) - \dfrac{1}{2}(|\w_1| + |\w_2|).
\]
The grading shift associated to $F^A_{W, \Gamma, \s}$ is then given by
\[
\Delta(W, \Gamma, \s) = \dfrac{c_1(\s)^2 - 2\chi(W) - 3\sigma(W)}{4} + \widetilde{\chi}(\Gamma).
\]
Note that if $\Gamma$ is a path, then the reduced Euler characteristic of $\Gamma$ is zero.

\subsection{Independence for paths}\label{sec:secZ.2}
In this subsection, we verify that if $\Gamma$ is a path, then the map $\smash{F^A_{W, \Gamma}}$ depends only on the homology class $[\Gamma] \in H_1(W, \partial W) / \text{Tors}$. This is rather well-known to experts, but we record it here for completeness. Note that if $\Gamma$ is a path, then $F^A$ and $F^B$ are homotopy equivalent and coincide with the usual construction of Ozsv\'ath and Szab\'o by \cite[Theorem B]{Zemkegraph}. In this situation we will thus write $F$ instead of $F^A$.

\begin{lemma}\label{lem:secZ.1}
Let $W$ be a cobordism between two singly-based (connected) 3-manifolds $(Y_{1},w_{1})$ and $(Y_{2},w_{2})$. Let $\gamma$ and $\gamma'$ be two paths in $W$ from $w_1$ to $w_2$. Suppose that 
\[
[\gamma - \gamma']=0 \in H_{1}(W) / \text{Tors}.
\] 
Then 
\[
F_{W,\gamma, \s} \simeq F_{W,\gamma', \s}.
\]
\end{lemma}

\begin{proof}
Decompose $W$ as in Section~\ref{sec:secZ.1}. Flow $\gamma$ and $\gamma'$ into $Y$ and denote the images of $w_1$ and $w_2$ in $Y$ by $v_1$ and $v_2$. We obtain two arcs in $Y$ that go between $v_1$ and $v_2$ which, by an abuse of notation, we continue to denote by $\gamma$ and $\gamma'$. Let $\fA_{\cG}$ and $\fA_{\cG'}$ be the graph action maps on $\CFm(Y)$ associated to $\gamma$ and $\gamma'$. Note that $c = \gamma*(\gamma')^{-1}$ is a closed loop in $Y$ which is zero when included into $H_{1}(W) / \text{Tors}$. We now have:
\begin{align*}
F_{W,\gamma, \s} - F_{W,\gamma', \s} &\simeq F_{W_2, \s_2} \circ  \fA_{\mathcal{G}} \circ F_{W_1, \s_1} -F_{W_2, \s_2} \circ  \fA_{\mathcal{G}'} \circ F_{W_1, \s_1}\\
& = F_{W_2, \s_2} \circ  (\fA_{\mathcal{G}}- \fA_{\mathcal{G}'}) \circ F_{W_1, \s_1} \\
& = F_{W_2, \s_2} \circ  S^{-}_{v_1}(A_{\gamma}-A_{\gamma'}) S^{+}_{v_2} \circ F_{W_1, \s_1} \\
& = F_{W_2, \s_2} \circ  S^{-}_{v_1}A_c S^{+}_{v_2} \circ F_{W_1, \s_1} \\
& \simeq F_{W_2, \s_2} \circ  A_cS^{-}_{v_1} S^{+}_{v_2} \circ F_{W_1, \s_1}
\end{align*}
Here, in the third line, we have used the definition of $\fA_{\cG_\Gamma}$ \cite[Equation 7.5]{Zemkegraph}, while in the fourth and fifth lines we have used \cite[Lemma 5.3]{Zemkegraph} and \cite[Lemma 6.13]{Zemkegraph}, respectively. 

Note that $A_c$ is the usual $H_1(Y)/\text{Tors}$-action on $\CFm(Y)$. We claim that the map $F_{W_2, \s_2} \circ  A_c$ is $U$-equivariantly nullhomotopic. For this, we use the following result from \cite{HN}. Let $W$ be a cobordism from $Y$ to $Y'$, and let $c \subseteq Y$ and $c' \subseteq Y'$ be two closed curves that are homologous in $W$. Then \cite[Theorem 3.6]{HN} states that
\[
F_{W, \s} \circ A_c \simeq A_{c'} \circ F_{W, \s},
\]
where $F_{W, \s}$ is the usual cobordism map of Ozsv\'ath and Szab\'o (see also results in \cite{OSabsgr}).\footnote{As written, \cite[Theorem 3.6]{HN} deals with the total homology map on $\HFhat$. However, the proof is easily modified to hold on the level of $U$-equivariant homotopy (for $\CFm$), and can be refined to take into account individual $\spinc$-structures. See \cite[Remark 3.7]{HN}.} In our case, note that $W_1$ consists of adding 1-handles to $Y_1$. An easy Mayer-Vietoris argument then shows that the inclusion of $H_1(W_2)$ into $H_1(W)$ is injective. Hence some multiple of $[c]$ is actually nullhomologous in $W_2$. The claim then follows from the above commutation relation by choosing $c'$ in $Y_2$ to be empty (or a small unknot).
\end{proof}

\subsection{Pseudo-homology bordisms and local equivalence}\label{sec:secZ.3}
In this subsection, we prove that any pseudo-homology bordism induces a local equivalence between the $\tau$-complexes (and $\ita$-complexes) of its incoming and outgoing ends. Throughout, let $(W, f)$ be a pseudo-homology bordism between $(Y_1, \tau_1)$ and $(Y_2, \tau_2)$, where $Y_1$ and $Y_2$ are disjoint unions of homology spheres. We equip each connected component of $Y_1$ and $Y_2$ with a single basepoint. For simplicity, assume that $W$ itself is connected.

Let $W_a$ be the cobordism formed by an iterated sequence of 1-handle attachments joining together the components of $Y_1$, as displayed in Figure~\ref{fig:secZ.2}. Let $W_b$ be (the reverse of) the analogous cobordism joining together the components of $Y_2$. Clearly, we can embed $W_a$ and $W_b$ in $W$ to obtain a decomposition
\[
W = W_b \circ W_0 \circ W_a,
\]
where $W_0$ is now a cobordism between two homology spheres. Note that the inclusion of $W_0$ into $W$ induces an isomorphism on $H_1$.

\begin{definition}\label{def:secZ.2}
We define a ribbon graph $\Gamma$ in $W$ as follows. On $W_a$, let $\Gamma$ be any trivalent 1-skeleton corresponding to the iterated sequence of 1-handle attachments, as displayed in Figure~\ref{fig:secZ.2}. For concreteness, we fix an ordering for the connected components of $Y_1$. (This specifies an order for taking the iterated connected sum, and also a way to choose a cyclic ordering at each internal vertex.) We define $\Gamma$ on $W_b$ similarly. To define $\Gamma$ on $W_0$, first choose a path $\gamma$ running between the two ends of $W_0$. Fix an ordered basis $e_1, \cdots, e_n$ of $H_1(W_0)$, and represent each $e_k$ by a simple closed curve $c_k$ that does not intersect $\gamma$. We then join each $c_k$ to $\gamma$ via an arc, which we refer to as a \textit{connecting arc}. Again, for concreteness, fix a cyclic ordering at each internal vertex. We call any $\Gamma$ constructed in this fashion a \textit{standard graph}. See Figure~\ref{fig:secZ.2}. 
\end{definition}

\begin{figure}[h!]
\center
\includegraphics[scale=0.9]{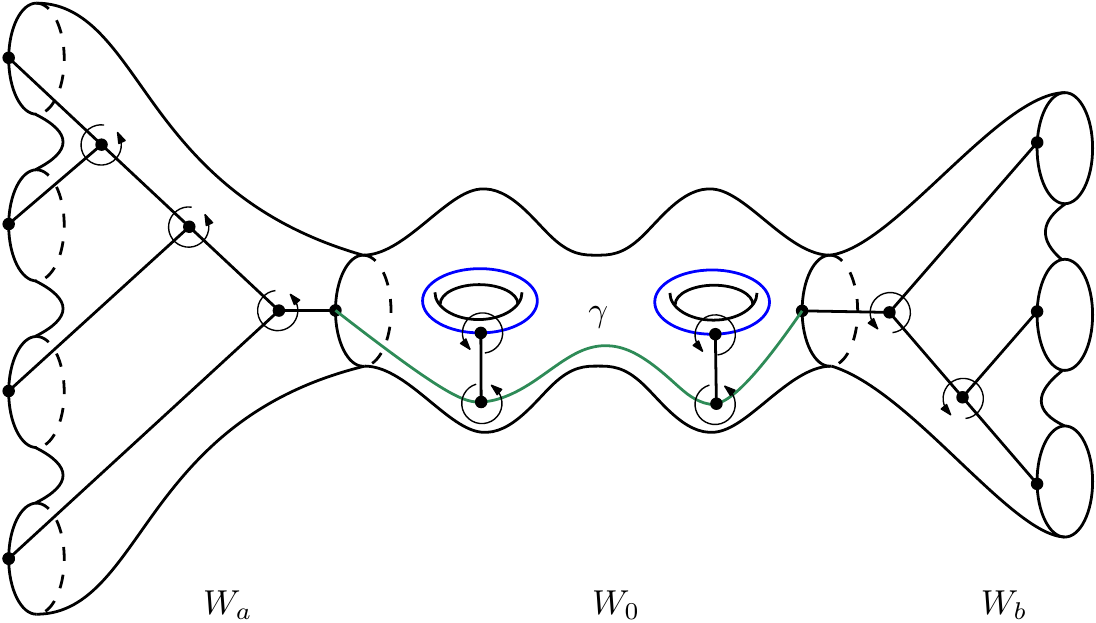}
\caption{Schematic decomposition $W = W_b \circ W_0 \circ W_a$. The path $\gamma$ is drawn in green, while the curves $c_k$ are drawn in blue. We choose the indicated cyclic ordering at each internal vertex.}\label{fig:secZ.2}
\end{figure} 

Now consider the cobordism map $\smash{F^A_{W, \Gamma}}$ associated to a standard graph. Our goal will be to show that this is a local map (with respect to both $\tau$ and $\ita$). As a first step, it will be helpful for us to have the following alternative formulation of $\smash{F^A_{W, \Gamma}}$. Let $\Gamma_\text{red}$ be the ``reduced" ribbon graph formed by replacing the subgraph $\Gamma \cap W_0$ in Definition~\ref{def:secZ.2} with the path $\gamma$. Let $W_\text{red}$ be obtained from $W$ by surgering out the curves $c_k$. Note that $W_\text{red} = W_b \circ W_h \circ W_a$, where $W_h$ is a homology cobordism. The image of $\Gamma_\text{red}$ under this surgery defines a ribbon graph in $W_\text{red}$, which we also denote by $\Gamma_{\text{red}}$. 

\begin{lemma}\label{lem:Z.3}
Let $\Gamma$ be a standard graph in $W$. Then
\[
F^A_{W, \Gamma} \simeq F^A_{W_\text{red}, \Gamma_\text{red}}.
\]
\end{lemma}
\begin{proof}
Note that by \cite[Proposition 11.1]{Zemkegraph}, the cobordism maps $F^A$ are unchanged under puncturing. More precisely, suppose that $(W, \Gamma)$ is any cobordism from $Y_1$ to $Y_2$. Puncture $W$ at any interior point and equip the new boundary $S^3$ with a single basepoint. We modify the original ribbon graph $\Gamma$ by joining this basepoint to $\Gamma$ via an arc (and choosing any cyclic ordering at the new internal vertex). Let the new incoming boundary be given by $Y_1 \sqcup S^3$. Then it follows from \cite[Proposition 11.1]{Zemkegraph} that under the identification of $\CFm(Y_1)$ with $\CFm(Y_1 \sqcup S^3) \simeq \CFm(Y_1) \otimes \CFm(S^3)$, the cobordism map remains unchanged up to $U$-equivariant homotopy. 

In our case, consider the cobordism $\smash{W_{S^1 \times B^3}}$ from $S^3$ to $S^1 \times S^2$ formed by puncturing $S^1 \times B^3$ at any interior point. We define a ribbon graph $\Gamma_{S^1 \times B^3}$ on $\smash{W_{S^1 \times B^3}}$ by taking a closed loop generating $H_1(S^1 \times B^3)$ and joining this to each boundary component via an arc. Now identify a neighborhood of each $c_k$ with $\nu(c_k) \cong S^1 \times B^3$, and puncture $W$ at an interior point of each of these neighborhoods. This punctured version of $W$ may be viewed as the composition of several copies of $(W_{S^1 \times B^3}, \Gamma_{S^1 \times B^3})$, together with the complement of the $\nu(c_k)$ in $W$. We similarly define $W_{D^2 \times S^2}$ by puncturing $D^2 \times S^2$ at any interior point and equipping this with an arc $\Gamma_{D^2 \times S^2}$ running between the two boundary components. Then $W_\text{red}$ may be viewed (after puncturing) as several copies of $(W_{D^2 \times S^2}, \Gamma_{D^2 \times S^2})$, together with the same complement as before. By the composition law, to establish the lemma it thus suffices to show that
\[
F^A_{W_{S^1 \times B^3}, \Gamma_{S^1 \times B^3}} \simeq F^A_{W_{D^2 \times S^2}, \Gamma_{D^2 \times S^2}}
\]
as maps from $\CFm(S^3)$ to $\CFm(S^1 \times S^2)$. This is a standard calculation.
\end{proof}

In light of Lemma~\ref{lem:Z.3}, the reader may wonder why we have not simply defined our cobordism maps directly in terms of $W_\text{red}$ and $\Gamma_\text{red}$, rather than $\Gamma$. (Indeed, this corresponds to the usual approach in Floer theory when dealing with cobordisms with $b_1 > 0$; see for example the proof of \cite[Theorem 9.1]{OSabsgr}.) The reason is that $c_k$ need not be fixed by $f$, so the surgered cobordism $W_\text{red}$ may not inherit an extension of $\tau_i$. Thus, \textit{a priori} there is no reason to think that the surgered cobordism interacts nicely with $\tau$. In actuality, we will show that $\smash{F^A_{W, \Gamma}}$ homotopy commutes with $\tau$, which implies that $\smash{F^A_{W_\text{red}, \Gamma_\text{red}}}$ does also. Alternatively, one can also define $F^A_{W, \Gamma}$ by considering the graph $\Gamma_\text{red}$ in $W$ and cutting down via the $H_1(W)/\text{Tors}$-actions of each of the $e_k$. This is essentially what we do in Lemma~\ref{lem:Z.7}, except in a language more amenable to that of \cite{Zemkegraph}.

When dealing with the action of $f$ on $W$, we will thus need to take a slightly different approach. We begin with a more refined decomposition theorem, which is essentially taken from the proof of \cite[Theorem 9.1]{OSabsgr}.

\begin{lemma}\label{lem:Z.4}
Let $W$ be a definite cobordism between two 3-manifolds. Then there exists a decomposition $W = W_2 \circ W_1$ of $W$ for which the following holds:
\begin{enumerate}
\item $W_1$ consists of 1- and 2-handles, 
\item $W_2$ consists of 2- and 3-handles; and, 
\item Let $Y$ be the slice given by the outgoing boundary of $W_1$. Then the map induced by the inclusion of $Y$ into $W$
\[
i_* : H_1(Y)/\text{Tors} \rightarrow H_1(W)/\text{Tors}
\]
is an isomorphism.
\end{enumerate}
\end{lemma}
\begin{proof}
Give $W$ a handle decomposition consisting of 1-handles, 2-handles, and 3-handles (attached in that order). According to the proof of \cite[Theorem 9.1]{OSabsgr}, we can re-index the sequence of 2-handle attachments as follows. Let the 2-handles be denoted by $\{h_i\}_{i = 1}^n$, and for each $i$ let $S_i$ be the outgoing boundary obtained after attaching $h_i$. Let the incoming boundary of the very first 2-handle be denoted by $S_0$. According to the proof of \cite[Theorem 9.1]{OSabsgr}, we may assume that the sequence of Betti numbers $\{b_1(S_i)\}_{i = 0}^n$ at first monotonically decreases with $i$, then is constant, and then finally monotonically increases with $i$. Ozsv\'ath and Szab\'o refer to such an ordering of the $h_i$ as a \textit{standard ordering}. This can be achieved whenever $W$ is definite.

We now choose $Y = S_i$ to be any slice in the above sequence for which $b_1(S_i)$ attains its minimum value. This decomposes $W$ into two subcobordisms $W_a$ and $W_b$ that obviously satisfy the first two desired properties. Let the 2-handles $h_j$ for $j > i$ be attached to $Y$ along a link whose components we denote by $\mathbb{K}_j$. We claim that each of these components is rationally nullhomologous in $Y$. Indeed, the condition $b_1(S_i) \leq b_1(S_{i+1})$ implies that $\mathbb{K}_{i+1}$ is rationally nullhomologous in $Y$; proceeding by induction, we assume that $\mathbb{K}_{i+1}, \dots, \mathbb{K}_l$ are rationally nullhomologous in $Y$. Now, $\mathbb{K}_{l+1}$ is rationally nullhomologous in $S_l$, which is obtained from $Y$ by integer surgery along $\mathbb{K}_{i+1}, \dots, \mathbb{K}_l$. The inductive hypothesis then easily implies that $\mathbb{K}_{l+1}$ is rationally nullhomologous in $Y$ also. 

It follows immediately that the induced inclusion map $i_* : H_1(Y)/\text{Tors} \rightarrow H_1(W_b)/\text{Tors}$ is an isomorphism, since $W_b$ is built from $Y \times I$ via attaching rationally nullhomologous 2-handles and then some 3-handles. Turning the cobordism around, we obtain the same result with $W_a$ in place of $W_b$. A standard Mayer-Vietoris argument then gives the desired claim. 
\end{proof}

\begin{definition}\label{def:Z.5}
Let $Y$ be any 3-manifold equipped with a collection of incoming basepoints $\mathcal{V}_\text{in}$ and outgoing basepoints $\mathcal{V}_\text{out}$. We say that a ribbon graph $\Lambda$ in $Y \times I$ is \textit{star-shaped} if it has a unique internal vertex, which is connected to each basepoint via a single arc. We also fix a formal ribbon structure; this corresponds to a cyclic ordering of $\mathcal{V}_\text{in} \cup \mathcal{V}_\text{out}$. Note that given any incoming basepoint $v_i$ and outgoing basepoint $v_j$, there is a unique path in $\Lambda$ going from $v_i$ to $v_j$, which we denote by $l_{ij}$.
\end{definition}

The proof of the next technical lemma is similar to that of \cite[Lemma 7.13]{Zemkegraph}. The authors would like to thank Ian Zemke for help with the proof and a discussion of the surrounding ideas.

\begin{lemma}\label{lem:Z.6}
Let $\Lambda$ and $\Lambda'$ be two star-shaped graphs in $Y \times I$. Suppose that for any incoming basepoint $v_i$ and outgoing basepoint $v_j$, we have
\[
[l_{ij} - l'_{ij}] = 0 \in H_1(Y) / \text{Tors}.
\]
Suppose moreover that $\Lambda$ and $\Lambda'$ have the same formal ribbon structure (viewed as cyclic orderings of the set of basepoints). Then for any $\spinc$-structure $\s$ on $Y \times I$, we have
\[
F^A_{Y \times I, \Lambda, \s} \simeq F^A_{Y \times I, \Lambda', \s}.
\]
\end{lemma}
\begin{proof}
Without loss of generality, we may isotope $\Lambda$ and $\Lambda'$ so that they share the same internal vertex $v$. For any basepoint $v_i$, denote the edge of $\Lambda$ joining $v_i$ to $v$ by $e_i$.\footnote{By \cite[Lemma 5.3]{Zemkegraph}, note that $A_{-e_i} = - A_{e_i}$. Since this coincides with $A_{e_i}$ mod $2$, we will occasionally use $e_i$ to also denote the same edge with reversed orientation.} We claim that there is a fixed element $\lambda \in H_1(Y)/\text{Tors}$ such that $[e_i' - e_i] = \lambda$ for all $i$. Indeed, consider any pair of incoming and outgoing vertices $v_i$ and $v_j$. Then
\[
[e_i' - e_i] - [e_j' - e_j] = [l_{ij}' - l_{ij}] = 0 \in H_1(Y)/\text{Tors}.
\]
Set $\lambda = [e_i' - e_i]$. Varying $j$ (and then varying $i$) gives the claim.

We now turn to the assertion of the lemma. Without loss of generality, let the basepoints of $Y$ be given by $\mathcal{V}_\text{in} \cup \mathcal{V}_\text{out} = \{v_i\}_{i = 1}^n$, and let the cyclic order corresponding to the formal ribbon structure be $v_1, \ldots, v_n$. By \cite[Equation 7.2]{Zemkegraph},
\[
F^A_{Y \times I, \Lambda} = \left( \prod_{x \in \mathcal{V}_\text{in} \cup \{v\}} S^-_{x} \right) \circ A_{e_n} \circ \cdots \circ A_{e_1} \circ \left( \prod_{x \in \mathcal{V}_\text{out} \cup \{v\}} S^+_{x} \right).
\]
A similar expression holds for $\Lambda'$ after replacing each $e_i$ with $e_i'$. By \cite[Lemma 5.3]{Zemkegraph} and the fact that $[e_i' - e_i] = \lambda$, we have $\smash{A_{e_i'} \simeq A_{e_i} + A_{\lambda}}$. Hence
\begin{align*}
A_{e_n'} \circ \cdots \circ A_{e_1'} &\simeq (A_{e_n} + A_\lambda) \circ \cdots \circ (A_{e_1} + A_\lambda) \\
&\simeq A_{e_n} \circ \cdots \circ A_{e_1} + A_{\lambda} \circ \left(\sum_i A_{e_n} \circ \cdots \circ \widehat{A}_{e_i} \circ \cdots \circ A_{e_1}\right).
\end{align*}
Here, the notation $\smash{\widehat{A}_{e_i}}$ means that $\smash{A_{e_i}}$ should be omitted from the composition. In the second line, we have expanded the product and used the fact that $A_\lambda \circ A_\lambda \simeq 0$ whenever $\lambda$ is a closed curve (see \cite[Lemma 5.5]{Zemkegraph}). Substituting this into the expression for $\smash{F^A_{Y \times I, \Lambda'}}$, it thus clearly suffices to show
\[
S^-_v \circ \left( \sum_i A_{e_n} \circ \cdots \circ \widehat{A}_{e_i} \circ \cdots \circ A_{e_1} \right) \circ S^+_v  \simeq 0.
\]
Throughout, we have used the fact that $A_{\lambda}$ commutes with the $A_{e_i}$ and the stabilization maps $S^{\pm}_v$, since $\lambda$ is a closed curve. (See \cite[Lemma 5.4]{Zemkegraph} and \cite[Lemma 6.13]{Zemkegraph}.)

\begin{figure}[h!]
\center
\includegraphics[scale=0.8]{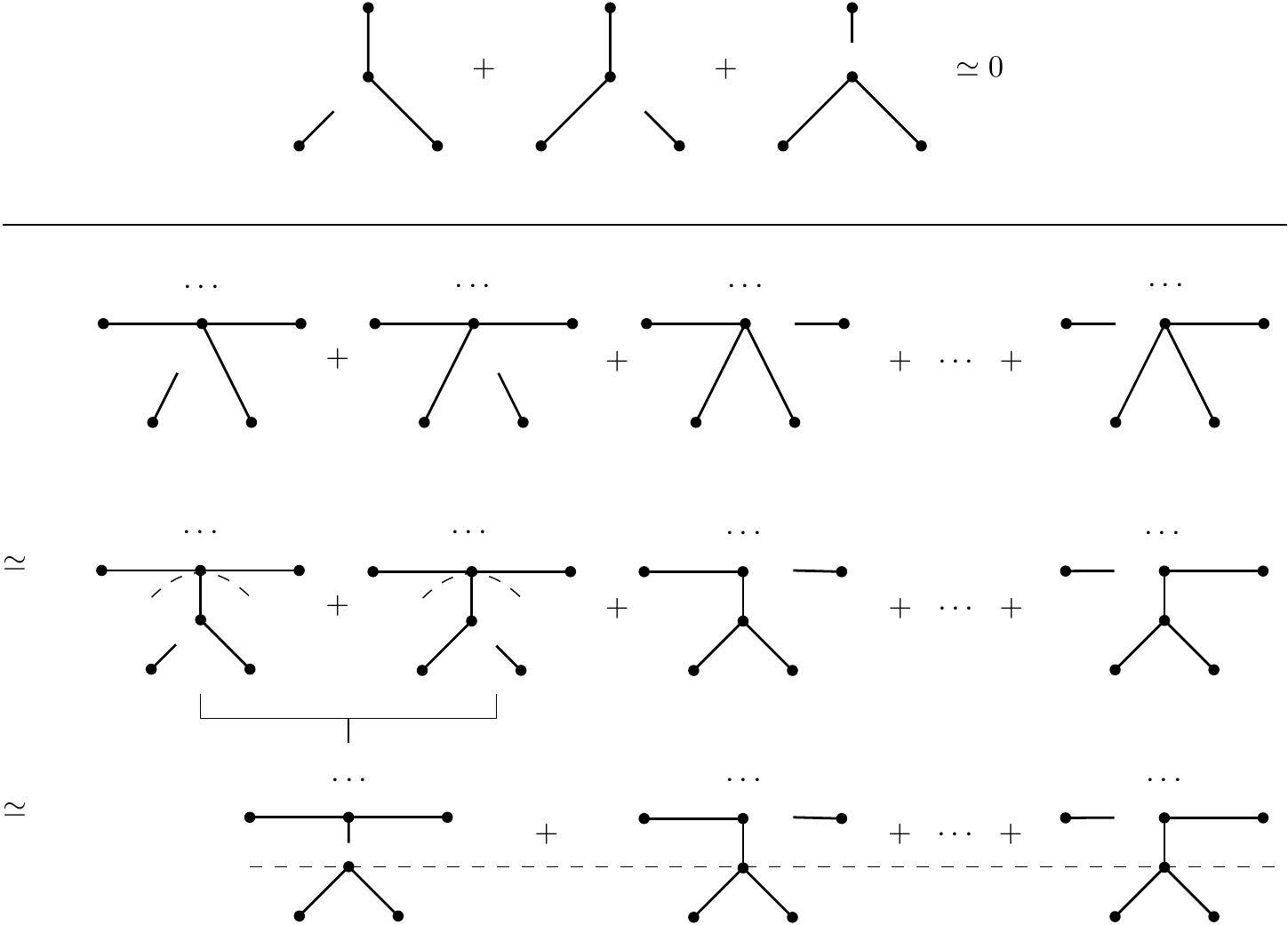}
\caption{Diagrammatic proof of Lemma~\ref{lem:Z.6}. The ellipses above each star-shaped graph indicate further edges attached to the interior vertex.}\label{fig:secZ.3}
\end{figure} 

We proceed by induction. For $n = 3$, we claim that 
\[
S^-_v (A_{e_3}A_{e_2} + A_{e_3}A_{e_1} + A_{e_2}A_{e_1}) S^+_v \simeq S^-_v(A_{e_3} + A_{e_2})(A_{e_2} + A_{e_1})S^+_v.
\]
This follows by expanding the right-hand side and noting that $S^-_v A_{e_2} A_{e_2} S^+_v \simeq U S^-_v S^+_v \simeq 0$ by Lemmas 5.5 and 6.15 of \cite{Zemkegraph}. On the other hand, we have
\[
S^-_v(A_{e_3} + A_{e_2})(A_{e_2} + A_{e_1})S^+_v \simeq S^-_v A_{e_3*e_2} A_{e_2* e_1}S^+_v \simeq A_{e_3 * e_{2}} A_{e_2 * e_1} S^-_v S^+_v 
\simeq 0.
\]
Here, to obtain the second homotopy equivalence, we have used \cite[Lemma 6.13]{Zemkegraph} and the fact that $e_3 * e_2$ and $e_2 * e_1$ are paths which do not have $v$ as an endpoint. This establishes the base case. 

The inductive step is diagrammatically described in Figure~\ref{fig:secZ.3}. In the first row of Figure~\ref{fig:secZ.3}, we have displayed three graphs corresponding to the three terms in the case $n = 3$. In the second row, we have displayed the sum in question for general $n$. We modify each of the graphs in the second row by introducing an additional internal vertex and edge, as in the third row of Figure~\ref{fig:secZ.3}. Note that this does not change the ribbon equivalence class. We then view the first two terms as composite graphs with the splittings indicated by the dashed arcs, and apply the $n = 3$ case to obtain the fourth row. We similarly view each graph in the fourth row as a composition of two subgraphs, corresponding to the pieces above and below the dashed line. Factoring out the map corresponding to the subgraph below the dashed line, the remaining sum is precisely the inductive hypothesis for $n-1$. This completes the proof.
\end{proof}

We now come to the central lemma of this section:

\begin{lemma}\label{lem:Z.7}
Let $(W, f)$ be a pseudo-homology bordism and let $\Gamma$ be a standard graph in $W$. Then
\[
F^A_{W, \Gamma} \simeq F^A_{W, f(\Gamma)}.
\]
\end{lemma}
\begin{proof}
For convenience, denote $\Gamma' = f(\Gamma)$. Decompose $W$ as in Lemma~\ref{lem:Z.4}, and flow $\Gamma$ and $\Gamma'$ into the slice $Y$ afforded by Lemma~\ref{lem:Z.4}. (Here, we are using the fact that $W_a$ consists of 1- and 2-handles, while $W_b$ consists of 2- and 3-handles.) Without loss of generality, we may thus assume that $\Gamma$ and $\Gamma'$ agree outside of $Y \times I$. By abuse of notation, we denote the subgraphs $\Gamma \cap (Y \times I)$ and $\Gamma' \cap (Y \times I)$ by $\Gamma$ and $\Gamma'$ also. Applying the composition law, it clearly suffices to prove that $F^A_{Y \times I, \Gamma} \simeq F^A_{Y \times I, \Gamma'}$. 
Note that we implicitly equip $Y \times I$ with the pullback of the single $\spinc$-structure on $W$. See the top-left of Figure~\ref{fig:secZ.4}.

\begin{figure}[h!]
\center
\includegraphics[scale=0.75]{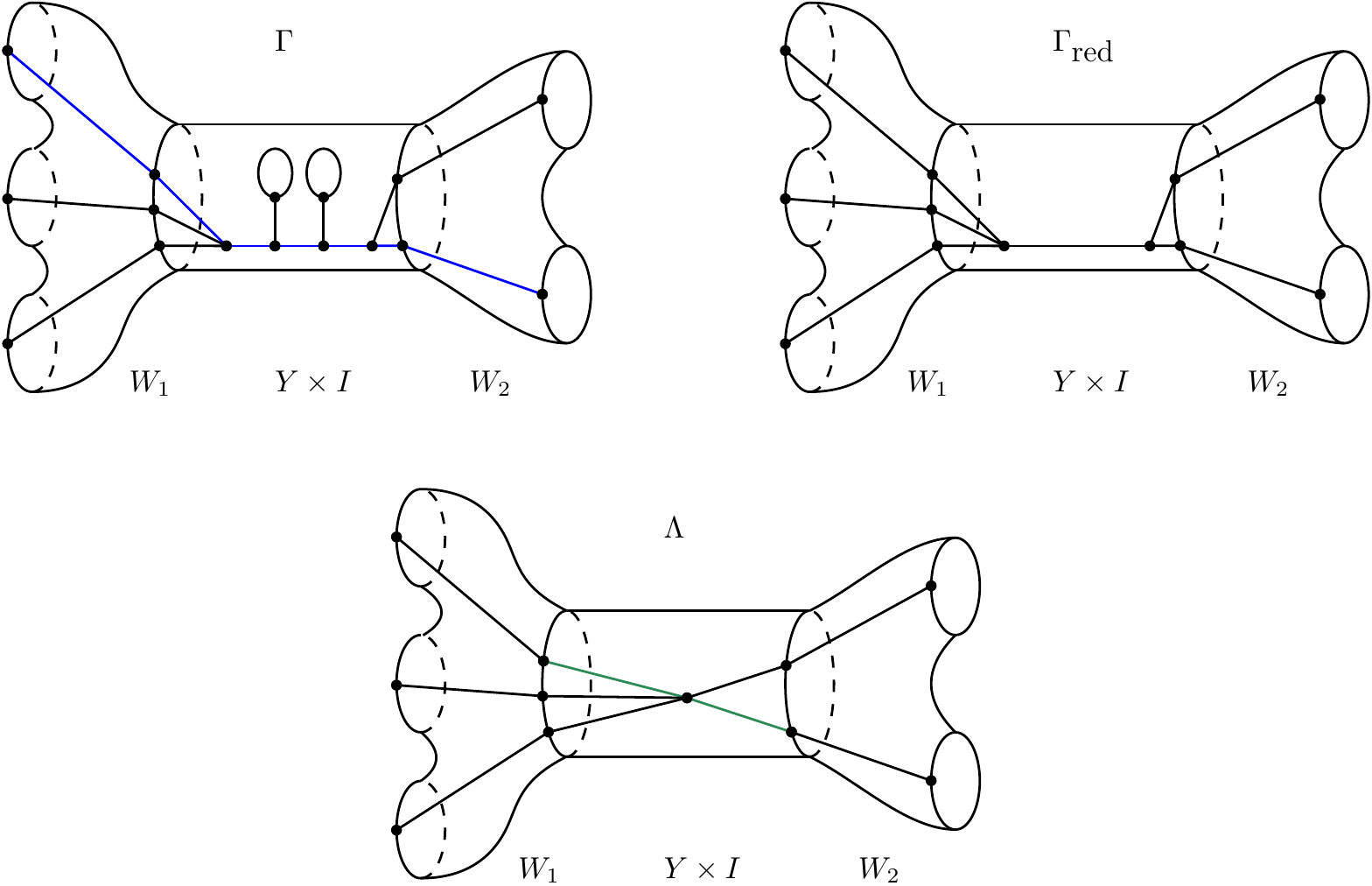}
\caption{Top left: the flowed graph $\Gamma$. Top right: the modified graph $\Gamma_\text{red}$. Bottom middle: the graph $\Lambda$. The path $l_{ij}$ from the proof of Lemma~\ref{lem:Z.7} is marked in green; the path $g_{ij}$ is marked in blue. In general, $Y$ will have some topology.}\label{fig:secZ.4}
\end{figure} 

Define $\Gamma_\text{red}$ to be $\Gamma$ with the curves $c_k$ and connecting arcs deleted. By \cite[Proposition 4.6]{Zemkeduality}, we have\footnote{Compare Figure~\ref{fig:secZ.4} and \cite[Figure 4.5]{Zemkeduality}. In our case, contracting each individual connecting arc to a point does not change the ribbon equivalence class.}
\[
F^A_{Y \times I, \Gamma} \simeq F^A_{Y \times I, \Gamma_\text{red}} \circ \left( \prod_k A_{c_k} \right).
\]
Note that $\Gamma'$ is combinatorially isomorphic to $\Gamma$. In particular, $\Gamma'$ consists of a set of closed loops $c_k'$, which are joined to an underlying tree via connecting arcs. These loops are in correspondence with the analogous loops $c_k$ in $\Gamma$. Defining $\Gamma'_\text{red}$ similarly, we have
\[
F^A_{Y \times I, \Gamma'} \simeq F^A_{Y \times I, \Gamma_\text{red}'} \circ \left( \prod_k A_{c_k'} \right).
\]
Since $f$ acts as the identity on homology, we have $[c_k'] = [c_k]$ in $H_1(W)$ for each $k$. By Lemma~\ref{lem:Z.4}, this implies that $[c_k'] = [c_k]$ in $H_1(Y)/\text{Tors}$, and thus that $\smash{A_{c_k'} \simeq A_{c_k}}$ for each $k$ by \cite[Proposition 5.8]{Zemkegraph}. Hence to establish the claim, it suffices to prove that $\smash{F^A_{Y \times I, \Gamma_\text{red}} \simeq F^A_{Y \times I, \Gamma_\text{red}'}}$. See the top-right of Figure~\ref{fig:secZ.4}.

We now contract all of the internal edges in $\Gamma_\text{red}$ to obtain a star-shaped graph $\Lambda$, as displayed in the second row of Figure~\ref{fig:secZ.4}. This does not change the ribbon equivalence class of $\Gamma_\text{red}$. We similarly contract all the edges of $\Gamma_\text{red}'$ to obtain a star-shaped graph $\Lambda'$. It remains to verify the hypotheses of Lemma~\ref{lem:Z.6}. Let $v_i$ be an incoming basepoint in $Y \times I$ and let $v_j$ be an outgoing basepoint. Let $g_{ij}$ be the obvious path in $\Gamma$ (viewed as a graph in $W$) going between the corresponding basepoints $w_i$ and $w_j$ of $W$, as in Figure~\ref{fig:secZ.4}. Define $\smash{g'_{ij}}$ similarly. Note that $g_{ij}$ and $\smash{g'_{ij}}$ agree outside of $Y \times I$, and $[g_{ij}] = [\smash{g'_{ij}}] \in H_1(W, \partial W)$ since $f$ acts as the identity on $H_1(W, \partial W)$. Clearly, $l_{ij}$ and $g_{ij} \cap (Y \times I)$ are isotopic in $Y \times I$ (rel boundary), and similarly for $\smash{l'_{ij}}$ and $\smash{g'_{ij}}$. Hence
\[
[l_{ij} - l'_{ij}] = [g_{ij} - g'_{ij}] = 0 \in H_1(W).
\]
By Lemma~\ref{lem:Z.4}, we thus have that $[l_{ij} - l'_{ij}] = 0$ in $H_1(Y)/\text{Tors}$. Applying Lemma~\ref{lem:Z.6} completes the proof.
\end{proof}

We are now in a position to prove that $h_{\tau}$ and $h_{\ita}$ are homomorphisms from $\G$ to $\Inv$:

\begin{proof}[Proof of Theorem~\ref{thm:1.2}]
Let $(W, f)$ be a pseudo-homology bordism from $(Y_1, \tau_1)$ to $(Y_2, \tau_2)$. We wish to show:
\begin{enumerate}
\item $F^A_{W, \Gamma} \circ \iota_1 \simeq \iota_2 \circ F^A_{W, \Gamma}$; 
\item $F^A_{W, \Gamma} \circ \tau_1 \simeq \tau_2 \circ F^A_{W, \Gamma}$; and,
\item $F^A_{W, \Gamma}$ maps $U$-nontorsion elements in homology to $U$-nontorsion elements in homology (and has zero grading shift).
\end{enumerate}
The first and third claims follow immediately from Lemma~\ref{lem:Z.3} and standard results of Hendricks, Manolescu, and Zemke. Indeed, according to Lemma~\ref{lem:Z.3}, we have
\[
F^A_{W, \Gamma} \simeq F^A_{W_\text{red}, \Gamma_\text{red}}.
\]
The latter cobordism is equal to the composition $W_b \circ W_h \circ W_a$, where the outer two terms are compositions of connected sum cobordisms (or their reverses), and $W_h$ is a homology cobordism equipped with a path $\gamma$. By \cite[Proposition 5.10]{HMZ} and \cite[Proposition 4.9]{HM}, the maps associated to each of these pieces commutes with $\iota$ up to homotopy. Applying the composition law, we thus see that $\smash{F^A_{W, \Gamma}}$ homotopy commutes with $\iota$ also. The third claim is similarly verified by establishing the desired condition for each piece. To prove the second claim, we apply (\ref{eq:1}) and Lemma~\ref{lem:Z.7}:
\[
F^A_{W, \Gamma} \circ \tau_1 \simeq \tau_2 \circ F^A_{W, f(\Gamma)} \simeq \tau_2 \circ F^A_{W, \Gamma}.
\]
This proves that $\smash{F^A_{W, \Gamma}}$ is a local map with respect to $\tau$. Turning $W$ around shows that $h_{\tau_1}(Y_1) = h_{\tau_2}(Y_2)$, as desired. To show that $\smash{F^A_{W, \Gamma}}$ preserves $h_{\ita}$, we apply the first and second claims to obtain
\[
F^A_{W, \Gamma} \circ (\iota_1 \circ \tau_1) \simeq (\iota_2 \circ \tau_2) \circ F^A_{W, \Gamma}.
\]
Hence $h_\tau$ and $h_{\ita}$ are well-defined maps from $\G$ to $\Inv$. Since $\CFm$ takes disjoint unions to tensor products (for the trivial coloring), this completes the proof.
\end{proof}

For completeness, we also record:

\begin{proposition}\label{lem:secZ.4}
Let $Y_1$ and $Y_2$ be homology spheres with involution $\tau_1$ and $\tau_2$. Suppose that we have $\tau_i$-equivariant balls $B_i$ in $Y_i$ and a diffeomorphism from $B_1$ to $B_2$ which intertwines $\tau_1$ and $\tau_2$. Then we have a $U$-equivariant homotopy equivalence
\[
(\CFm(Y_1 \# Y_2), \tau_1 \# \tau_2) \simeq (\CFm(Y_1) \otimes \CFm(Y_2) [-2], \tau_1 \otimes \tau_2).
\]
Here, recall that the connected sum is performed along the given diffeomorphism.
\end{proposition}
\begin{proof}
This is just a special case of the proof of Theorem~\ref{thm:1.2}, except that the maps associated to connected sum cobordisms are homotopy equivalences (see \cite[Proposition 5.2]{HMZ}).
\end{proof}

\begin{remark}\label{rem:summand}
Using the connected sum formula, one can prove that the kernel of $F$ contains an infinite-rank summand via the following observation, which is essentially due to Lin, Ruberman, and Saveliev \cite[Remark 8.4]{LRS}. Let $Y$ be a Brieskorn sphere with $d(Y) = 0$ but $h(Y) \neq 0$, and let $\tau_1$ be the Montesinos involution on $Y$. If $\tau_2$ is any involution isotopic to the identity map on $Y$, then we may form the connected sum involution $\tau = \tau_1 \# \tau_2$ on $Y \# \overline{Y}$. Then
\[
h_{\tau}(Y \# \overline{Y}) = h_{\tau_1}(Y) = h(Y) \neq 0,
\]
while $Y \# \overline{Y}$ tautologically bounds a homology ball. Combining such examples with the results of \cite{DHSTcobordism} yields the desired summand. However, these need not in general bound contractible manifolds. Moreover, they are all manifestly reducible, while in our case each $W_n$ is irreducible (see \cite[Corollary 2.7]{Yamada}).
\end{remark}


\section{Results and Computations}\label{sec:5}
\subsection{Preliminary examples}\label{sec:5.1}
We begin with several toy examples of the theory we have built up so far. We remind the reader that all cobordisms in this section are equivariant (in the sense that they admit diffeomorphisms extending involutions on the boundary components). 

Let $Y_1 = \Sigma(2, 3, 7)$ be given by $(-1)$-surgery on the right-handed trefoil. There are two involutions on the trefoil, which are displayed in Figure~\ref{fig:5.1}. 

\begin{lemma}\label{lem:5.1}
Let $Y_1 = \Sigma(2, 3, 7)$ be given by $(-1)$-surgery on the right-handed trefoil. Let $\tau$ and $\sigma$ be the strong and periodic involutions in Figure~\ref{fig:5.1}, respectively. Then
\begin{enumerate}
\item $h_\tau(Y_1) = h(Y_1) < 0$ and $h_{\ita}(Y_1) = 0$; and,
\item $h_\sigma(Y_1) = 0$ and $h_{\iota \circ \sigma}(Y_1) = h(Y_1) < 0$.
\end{enumerate}
\end{lemma}
\begin{proof}
Figure~\ref{fig:5.1} constitutes a $(-1)$-cobordism from $S^3$ to $Y_1$, which conjugates $\spinc$-structures in the case of $\tau$ and fixes $\spinc$-structures in the case of $\sigma$. By Theorem~\ref{thm:1.4}, we thus have $0 \leq h_{\ita}(Y_1)$.  Now, $\id$ and $\iota$ are the only two possible homotopy involutions on the standard complex of $\CFm(Y_1)$, and the involutive complex corresponding to $\iota$ is strictly less than zero. Hence $\ita = \id$, which shows $\tau = \iota$. Similarly, we have that $0 \leq h_\sigma(Y_1)$, which implies $\sigma = \id$ and $\iota \circ \sigma = \iota$.
\end{proof}
\begin{figure}[h!]
\center
\includegraphics[scale=0.6]{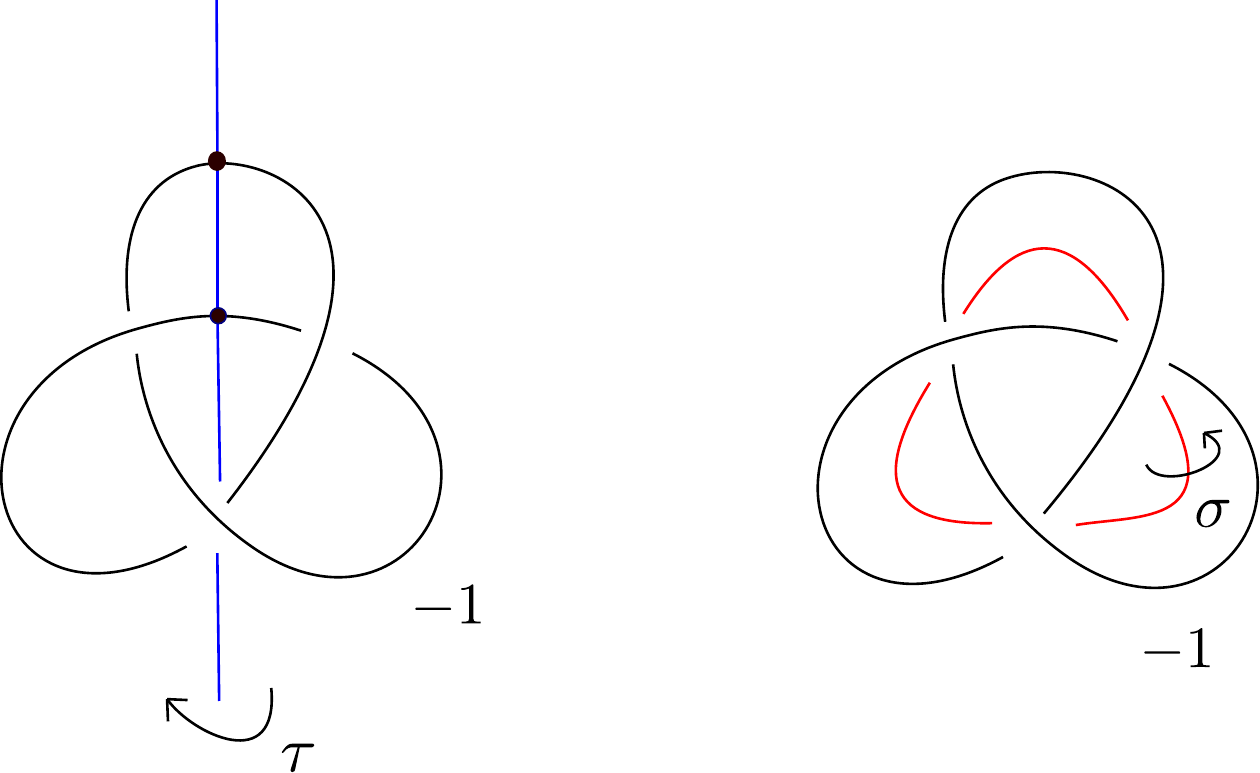}
\caption{Strong and periodic involutions on the right-handed trefoil. The cobordism from $S^3$ to $Y_1$ is tautologically given by $(-1)$-framed handle attachment along the right-handed trefoil in $S^3$.}\label{fig:5.1}
\end{figure}

Note that there are certainly other ways to determine the actions of $\tau$ and $\sigma$ on $\HFm(Y_1)$. For example, one can check Lemma~\ref{lem:5.1} by turning the equivariant surgery diagram of Figure~\ref{fig:5.1} into a double-branched-cover picture (see for example the algorithm described \cite[Section 1.1.12]{Savelievinvariants}). In the case of $\tau$, this exhibits $Y_1$ as coming from the double branched cover over the Montesinos knot $k(2, 3, 7)$, while $\sigma$ comes from the double branched cover over the torus knot $T_{3, 7}$. (For $p$ and $q$ odd, the involution coming from the double branched cover over $T_{p, q}$ is isotopic to the identity, and hence acts trivially on Floer homology.) Observe, however, that identifying the branch locus even in this simple case is not immediate, as the knots in question have fourteen crossings. In contrast, our computation is almost trivial, given the formal properties of Floer cobordism maps.

We now consider $Y_1 = \Sigma(2, 3, 7)$ as $(+1)$-surgery on the figure-eight knot. We have drawn two mirrored copies of this in Figure~\ref{fig:5.2}, with involutions suggestively denoted by $\tau$ and $\sigma$. It turns out that $\tau$ and $\sigma$ are \textit{not} the same element of $\textit{MCG}(Y_1)$, and in fact act differently on the Floer homology. (This is because even though the figure-eight knot is amphichiral, it is not \textit{equivariantly} amphichiral.) Of course, to show this, one can simply find a diffeomorphism between the trefoil and the figure-eight surgery descriptions of $\Sigma(2, 3, 7)$ and show that the maps $\tau$ and $\sigma$ in Figure~\ref{fig:5.2} correspond to those in Figure~\ref{fig:5.1} . Here, however, we show that our approach immediately distinguishes the two actions:

\begin{lemma}\label{lem:5.2}
Let $Y_1 = \Sigma(2, 3, 7)$ be given by $(+1)$-surgery on the figure-eight knot, and let $\tau$ and $\sigma$ be as in Figure~\ref{fig:5.2}. Then 
\begin{enumerate}
\item $h_\tau(Y_1) = h(Y_1) < 0$ and $h_{\ita}(Y_1) = 0$
\item $h_\sigma(Y_1) = 0$ and $h_{\iota \circ \sigma}(Y_1) = h(Y_1) < 0$.
\end{enumerate}
\end{lemma}
\begin{proof}
Doing $(+1)$-surgery on the unknot indicated on the left in Figure~\ref{fig:5.2} (and blowing down) gives a $\spinc$-conjugating $(+1)$-cobordism from $(Y_1, \tau)$ to $S^3$. Hence $h_{\ita}(Y_1) \geq 0$. Similarly, doing $(+1)$-surgery on the unknot indicated on the right gives a $\spinc$-fixing $(+1)$-cobordism from $(Y_1, \sigma)$ to $S^3$. Hence $h_\sigma(Y_1) \geq 0$. The claim then follows as in the proof of Lemma~\ref{lem:5.1}.
\end{proof}

\begin{figure}[h!]
\center
\includegraphics[scale=0.47]{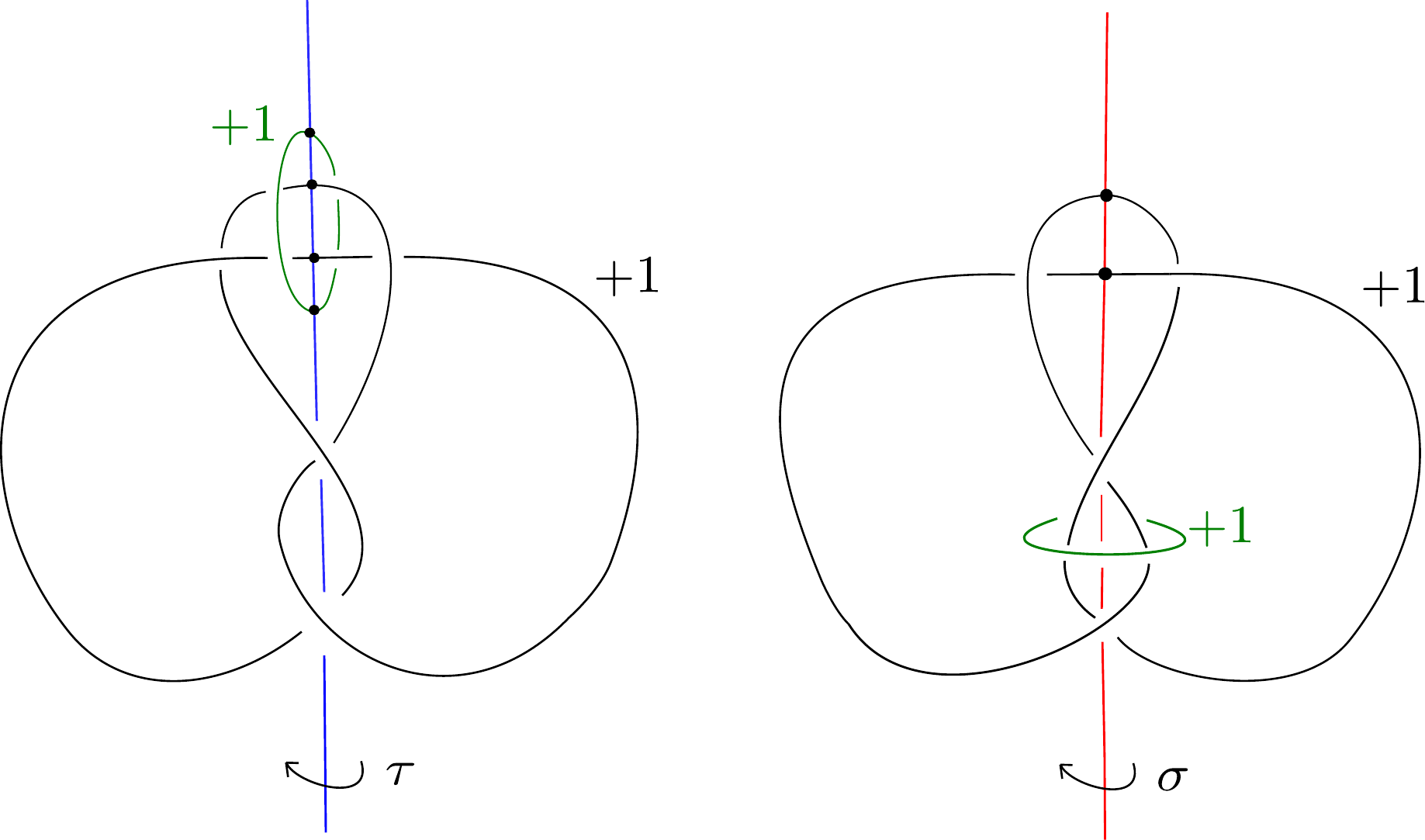}
\caption{Two involutions on the figure-eight knot, with equivariant cobordisms of Lemma~\ref{lem:5.2}.}\label{fig:5.2}
\end{figure}

We now turn to our first example of a cork. Let $Y_2$ be given by $(+1)$-surgery on the stevedore knot $6_1$, displayed on the left in Figure~\ref{fig:5.3}. Note that $Y_2$ bounds a Mazur manifold (see for example Lemma~\ref{lem:5.7}). Alternatively, it is a general fact that any $(1/k)$-surgery on a slice knot is a homology sphere which bounds a contractible manifold; see for example \cite[\S 6, Corollary 3.1.1]{Gordon}. We begin with the following weak form of the calculation in Example~\ref{ex:3.5}.

\begin{figure}[h!]
\center
\includegraphics[scale=0.6]{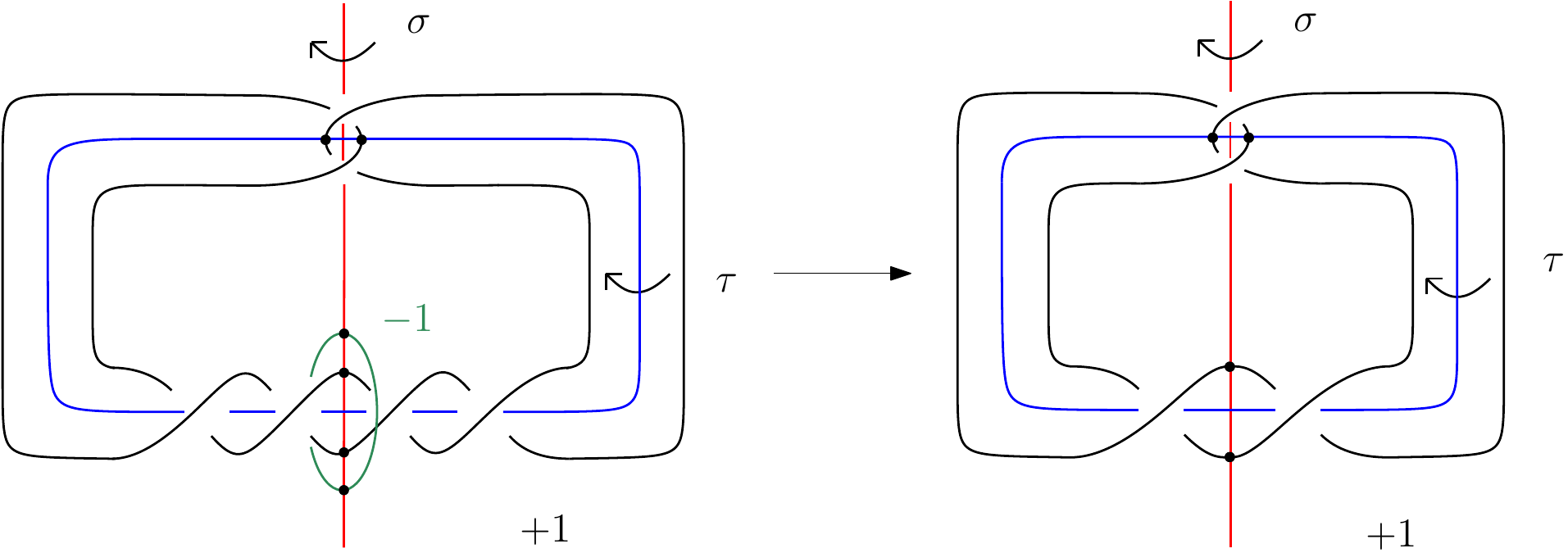}
\caption{Cobordism from $S_{+1}(6_1)$ to $\Sigma(2, 3, 7)$.}\label{fig:5.3}
\end{figure}

\begin{lemma}\label{lem:5.3}
Let $Y_2 = S_{+1}(6_1)$ be given by $(+1)$-surgery on the stevedore knot $6_1$, and let $\tau$ and $\sigma$ be as shown on the left in Figure~\ref{fig:5.3}. Then $h_\tau(Y_2) < 0$ and $h_{\iota \circ \sigma}(Y_2) < 0$. In particular, neither $\tau$ nor $\sigma$ extends over any homology ball that $Y$ bounds.
\end{lemma}
\begin{proof}
The claim is immediate from Figure~\ref{fig:5.3}. Doing $(-1)$-surgery on the indicated unknot gives a $\spinc$-fixing cobordism from $(Y_2, \tau)$ to $(Y_1, \tau)$ and a $\spinc$-reversing cobordism from $(Y_2, \sigma)$ to $(Y_1, \sigma)$. It is easily checked that the involutions $\tau$ and $\sigma$ on the right in Figure~\ref{fig:5.3} are the same as those defined in Lemma~\ref{lem:5.2}.
\end{proof}

Lemma~\ref{lem:5.3} already shows that $Y_2 = S_{+1}(6_1)$ is a (strong) cork (with either of the involutions $\tau$ and $\sigma$). To the best of the authors' knowledge, even the fact that $Y_2$ bounds a cork was not previously known. Again, we stress here that the entire argument is almost completely formal: the only actual computation we have used so far is the (involutive) Floer homology of $Y_1 = \Sigma(2, 3, 7)$. In particular, we have not needed to determine the Floer homology of $Y_2$ (involutive or otherwise). 

We now demonstrate that in simple cases, it is possible to use the techniques we have developed so far to explicitly compute the action of the mapping class group on $\HFm(Y)$. Note that by Lemma~\ref{lem:3.3}, the induced action of any involution must commute with $\iota$. In the case of $Y_2$, this constraint (together with Lemma~\ref{lem:5.3}) will suffice to prove that the induced actions of $\tau$ and $\sigma$ are as advertised in Example~\ref{ex:3.5}.

\begin{lemma}\label{lem:computation}
The induced actions of $\tau$ and $\sigma$ on $\HFm(Y_2)$ are as in Example~\ref{ex:3.5}.
\end{lemma}
\begin{proof}
The first order of business is to determine the action of $\iota$ on $\HFm(Y_2)$. Since the stevedore knot is alternating, its knot Floer homology is determined by its Alexander polynomial and signature \cite{Petkova}. It is thus easily checked that the knot Floer complex is as given on the left in Figure~\ref{fig:5.4}. We can then use the involutive Heegaard Floer large surgery formula of Hendricks and Manolescu \cite[Theorem 1.5]{HM} to determine the homotopy type of the complex $(\CFm(Y_2), \iota)$. (Note that the stevedore knot has genus one.) More precisely, we consider the map $\iota_0$ on $A_0^-$ induced by the knot Floer involution $\iota_K$. This determines the map $\iota$ on the surgered complex $\CFm(Y_2)$; see \cite[Section 6.1]{HM} and \cite[Section 6.3]{HM}. In \cite[Section 8.3]{HM}, Hendricks and Manolescu determined $\iota_K$ for all thin knots, using the fact that $\iota_K$ squares to the Sarkar involution. The resulting calculation in our case is displayed in Figure~\ref{fig:5.4}, and gives the action of $\iota$ on $\HFm(Y)$ shown on the right. A straightforward basis change then gives the action of $\iota$ shown in Figure~\ref{fig:3.5}.

\begin{figure}[h!]
\center
\includegraphics[scale=1.1]{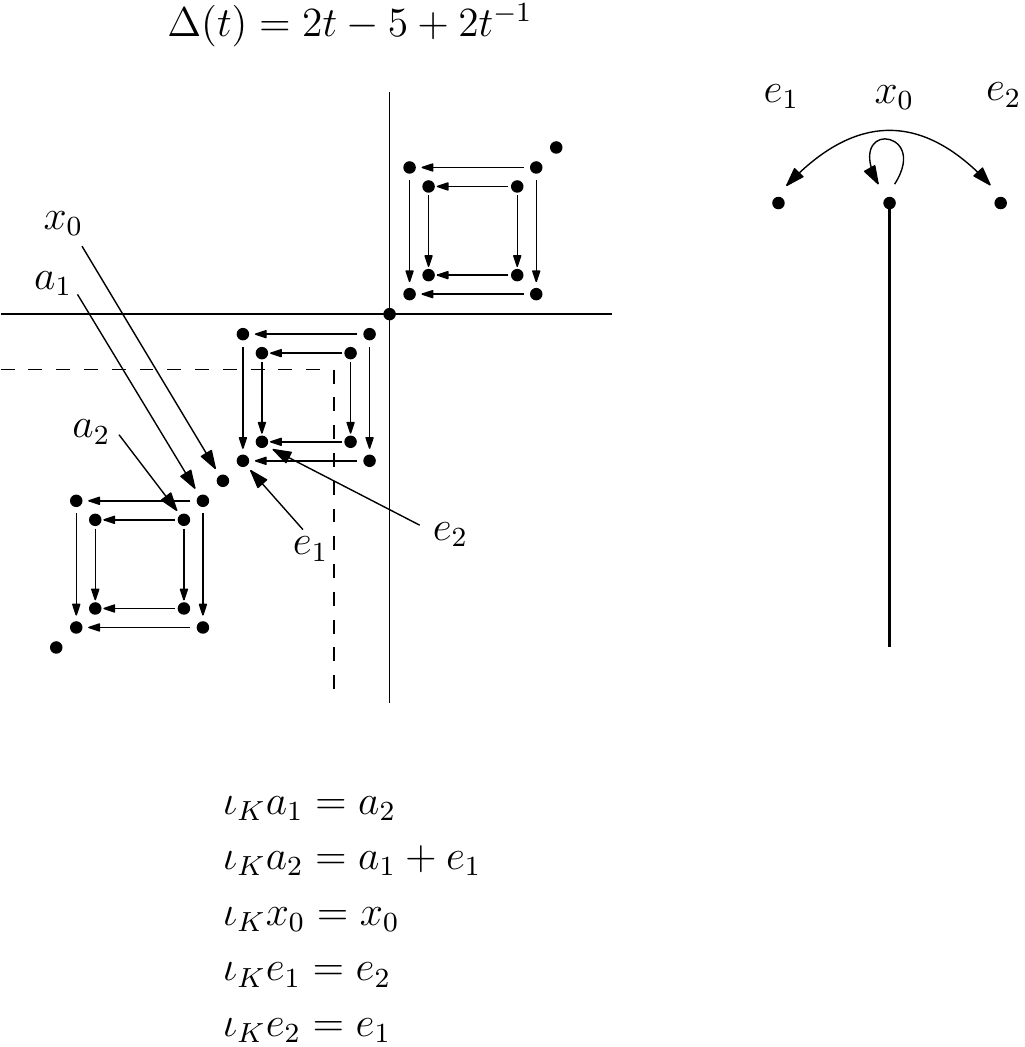}
\caption{Left: knot Floer complex of $6_1$. Right: Heegaard Floer homology of $S_{+1}(6_1)$, with $\iota$-action. Below: the action of $\iota_K$.}\label{fig:5.4}
\end{figure}

Now let us consider the actions of $\tau$ and $\sigma$ on $\HFm(Y_2)$. By Lemma~\ref{lem:3.3}, these commute with the action of $\iota$. We claim that up to isomorphism, there are only four involutions on $\HFm(Y)$ which commute with $\iota$: the identity, $\iota$ itself, and the maps $\tau$ and $\sigma$ displayed in Figure~\ref{fig:3.5}. Indeed, let $f : \HFm(Y_2) \rightarrow \HFm(Y_2)$ be such an involution. Since $f^2 = \id$, clearly $f$ must map $U$-nontorsion elements to $U$-nontorsion elements. Thus we have that $f(v_1)$ equals one of $v_1, v_2, v_3,$ or $v_1 + v_2 + v_3$, where the $v_i$ are as in Figure~\ref{fig:3.5}. In this last case, we may perform an $\iota$-equivariant change of basis and set $v_2' = v_1 + v_2 + v_3$, so without loss of generality we may assume that $f(v_1) = v_i$ for some $i$. Applying $\iota$-equivariance, we thus see that we must have
\begin{enumerate}
\item $f(v_1) = v_1$ and $f(v_3) = v_3$; or,
\item $f(v_1) = v_3$ and $f(v_3) = v_1$; or,
\item $f(v_1) = f(v_3) = v_2$.
\end{enumerate}
A similar argument shows that either
\begin{enumerate}
\item $f(v_2) = v_2$; or,
\item $f(v_2) = v_1 + v_2 + v_3$.
\end{enumerate}
It is straightforward to check that the third possibility for $f(v_1)$ is then ruled out, as neither possibility for $f(v_2)$ gives a map which squares to the identity. This shows that there are exactly four possibilities for $f$, which (together with the identity) are precisely those enumerated in Figure~\ref{fig:3.5} As we saw in Example~\ref{ex:3.5}, three of the resulting local equivalence classes are trivial, while one is strictly less than zero. Hence Lemma~\ref{lem:5.3} actually forces the actions of $\tau$ and $\sigma$ to be as described in Figure~\ref{fig:3.5}.
\end{proof}
We have essentially the same computation for the Akbulut cork $M_1$. Figure~\ref{fig:3.6} displays a $(-1, -1)$-cobordism from $M_1$ to $\Sigma(2,3,7)$, which is interchanging for both $\tau$ and $\sigma$. (The reader should check that the involutions $\tau$ and $\sigma$ on $\Sigma(2, 3, 7)$ in Figure~\ref{fig:3.6} are the same as those displayed in Figure~\ref{fig:5.2}.) In particular, we have that $h_\tau(M_1) < 0$ and $h_{\iota \circ \sigma}(M_1) < 0$. This immediately shows that $(M_1, \tau)$ and $(M_1, \sigma)$ are both strong corks. The reader should compare the following with \cite[Section 8.1]{LRS}:

\begin{lemma}\label{lem:akbulut}
The induced actions of $\tau$ and $\sigma$ on $\HFm(M_1)$ are as in Example~\ref{ex:akbulut}.
\end{lemma}

\begin{proof}
The computation of $\HFm(M_1)$ appears in \cite{AD}, but for our purposes it is more convenient to use the involutive large surgery formula and the fact that $M_1$ is $(+1)$-surgery on the pretzel knot $P(-3, 3, -3)$. Indeed, $P(-3, 3, -3)$ is alternating and has Alexander polynomial given by $\Delta(t) = 2t - 5 + 2t^{-1}$. Hence the same computation as in Lemma~\ref{lem:computation} applies. Since we analogously have $h_\tau(M_1) < 0$ and $h_{\iota \circ \sigma}(M_1) < 0$, the same algebraic analysis as in Lemma~\ref{lem:computation} completes the proof.
\end{proof}

\noindent
One can also prove that $M_1$ is a strong cork by using its (usual) description as surgery on a two-component link. See the proof of Theorem~\ref{thm:1.10}.
\subsection{Brieskorn spheres}\label{sec:5.2}

In this subsection, we analyze our invariants for a number of Brieskorn spheres. Though these examples do not bound corks, exploiting the equivariant $2$-handle cobordisms constructed in Section \ref{sec:4} will allow us to use these computations to produce the families of strong corks claimed in Section~\ref{sec:1}. One input here comes from the work of Alfieri-Kang-Stipsicz \cite{AKS}, who showed that if $\tau$ is the involution on $Y = \Sigma(p, q, r)$ obtained by viewing $Y$ as the double branched cover of the Montesinos knot $k(p, q, r)$, then $\tau \simeq \iota$. In fact, thanks to the following theorem (implicit in \cite{BO} and \cite{MS}), all we will need from their work is that there \textit{exists} a geometric involution on $Y$ which acts as $\iota$ on $\CFm(Y)$:

\begin{theorem}\cite{BO, MS}\label{thm:5.4}
If $Y = \Sigma(p, q, r)$ is a Brieskorn (integer) homology sphere other than $S^3$ or $\Sigma(2, 3, 5)$, then $\textit{MCG}(Y) = \Z/2\Z$.
\end{theorem}
\begin{proof}
We give an explanation here, as the authors could not find the result explicitly stated in the literature. According to \cite[Theorem 9.1]{MS}, every diffeomorphism of $Y$ is isotopic to one that maps fibers to fibers. Each such diffeomorphism thus induces a diffeomorphism of the base orbifold $O$. The map from the space of fiber-preserving diffeomorphisms $\text{Diff}_f(Y)$ to the space of orbifold diffeomorphisms $\smash{\text{Diff}^{\phantom{.}orb}(O)}$ is a fibration over its image (see the discussion in \cite[Section 8]{MS}):
\[
\text{Diff}_v(Y) \rightarrow \text{Diff}_f(Y) \rightarrow \text{Diff}_0^{\phantom{.}orb}(O).
\]
The fiber $\text{Diff}_v(Y)$ is the space of vertical diffeomorphisms, which by Lemmas 8.1 and 8.2 of \cite{MS} is homotopy equivalent to $S^1$. We claim that $\smash{\text{Diff}^{\phantom{.}orb}(O)}$ has two connected components. Indeed, since $p, q,$ and $r$ are distinct, clearly any diffeomorphism of $O$ must fix the three cone points individually; it is then easily checked that up to isotopy, there is one orientation-preserving and one orientation-reversing diffeomorphism of $O$. Hence $\textit{MCG}(Y)$ has at most two elements. To see that the Montesinos involution $\tau$ is nontrivial, note that if $\tau$ were trivial in $\textit{MCG}(Y)$, then the induced action of $\tau$ on $\CFm(Y)$ would be the identity. Since $\tau \simeq \iota$, by \cite{DM} this is only possible if $Y$ is a Heegaard Floer $L$-space. However, the only Brieskorn sphere which is an $L$-space is $\Sigma(2, 3, 5)$ (see for example \cite[Theorem 1.6]{CK}).
\end{proof}
\noindent
Hence for Brieskorn spheres, there are at most two possibilities for the induced action of $\tau$ on $\CFm(Y)$ (up to $U$-equivariant homotopy): either the identity or $\iota$.

For the results in this paper, we will only need to consider two families of Brieskorn spheres. The computation of the invariants of these may be viewed as generalizations of Lemmas~\ref{lem:5.2} and \ref{lem:5.1}, respectively.

\begin{figure}[h!]
\center
\includegraphics[scale=1]{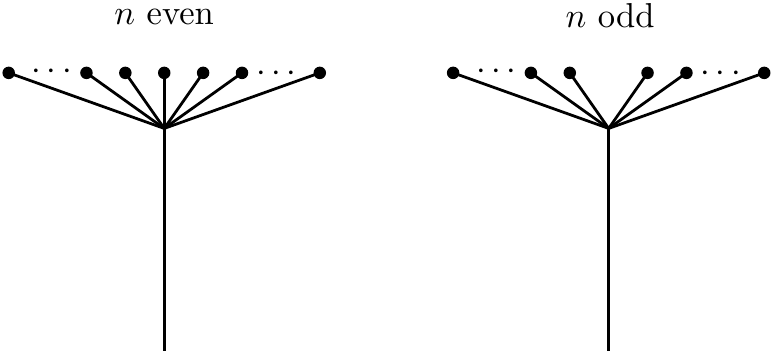}
\caption{Local equivalence class $h(A_n)$.}\label{fig:5.5}
\end{figure}

\begin{lemma}\label{lem:5.5}
Let $K_n$ be the family of twist knots displayed in Figure~\ref{fig:5.6}, equipped with the indicated involutions $\tau$ and $\sigma$. Let $A_n = S_{+1}(K_n) = \Sigma(2, 3, 6n+1)$. For $n$ positive and odd, we have
\begin{enumerate}
\item $h_{\tau}(A_n) = h_{\iota \circ \sigma}(A_n) = h(A_n) < 0$
\item $h_{\ita}(A_n) = h_{\sigma}(A_n) = 0$.
\end{enumerate}
\end{lemma}

\begin{proof}
The Heegaard Floer homology $\HFm(A_n)$ is displayed in Figure~\ref{fig:5.5}. This can be computed either by using the usual Heegaard Floer surgery formula, or by using the graded roots algorithm of \cite{CK}. The action of $\iota$ on $\HFm(A_n)$ is given by reflection across the obvious vertical axis. Using the monotone root algorithm of \cite[Section 6]{DM}, $h(A_n)$ is locally trivial for $n$ even and locally equivalent to $h(\Sigma(2, 3, 7))$ for $n$ odd. In the latter case, this means that $h(A_n) < 0$. In Figure~\ref{fig:5.6}, we have displayed a cobordism from $A_n$ to $S^3$ consisting of $n$ unknots with framing $+1$. Note that this is $\spinc$-conjugating for $\tau$ (since $n$ is odd) and $\spinc$-fixing for $\sigma$. Hence $h_{\ita}(A_n) \geq 0$. This implies $\tau \simeq \iota$, since either $\tau \simeq \iota$ or $\tau \simeq \id$. Similarly, we have $h_{\sigma}(A_n) \geq 0$, which implies $\sigma \simeq \id$.
\end{proof}

\begin{figure}[h!]
\center
\includegraphics[scale=0.6]{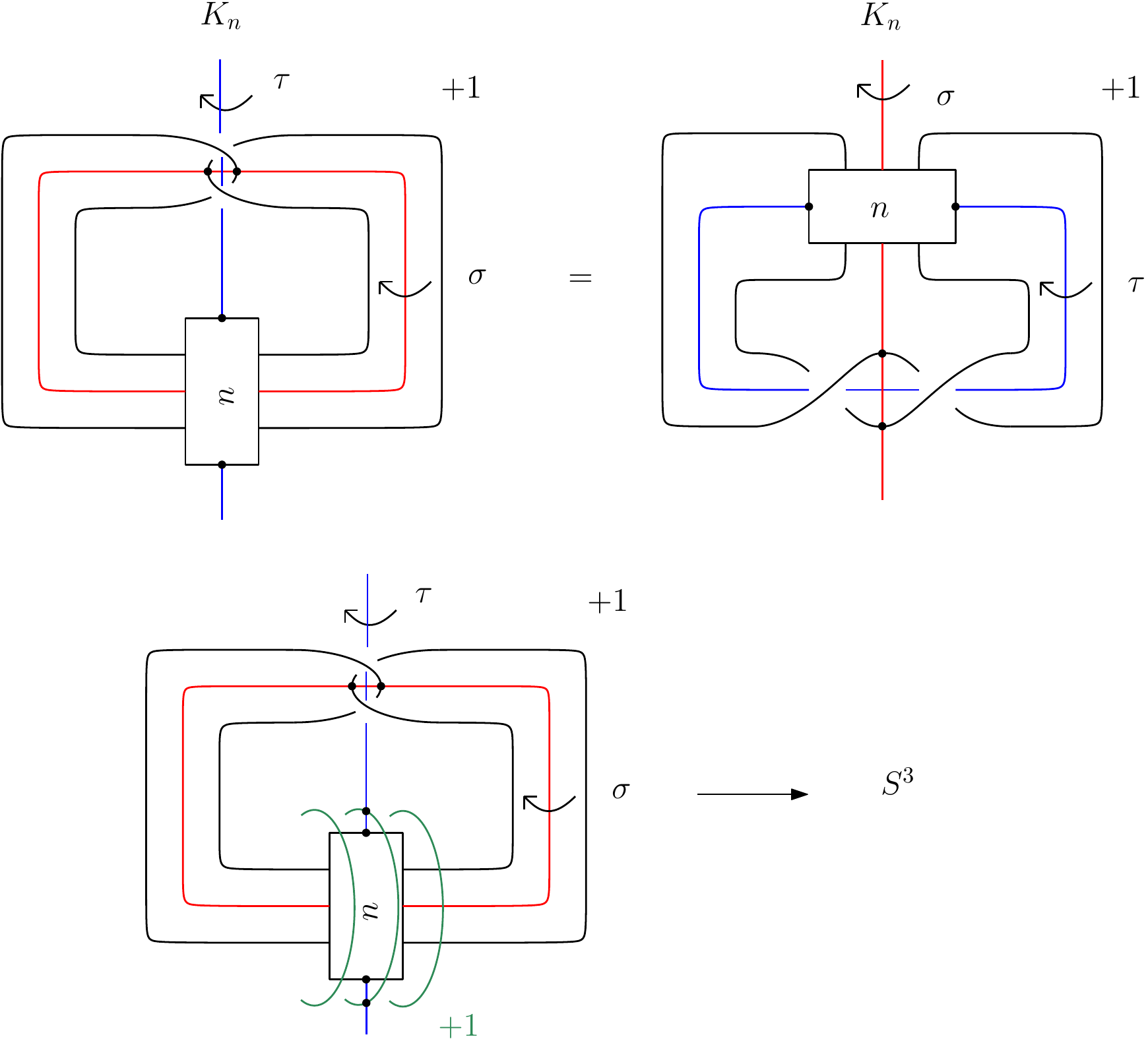}
\caption{Top: two equivalent diagrams for $A_n = S_{+1}(K_n)$. Bottom: cobordism from $A_n$ to $S^3$.}\label{fig:5.6}
\end{figure}

\begin{lemma}\label{lem:5.6}
Let $T_{2, 2n+1}$ be the right-handed $(2, 2n+1)$-torus knot, equipped with the involutions $\tau$ and $\sigma$ indicated in Figure~\ref{fig:5.7}. Let $B_n = S_{-1}(T_{2, 2n+1}) = \Sigma(2, 2n+1, 4n+3)$. Then
\begin{enumerate}
\item $h_{\tau}(B_n) = h_{\iota \circ \sigma}(B_n) = h(B_n) = X_{\lfloor (n+1)/2 \rfloor} < 0$
\item $h_{\ita}(B_n) = h_{\sigma}(B_n) = 0$
\end{enumerate}
\end{lemma}
\begin{proof}
As discussed after Example~\ref{ex:2.9}, $\smash{h(B_n) = X_{\lfloor (n+1)/2 \rfloor}}$, which is strictly less than zero. Figure~\ref{fig:5.7} constitutes a $(-1)$-cobordism from $S^3$ to $Y$, which conjugates $\spinc$-structures in the case of $\tau$ and fixes $\spinc$-structures in the case of $\sigma$. The proof then proceeds as in Lemma~\ref{lem:5.5}.
\end{proof}

\begin{figure}[h!]
\center
\includegraphics[scale=0.8]{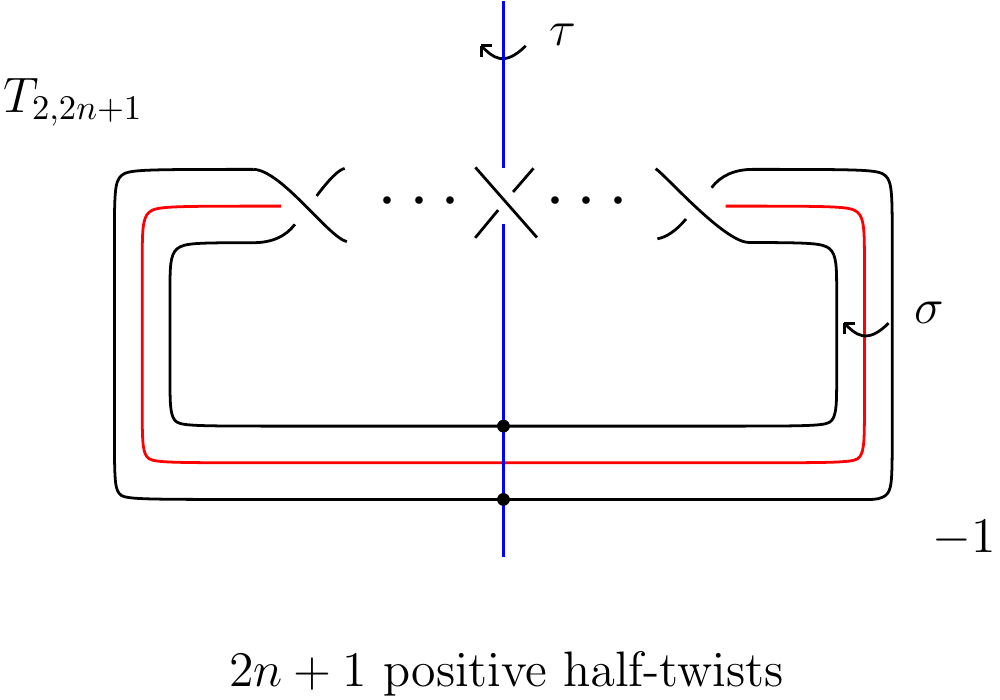}
\caption{Diagram of $B_n = S_{-1}(T_{2, 2n+1})$.}\label{fig:5.7}
\end{figure}
\noindent
We will actually not need the details of Lemma~\ref{lem:5.6}, but we include it for completeness. Note of course that $B_1 = A_1 = \Sigma(2, 3, 7)$.

\subsection{Surgeries on knots} \label{sec:5.3}
We now turn to several examples of corks given by surgeries on equivariant slice knots. To the best of the authors' knowledge, this method for generating corks has not generally been explored in the literature. Note that $(1/k)$-surgery on a slice knot always bounds a contractible manifold; see for example \cite[Section 6]{Gordon}. In fact, we have the following standard result:

\begin{lemma}\label{lem:5.7}
Let $K$ be a slice knot that bounds a ribbon disk with two minima. Then $(1/k)$-surgery on $K$ yields a homology sphere that bounds a Mazur manifold.
\end{lemma}
\begin{proof}
Let $Y$ be $(1/k)$-surgery on $K$. Note that $K$ can be drawn by starting with an unlink of two components, and then attaching a band $b$ connecting them (where $b$ of course may link the two components many times). However, doing an additional $0$-surgery on an unknot that goes around $b$ turns $Y$ into $S^1 \times S^2$, as shown in Figure~\ref{fig:5.8}. Hence $Y$ admits a cobordism to $S^3$ consisting of a single $2$-handle (given by the trace of the surgery) and then a single $3$-handle (filling in the $S^1 \times S^2$). Turning this cobordism around gives the desired Mazur manifold.
\end{proof}
\begin{figure}[h!]
\center
\includegraphics[scale=0.85]{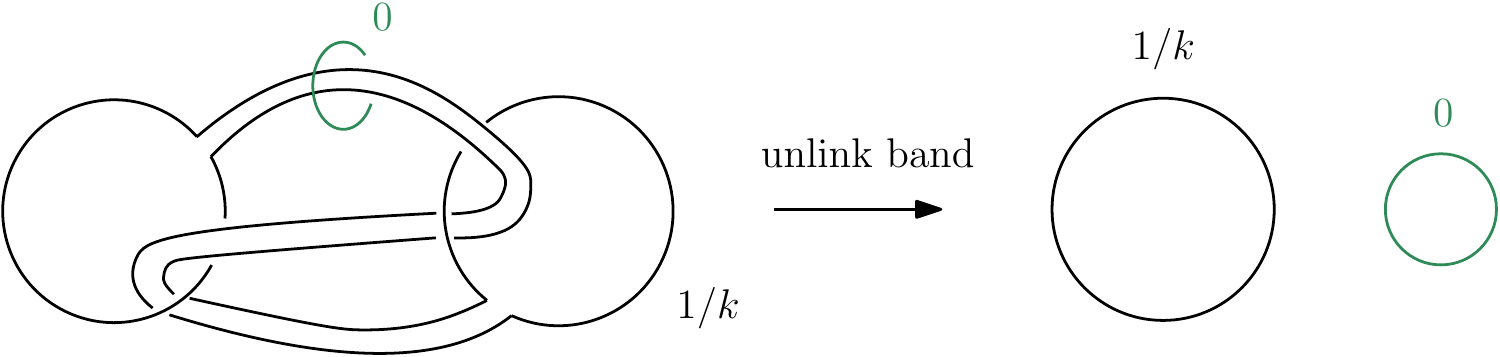}
\caption{Ribbon disk with two minima (and no maxima).}\label{fig:5.8}
\end{figure}

We are now in place to establish Theorems~\ref{thm:1.8} and \ref{thm:1.9}. Recall that $V_{n, k}$ is defined to be $(1/k)$-surgery on the doubly twist knot:
\[
V_{n,k}  = 
\begin{cases}
S_{1/k}(\overline{K}_{-n, n+1}) &\text{if } n \text{ is odd}\\
S_{1/k}(K_{-n, n+1}) &\text{if } n \text{ is even}.
\end{cases}
\]
Each $V_{n,k}$ is equipped with the involutions $\tau$ and $\sigma$ displayed in Figure~\ref{fig:1.1} (or rather, the mirrored involutions in the case that $n$ is odd). \\

\begin{figure}[h!]
\center
\includegraphics[scale=0.7]{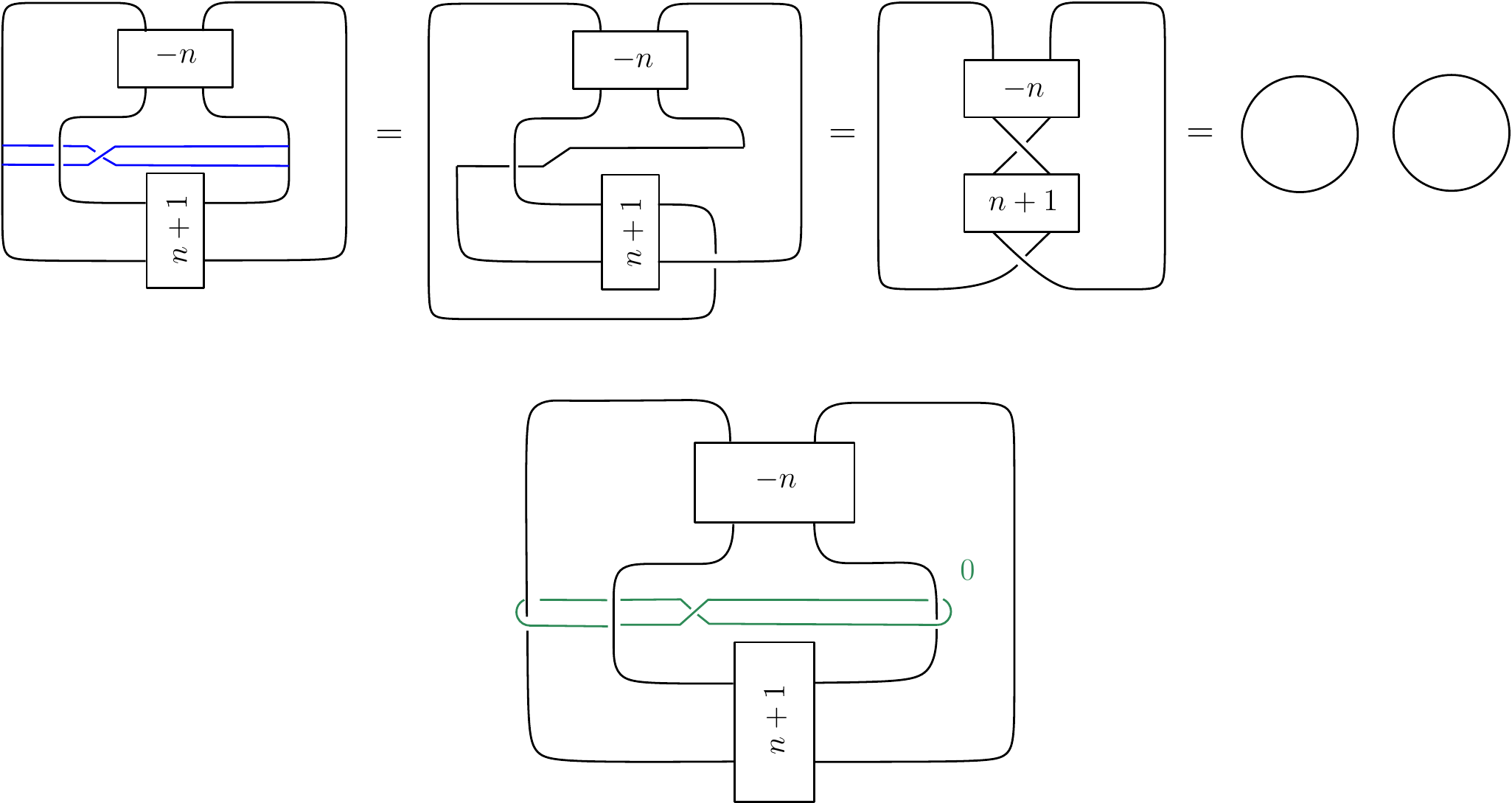}
\caption{Ribboning for $K_{-n, n+1}$.}\label{fig:5.9}
\end{figure}

\begin{proof}[Proof of Theorem~\ref{thm:1.8}]
First note that $K_{-n, n+1}$ bounds a slice disk with two minima (and no maxima). One possible ribboning is displayed along the top row of Figure~\ref{fig:5.9}. We claim $\tau \circ \sigma$ extends over the Mazur manifold afforded by Lemma~\ref{lem:5.7}. To see this, we utilize the following technique of Akbulut \cite{Akbulut4}. Note that the fusion band $b$ in the proof of Lemma~\ref{lem:5.7} is dual to the blue fission band of Figure~\ref{fig:5.9}. An unknot $\gamma$ going around $b$ is thus drawn in green in Figure~\ref{fig:5.9}. It is easily checked that the isotopy class of $\gamma$ is fixed by $\tau \circ \sigma$, so $\tau \circ \sigma$ extends over the $2$-handle attachment along $\gamma$. The claim then follows since every diffeomorphism of $S^1 \times S^2$ extends over $S^1 \times B^3$.

We now turn to a study of $\tau$ and $\sigma$. We claim that for $k = 1$, we have
\begin{enumerate}
\item If $n$ is odd, $h_{\ita}(V_{n,1}) \leq h_{\iota \circ \sigma}(A_n)$ and $h_{\sigma}(V_{n,1}) \leq h_\tau(A_n)$.
\item If $n$ is even, $h_{\tau}(V_{n,1}) \leq h_{\tau}(A_{n+1})$ and $h_{\iota \circ \sigma}(V_{n,1}) \leq h_{\iota \circ \sigma}(A_{n+1})$.
\end{enumerate}
The relevant equivariant surgeries are displayed in Figure~\ref{fig:5.10}. (Compare Figure~\ref{fig:5.6}.) Note that we always attach an odd number of $(-1)$-framed 2-handles. In the case that $n$ is odd, note that $\tau$ acts as a strong involution on a single unknot and interchanges the others in pairs, while $\sigma$ acts as a periodic involution on each unknot. (The roles of $\tau$ and $\sigma$ are reversed in the case where $n$ is even.) By Lemma~\ref{lem:5.5}, we thus see that all of the above local equivalence classes are strictly less than zero. Applying Theorem~\ref{thm:1.5} completes the proof.
\end{proof}

\begin{figure}[h!]
\center
\includegraphics[scale=0.65]{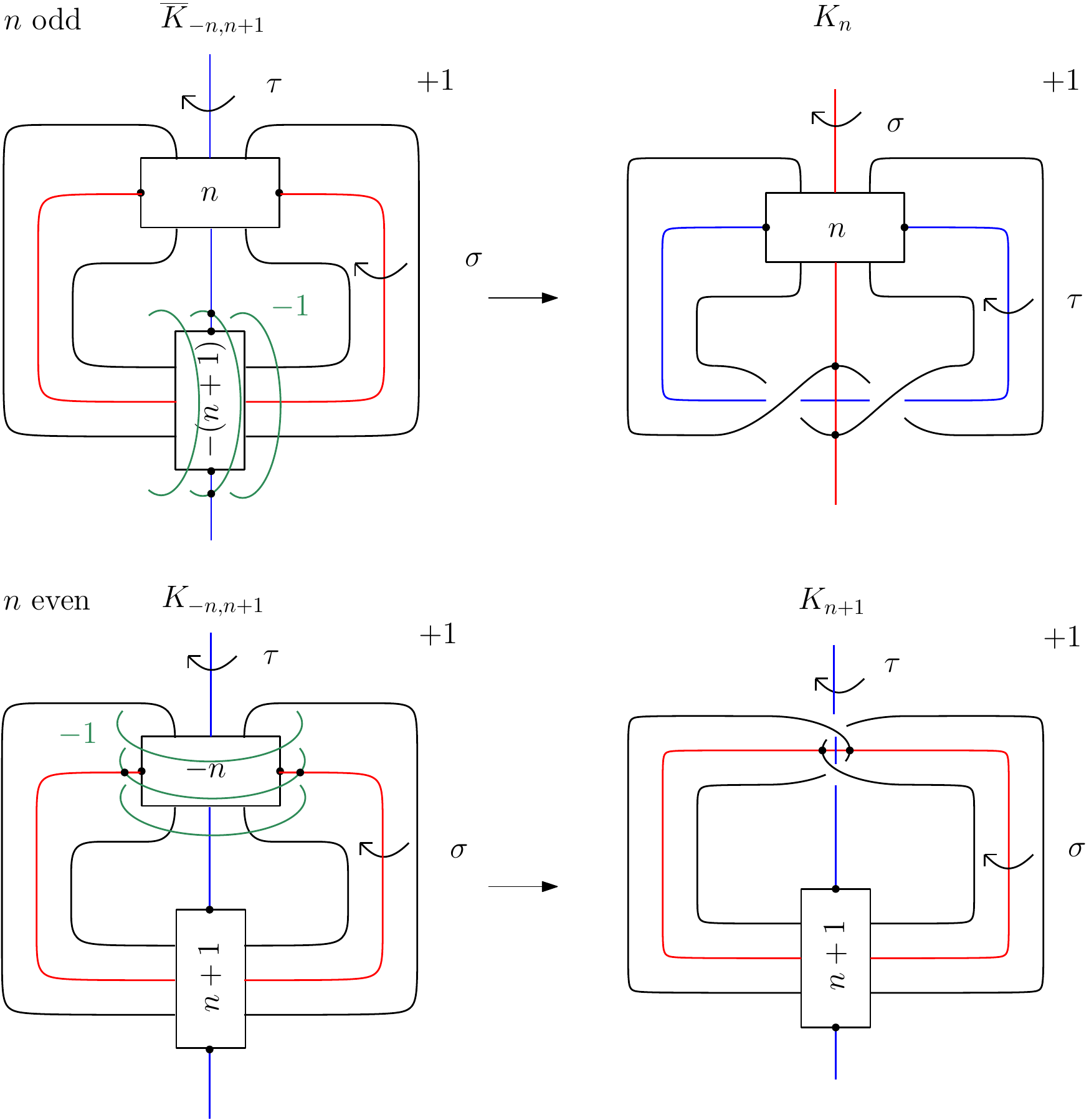}
\caption{Top ($n$ odd): cobordism from $V_{n,1}$ to $A_n$; there are $n$ green curves. Note that $\tau$ on $\overline{K}_{-n, n+1}$ is sent to $\sigma$ on $K_n$. Bottom ($n$ even): cobordism from $V_{n,1}$ to $A_{n+1}$; there are $n - 1$ green curves.}\label{fig:5.10}
\end{figure}

\begin{proof}[Proof of Theorem~\ref{thm:1.9}]
We first consider $(+1)$-surgery on the slice knots in question. In each case, Figure~\ref{fig:5.11} displays a negative-definite cobordism to $\Sigma(2, 3, 7)$ with appropriate involution(s). Applying Theorem~\ref{thm:1.5} completes the proof.
\end{proof}

\begin{figure}[h!]
\center
\includegraphics[scale=0.23]{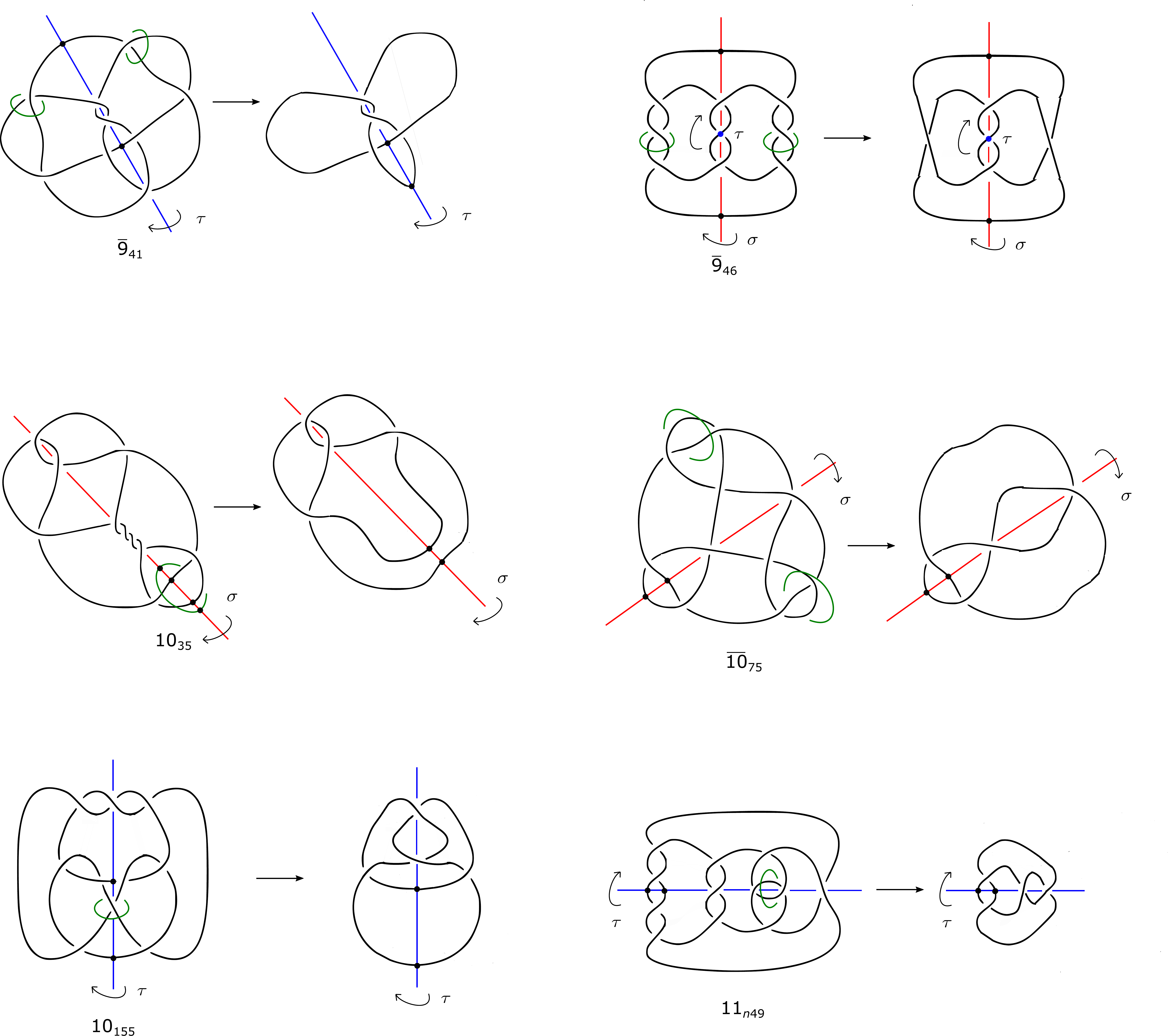}
\caption{Equivariant cobordisms in the proof of Theorem~\ref{thm:1.9}. All green unknots are equipped with framing $-1$.}\label{fig:5.11}
\end{figure}

\subsection{Surgeries on links} \label{sec:5.4}
We now turn to some examples given by surgeries on links. 

\begin{proof}[Proof of Theorem~\ref{thm:1.10}]
We begin by describing a handle attachment cobordism on $M_n$. Let the components of $M_n$ be $\alpha$ and $\beta$, oriented such that $(\alpha, \beta) = 1$. Consider a pair of $(-1)$-framed unknots $x$ and $y$ that link parallel strands of $\alpha$ and $\beta$, as displayed on the left in Figure~\ref{fig:5.12}. We claim that the handle attachment cobordism corresponding to $x$ and $y$ is an interchanging $(-1, -1)$-cobordism from $M_n$ to some manifold $Y_n$. Indeed, a quick computation shows that sliding $x$ over $\beta$ and $y$ over $\alpha$ gives the desired claim (see Figure~\ref{fig:5.12}). On the right in Figure~\ref{fig:5.12}, we have displayed an alternative diagram for $Y_n$ in which $x$ and $y$ are replaced by two zero-framed unknots $p$ and $q$, which are themselves linked by a $(+1)$-framed unknot $r$. As a surgery diagram for $Y_n$, this is equivariantly diffeomorphic to the previous.

\begin{figure}[h!]
\center
\includegraphics[scale=0.85]{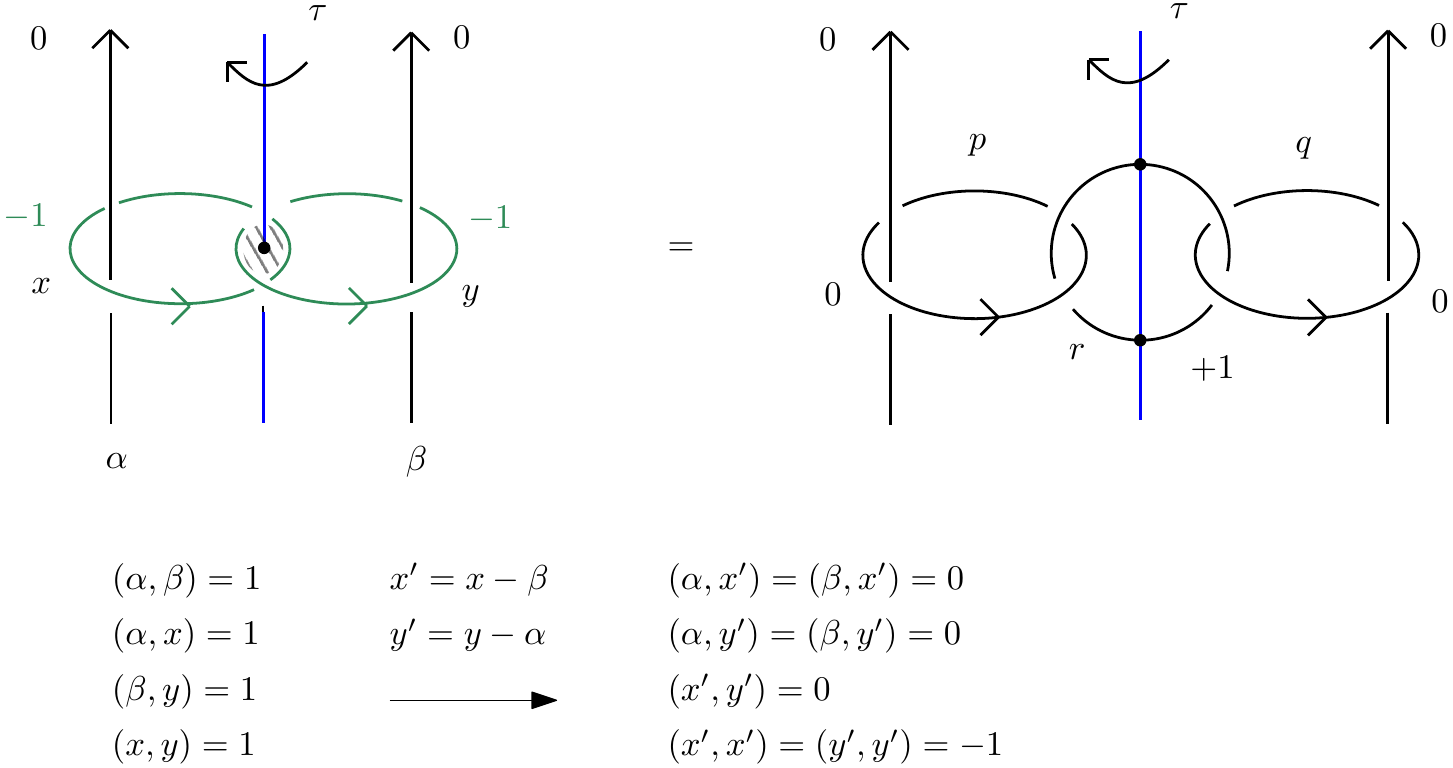}
\caption{Fundamental cobordism in the proof of Theorem~\ref{thm:1.10}.}\label{fig:5.12}
\end{figure}

We attach the configuration of Figure~\ref{fig:5.12} to the bottom of the link defining $M_n$. Clearly, $Y_n$ can be given the alternative equivariant surgery diagram shown in Figure~\ref{fig:5.13}. We modify this diagram by equivariantly sliding all of the $(-1)$-framed horizontal unknots over $p$ and $q$ and deleting them. This yields the second diagram in Figure~\ref{fig:5.13}. Through equivariant isotopy, we transfer the two half-twists of the vertical $(-1)$-curves onto $r$, and then slide the horizontal $(+1)$-framed unknots over $p$ and $q$. We then blow down everything except for $r$. This yields the final diagram in Figure~\ref{fig:5.13}. 

\begin{figure}[h!]
\center
\includegraphics[scale=0.6]{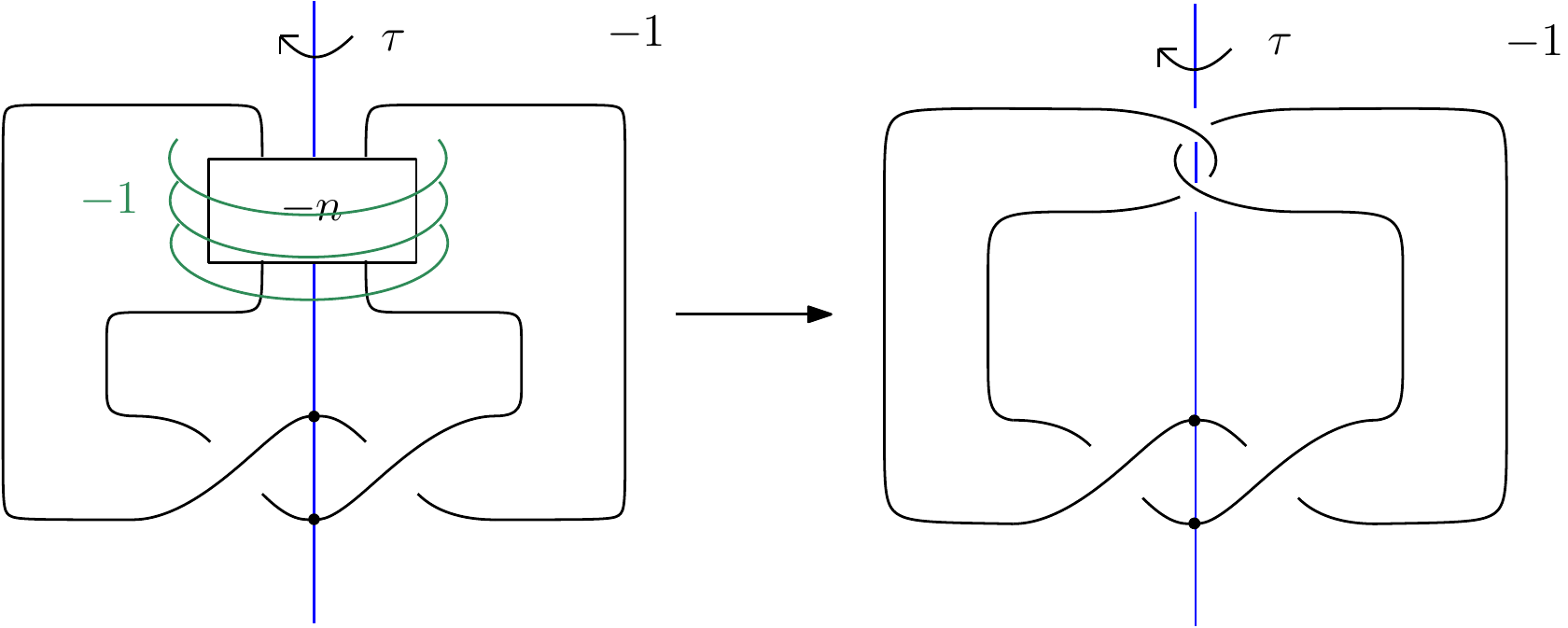}
\caption{Completing the cobordism from $M_n$ to $\Sigma(2, 3, 7)$.}\label{fig:5.14}
\end{figure}

In Figure~\ref{fig:5.14}, we display a $\spinc$-fixing equivariant cobordism from $Y_n$ to $\Sigma(2, 3, 7)$. This consists of attaching $(-1)$-framed unknots and blowing down until only one full negative twist remains. The resulting knot is just the right-handed trefoil, equipped with the strong involution of Lemma~\ref{lem:5.1}. It follows that $h_\tau(M_n) < 0$, as desired. Moreover, it is clear that if $M'$ is constructed from $M_n$ by introducing any number of symmetric pairs of negative full twists (as in Figure~\ref{fig:1.3}), then $M'$ admits a sequence of interchanging $(-1, -1)$-cobordisms to $M_n$. This completes the proof.
\end{proof}

\begin{figure}[h!]
\center
\includegraphics[scale=0.77]{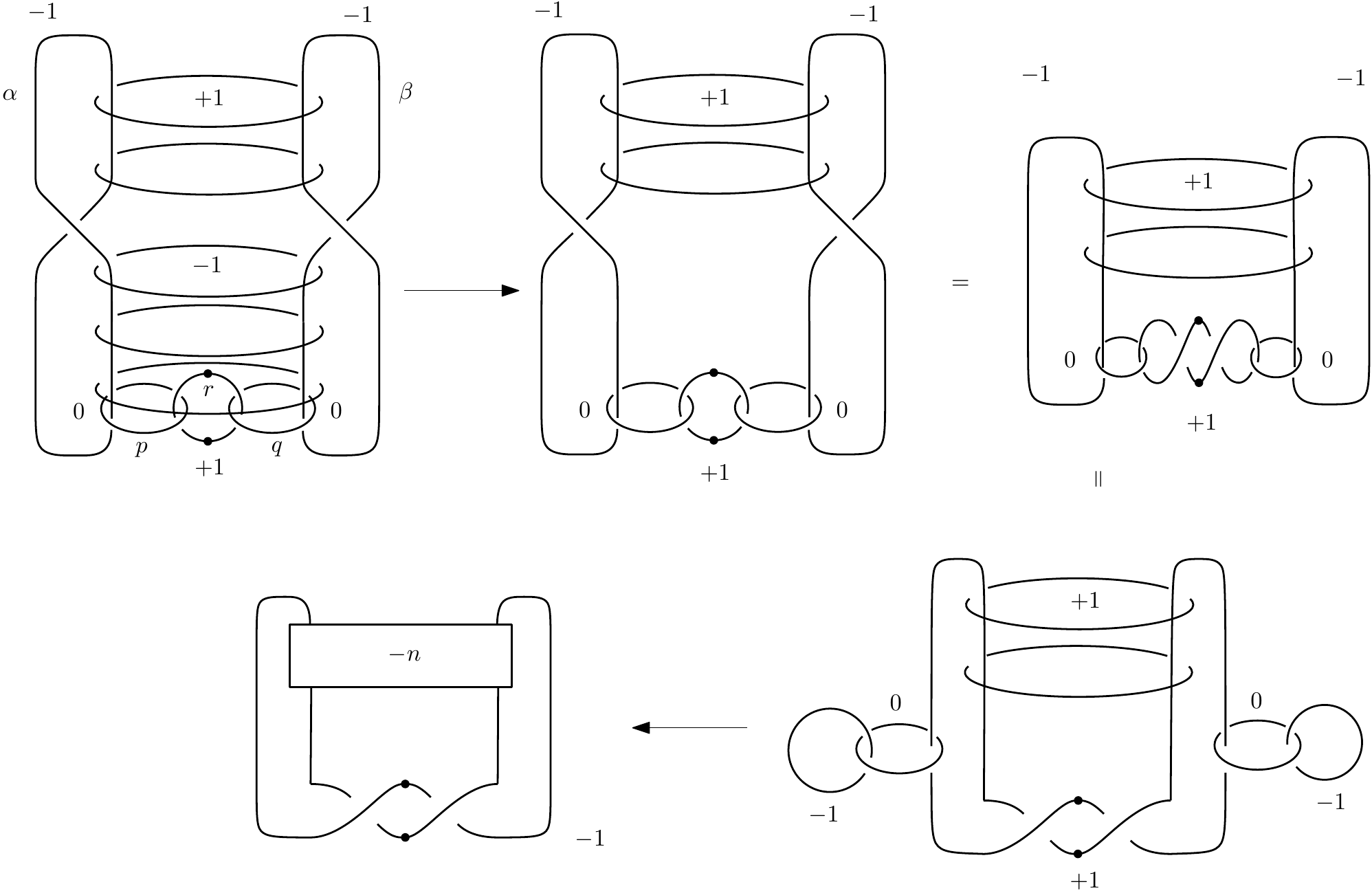}
\caption{Proof of Theorem~\ref{thm:1.10}. In the upper left, there are $n$ horizontal $(+1)$-curves and $n+1$ horizontal $(-1)$-curves.}\label{fig:5.13}
\end{figure}

We now turn to the family of examples in Theorem~\ref{thm:1.11}. The precise definition of the involution $\tau$ is slightly different in this case, as the link is not as obviously symmetric. Explicitly, we first consider an ambient isotopy $f_t$ of $S^3$ which transfers the half-twist in $L_2$ onto $L_1$, as displayed in Figure~\ref{fig:5.15}. This can be done by twisting (for example) only the upper half of the diagram, so that the lower half remains fixed throughout $f_t$. We then compose the diffeomorphism $f_1$ with $180^{\circ}$ rotation about the usual vertical axis. The reader can check that if $f_t$ is chosen carefully, this composition in fact defines an involution $\tau$ of $S^3$ which exchanges $L_1$ and $L_2$. Note that near the lower half of the link, $\tau$ is simply equal to rotation about the vertical axis.

\begin{figure}[h!]
\center
\includegraphics[scale=0.7]{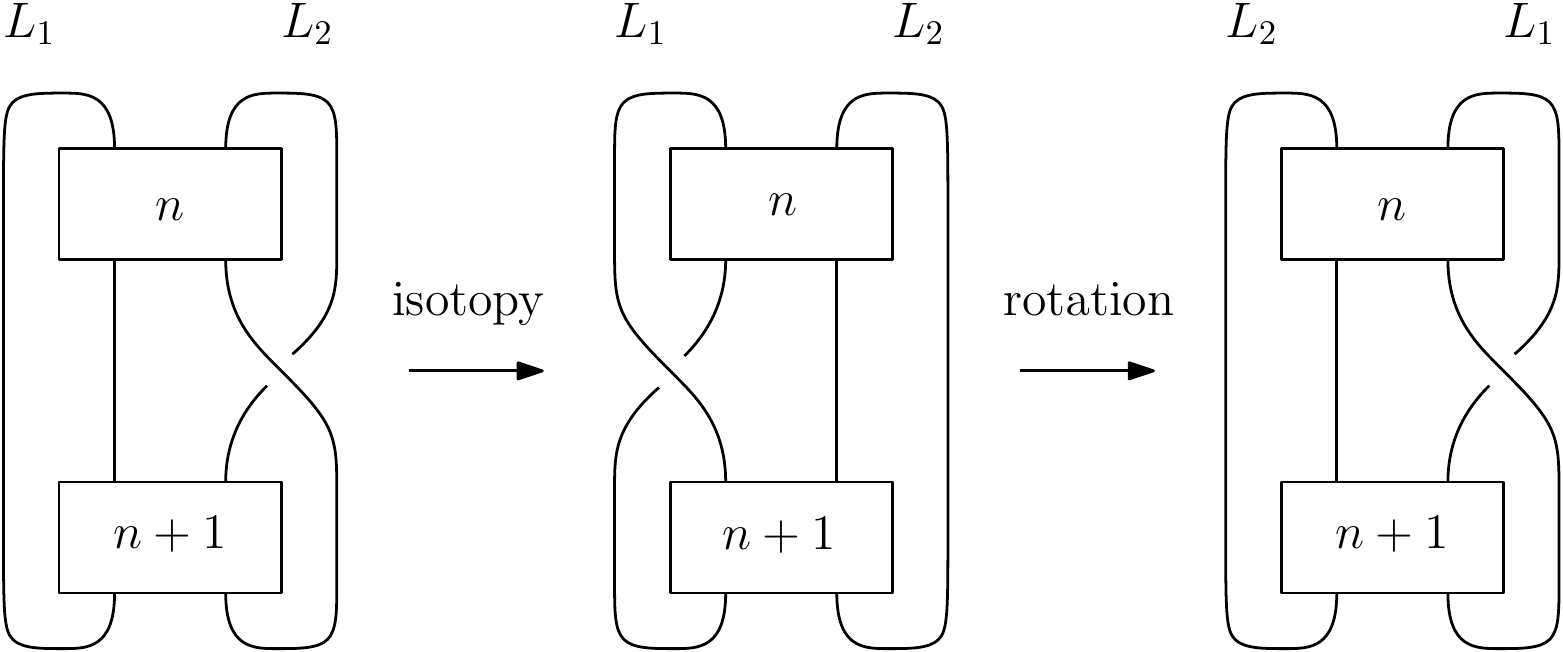}
\caption{Schematic action of $\tau$ on $W_n$.}\label{fig:5.15}
\end{figure}

\begin{proof}[Proof of Theorem~\ref{thm:1.11}]
We begin by using the same handle attachment cobordism as in Figure~\ref{fig:5.12}. This constitutes an interchanging $(-1, -1)$-cobordism from $W_n$ to some other manifold-with-involution, which we claim is $S_{-1}(T_{2, 2n+1})$. To see this, first examine the alternative surgery diagram in Figure~\ref{fig:5.16}. By sliding the $(-1)$-curves over $p$ and $q$ and canceling, we obtain the second diagram in Figure~\ref{fig:5.16}. (For the moment, we will not worry about isotopy/handleslide equivariance with respect to $\tau$.) As before, we transfer the half-twist onto $r$ and slide the remaining $(+1)$-curves over $p$ and $q$. Blowing down everything except for $r$ (keeping in mind the first row of Figure~\ref{fig:5.16}) establishes the claim.

\begin{figure}[h!]
\center
\includegraphics[scale=0.75]{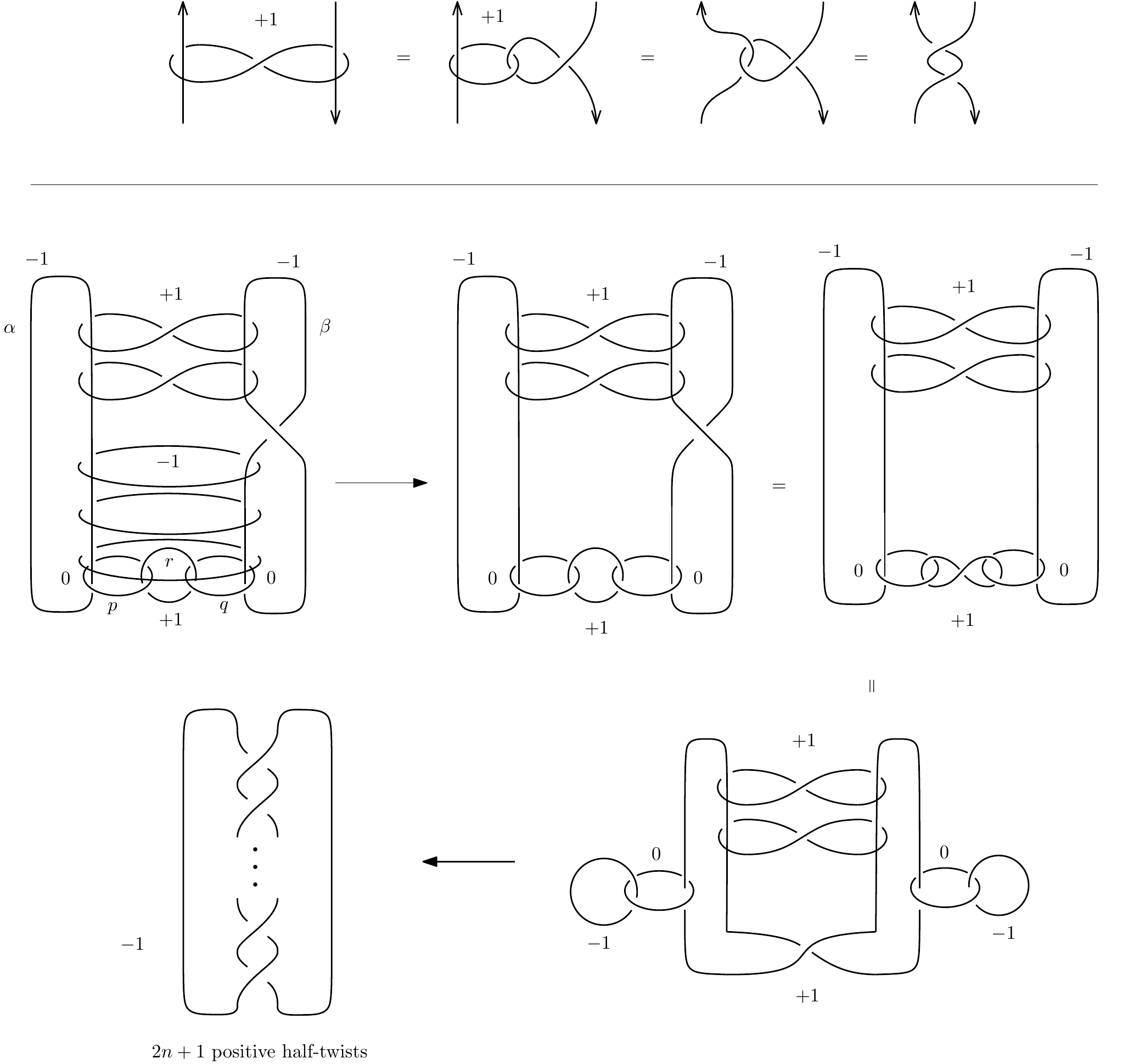}
\caption{Proof of Theorem~\ref{thm:1.11}. In the upper left, there are $n$ horizontal $(+1)$-curves and $n + 1$ horizontal $(-1)$-curves.}\label{fig:5.16}
\end{figure}

This shows that $W_n$ admits an interchanging $(-1, -1)$-cobordism to $B_n = S_{-1}(T_{2, 2n+1})$, where the latter is equipped with some involution $\upsilon$ which we have not identified.\footnote{The enterprising reader can verify that $\upsilon$ is actually the strong involution $\tau$ of Lemma~\ref{lem:5.6}. However, we will not need this for the proof given here.} Hence 
\[
h_\tau(W_n) \leq h_\upsilon(B_n) \text{ and } h_{\ita}(W_n) \leq h_{\iota \circ \upsilon}(B_n).
\] 
However, as discussed in  Section~\ref{sec:5.2}, either $\upsilon \simeq \iota$ or $\upsilon \simeq \id$. Thus one of $h_\upsilon(B_n)$ and $h_{\iota \circ \upsilon}(B_n)$ is equal to $h(B_n) = X_{\lfloor (n+1)/2 \rfloor} < 0$. This shows $W_n$ is a strong cork. Moreover, it is clear that if $W'$ is constructed from $W_n$ by introducing any number of symmetric pairs of negative full twists (as in Figure~\ref{thm:1.4}), then $W'$ admits a sequence of interchanging $(-1, -1)$-cobordisms to $W_n$. This completes the proof.
\end{proof}

Finally, we show that the examples of Theorem~\ref{thm:1.11} span an infinite-rank subgroup of $\G$. Examining the proof of Theorem~\ref{thm:1.11}, it is clear that we can either find an infinite sequence of the $W_i$ for which $h_\tau(W_i) \leq X_{\lfloor (i+1)/2 \rfloor}$, or the same for $h_{\ita}$. Without loss of generality, assume the former. Since the classes $X_i$ are linearly independent (see Section~\ref{sec:2.4}), one might naively expect the $h_\tau(W_i)$ to be linearly independent also. Unfortunately, it is not immediately obvious that this is the case. For example, the $h_{\tau}(W_i)$ could hypothetically all be equal to a fixed $\iota$-complex, which is itself dominated by all of the $X_i$. (While one can prove that this does not happen, we must take care to rule out such situations.) We will thus need to employ some technical results from \cite{HHL} and \cite{DHSTcobordism} in order to complete the proof.

We begin with the following straightforward criterion:

\begin{lemma}\label{lem:5.8}
Let $C_1, \ldots, C_n \in \Inv$ be a sequence of local equivalence classes, and let $C$ also be an equivalence class in $\Inv$. Suppose that $\Hconn(C)$ contains a $U$-torsion tower of length $l$. Assume that $l$ does not appear as a tower length in any $\Hconn(C_1), \ldots, \Hconn(C_{n})$. Then $C$ does not lie in the span of $C_1, \ldots, C_{n}$.
\end{lemma}
\begin{proof}
We claim that if $A$ and $B$ are $\iota$-complexes, then any $U$-torsion tower length appearing in $\Hconn(A \otimes B)$ must appear as a tower length in either $\Hconn(A)$ or $\Hconn(B)$. To see this, note that $\Hconn$ is an invariant of local equivalence, so when computing the connected homology of $A \otimes B$, without loss of generality we can replace $A \otimes B$ with $A_\textrm{conn} \otimes B_\textrm{conn}$. Thus $\Hconn(A \otimes B)$ is a summand of $H_*(A_\textrm{conn} \otimes B_\textrm{conn})$. Using the K\"unneth formula and an examination of the Tor functor, it is easily checked that any tower length appearing in $H_*(A_\textrm{conn} \otimes B_\textrm{conn})$ must appear in either $H_*(A_\textrm{conn})$ or $H_*(B_\textrm{conn})$, giving the desired claim. This shows that if $C$ is a linear combination of $C_1, \ldots, C_{n}$, then any tower length appearing in $\Hconn(C)$ must appear in the connected homology of some factor.
\end{proof}

We now come to a crucial technical lemma. The proof of this relies on a more subtle analysis of the machinery of $\iota$-complexes than we have encountered so far. In particular, we will need some familiarity with the results of \cite{DHSTcobordism}, in which the algebra of $\iota$-complexes is studied through the use of the \textit{almost local equivalence group}. As this is the only place in which we utilize this notion, we have chosen not to emphasize it, but for convenience of the reader, we give a brief overview here.

In \cite[Section 3.1]{DHSTcobordism}, it is shown that one can define a slightly modified group $\smash{\widehat{\Inv}}$ by requiring various equalities in Definitions~\ref{def:2.2} and \ref{def:2.3} to hold only modulo $U$. More precisely, we say that a pair $(C, \iota)$ is an \textit{almost $\iota$-complex} if $C$ is an $\ff[U]$-complex (as in Definition~\ref{def:2.2}) and $\iota: C \rightarrow C$ is a grading-preserving endomorphism such that:
\begin{enumerate}
\item $\iota \partial + \partial \iota \equiv 0$ mod $U$; and,
\item $\iota^2 + \id \equiv \partial H + H \partial$ mod $U$ for some homotopy $H$.
\end{enumerate}
Similarly, local maps between almost $\iota$-complexes are only required to commute with $\iota$ up to homotopy modulo $U$. We may analogously form an almost local equivalence group $\smash{\widehat{\Inv}}$ by considering the set of almost $\iota$-complexes modulo this modified notion of local equivalence. This comes with a forgetful homomorphism
\[
f: \Inv \rightarrow \smash{\widehat{\Inv}},
\]
so that $\Inv$ can be partially understood by studying $\smash{\widehat{\Inv}}$. Indeed, the precise details of the above definitions are not very important here; instead, the reader should think of $\smash{\widehat{\Inv}}$ as an algebraic artifice which turns out to be easier to understand than $\Inv$.

There are several advantages to working with $\smash{\widehat{\Inv}}$. Firstly, it turns out that it is possible to enumerate the different local equivalence classes in $\smash{\widehat{\Inv}}$. Each local equivalence class has a preferred representative, called a \textit{standard complex}. The set of standard complexes is parameterized by certain finite sequences of symbols, which (among other things) record $U$-torsion tower lengths in the usual $\ff[U]$-homology of each standard complex; see \cite[Definition 4.1]{DHSTcobordism}. Secondly, the partial order on $\smash{\widehat{\Inv}}$ turns out to be a total order \cite[Theorem 3.25]{DHSTcobordism}, and the forgetful homomorphism respects this ordering. The ordering on $\smash{\widehat{\Inv}}$ can be described in terms of a lexicographic ordering on the set of parameter sequences \cite[Section 4.2]{DHSTcobordism}.

\begin{lemma}\label{lem:5.9}
Let $C \in \Inv$ be a local equivalence class. Suppose $C \leq X_i$ for some fixed $i$, where $X_i$ is the $\iota$-complex of Example~\ref{ex:2.9}. Then for all $n \neq 0$, the connected homology $\Hconn(nC)$ contains a $U$-torsion tower of length at least $i$.
\end{lemma}
\begin{proof}
For simplicity, we first consider the case $n = 1$. Let the almost $\iota$-complex parameter sequence corresponding to $f(C)$ be given by $(a_1, b_1, \ldots, a_k, b_k)$. Since the parameter sequence corresponding to $f(X_i)$ is given by $(-, +i)$ (see \cite[Section 4.1]{DHSTcobordism}), by \cite[Theorem 4.6]{DHSTcobordism} we have
\[
(a_1, b_1, \ldots, a_k, b_k) \leq^{!} (-, +i),
\]
where $\leq^{!}$ is the lexicographic order of \cite[Section 4.2]{DHSTcobordism}. Examining the definition of $\leq^{!}$, it is clear that $a_1 = -$. There are then two possibilities for $b_1$ consistent with the above inequality: either $b_1 < 0$, or $b_1 \geq i$. The first possibility is ruled out by \cite[Theorem 8.3]{DHSTcobordism}, since $C$ is a genuine $\iota$-complex. This means that the standard complex representative of $f(C)$ has a $U$-torsion tower of length $b_1 \geq i$ in its homology. 

It follows from the proof of \cite[Theorem 6.3]{DHSTcobordism} that the homology of the standard complex representative of $f(C)$ is a summand of the usual homology of $C$. (In \cite[Theorem 6.3]{DHSTcobordism}, it is stated that $C$ is locally equivalent to an almost $\iota$-complex of which the standard complex is a summand. However, an examination of the proof shows that actually the local equivalence condition can be replaced with homotopy equivalence. For the analogous statement in knot Floer homology, see \cite[Corollary 6.2]{DHSTconcordance}.) Since the standard complex is an invariant of local equivalence, we may replace $C$ with $C_\mathrm{conn}$ without loss of generality. Hence we see that $\Hconn(C)$ contains a $U$-torsion tower of length $b_1 \geq i$, as desired.

Now suppose $n > 1$. Note that the inequality $C \leq X_i$ implies $nC \leq nX_i$ for all $n \geq 0$. By \cite[Theorem 8.1]{DHSTcobordism}, the parameter sequence corresponding to $f(nX_i)$ is given by
\[
(-, +i, -, +i, \ldots, -, +i).
\]
Thus, the same argument as before gives the desired result when $n > 1$. The case when $n \leq -1$ follows from the behavior of $\Hconn$ under dualizing; see \cite[Proposition 4.1]{HHL}.
\end{proof}

\begin{proof}[Proof of Theorem~\ref{thm:1.3}]
As mentioned previously, the proof of Theorem~\ref{thm:1.11} makes clear that we can either find an infinite sequence of the $W_i$ for which $h_\tau(W_i) \leq X_{\lfloor (i+1)/2 \rfloor}$, or the same for $h_{\ita}$. Without loss of generality, assume the former. It is straightforward to inductively construct an infinite linearly independent subsequence $W_{i_p}$, as follows. At the $p$-th stage, let $i_p$ be any integer for which $\lfloor (i_p+1)/2 \rfloor$ is larger than the maximal $U$-torsion tower length appearing amongst $\Hc^\tau(W_{i_1}), \ldots, \Hc^\tau(W_{i_{p-1}})$. By Lemma~\ref{lem:5.9}, the connected homology $\Hc^\tau(nW_{i_p})$ must contain a $U$-torsion tower of length at least $\lfloor (i_p+1)/2 \rfloor$ for all $n \neq 0$. It follows from Lemma~\ref{lem:5.8} that no nonzero multiple of $h_\tau(W_{i_p})$ lies in the span of $h_\tau(W_{i_1}), \ldots, h_\tau(W_{i_{p-1}})$. Since $h_\tau$ is a homomorphism from $\G$ to $\Inv$, this proves that $W_{i_1}, \ldots, W_{i_p}$ are linearly independent. Proceeding inductively yields a $\Z^\infty$-subgroup.
\end{proof}

\begin{remark}\label{rem:5.10}
We do \textit{not} show that all the $W_i$ are linearly independent, although we conjecture this is the case. 
\end{remark}

We now turn to a brief discussion of $\Z/2\Z$-homology spheres. It is straightforward to alter Definition~\ref{def:1.13} by changing all coefficient rings to $\Z/2\Z$; following Definitions~\ref{def:homologydiff} and \ref{def:1.17}, we then obtain a $\Z/2\Z$-bordism group of involutions, which we denote by $\smash{\Theta^{\tau}_{\Z/2\Z}}$. Note that in this context, each homology sphere and cobordism comes equipped with a unique spin structure, which is necessarily fixed under any self-diffeomorphism. We leave it to the reader to formulate and verify the appropriate analogues of Theorems~\ref{thm:1.1} and \ref{thm:1.2}.

\begin{proof}[Proof of Corollary~\ref{cor:branched}]
As can be seen from \cite[Figure 4]{Harper}, the pairs $(W_i, \tau_i)$ arise as double branched covers of knots in $S^3$. It is easily checked that the proof of Theorem~\ref{thm:1.3} establishes linear independence (of the appropriate subsequence) over the group $\smash{\Theta^{\tau}_{\Z/2\Z}}$.
\end{proof}

\begin{remark}
One can also consider the family of knots
\[
k(2, p, q) \# - T(p, q)
\]
given by the connected sum of a Montesinos knot with an orientation-reversed torus knot. The double branched cover of this is $\Sigma(2, p, q) \# - \Sigma(2, p, q)$, which evidently bounds a homology ball. The covering involution on the first factor is nontrivial, while the covering involution on the second factor is isotopic to the identity. As in Remark~\ref{rem:summand}, it follows that an appropriate family of such knots spans a $\mathbb{Z}^\infty$-summand of $\mathcal{C}$ and maps injectively into $\smash{\Theta^{\tau}_{\Z/2\Z}}$. However, note that the resulting double branched covers do not in general bound contractible manifolds.
\end{remark}

\subsection{Further computations}\label{sec:secP}
We now turn to a slightly more involved example, in which we construct an equivariant cobordism from a candidate strong cork to a manifold other than a Brieskorn sphere. To constrain the invariants of this manifold, we leverage some simple calculations coming from involutive Heegaard Floer homology. We have included this computation to demonstrate that in favorable situations, one can apply our techniques to other ``base cases" than the manifolds established in \cite{AKS}, so as to create more examples of strong corks. 

\begin{proof}[Proof of Theorem~\ref{thm:P}]
We begin by constructing an interchanging $(-1, -1)$-cobordism from $P$ to another manifold-with-involution. To this end, consider the fundamental cobordism displayed on the left in Figure~\ref{fig:5.17}. This is formed by attaching two $(-1)$-handles to parallel strands of $P$. Figure~\ref{fig:5.17} is analogous to Figure~\ref{fig:5.12}, but differs slightly due to the fact that the two components of $P$ (with the orientations displayed in Figure~\ref{fig:5.17}) have linking number $-1$, rather than $+1$. Performing a change-of-basis shows that this is an interchanging $(-1, -1)$-cobordism. On the right in Figure~\ref{fig:5.17}, we have displayed an alternative surgery diagram for the resulting manifold. The reader should check that this is equivariantly diffeomorphic to the previous.

\begin{figure}[h!]
\center
\includegraphics[scale=0.8]{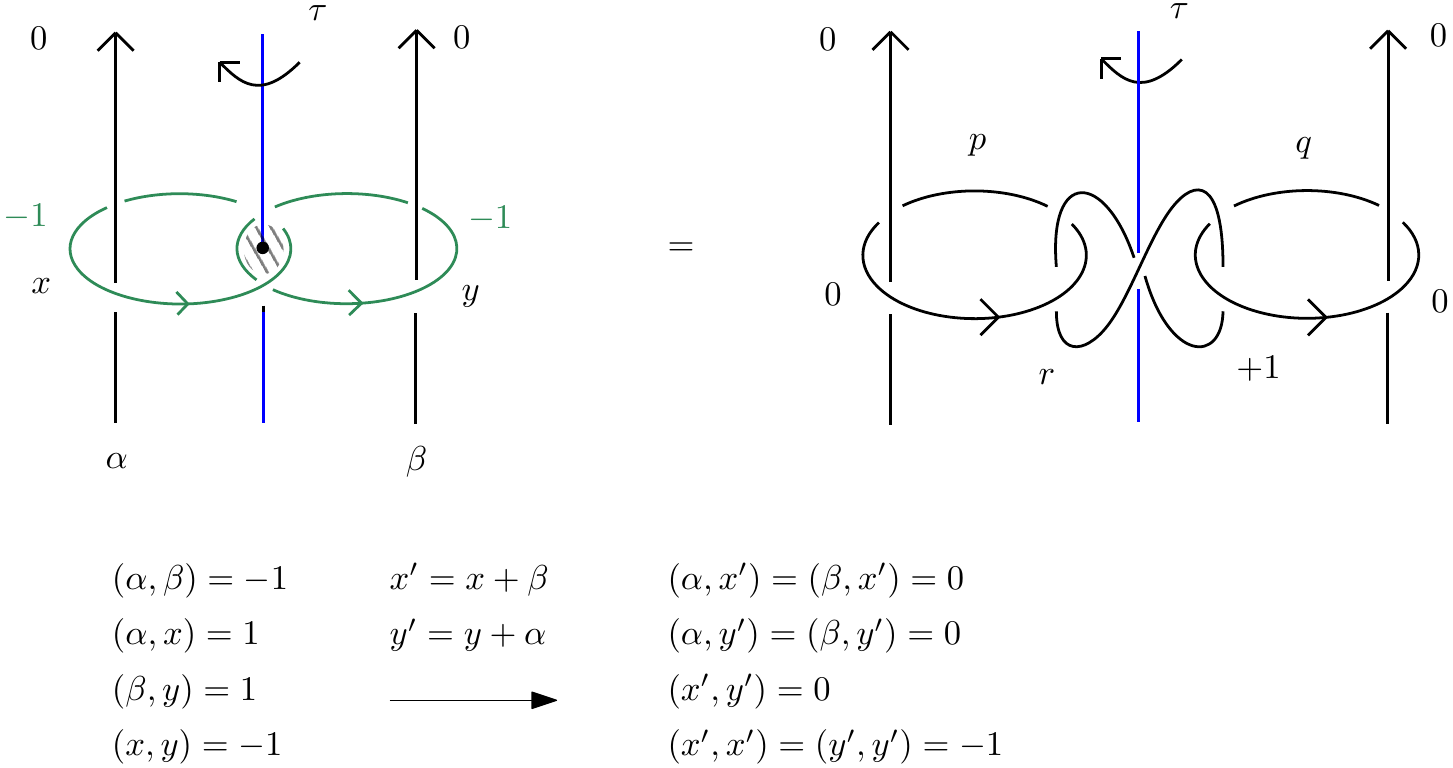}
\caption{Fundamental cobordism in the proof of Theorem~\ref{thm:P}. Here, $\alpha$ and $\beta$ are parallel strands in the two components of $P$. Note the difference in crossings from Figure~\ref{fig:5.12}.}\label{fig:5.17}
\end{figure}

\begin{figure}[h!]
\center
\includegraphics[scale=0.83]{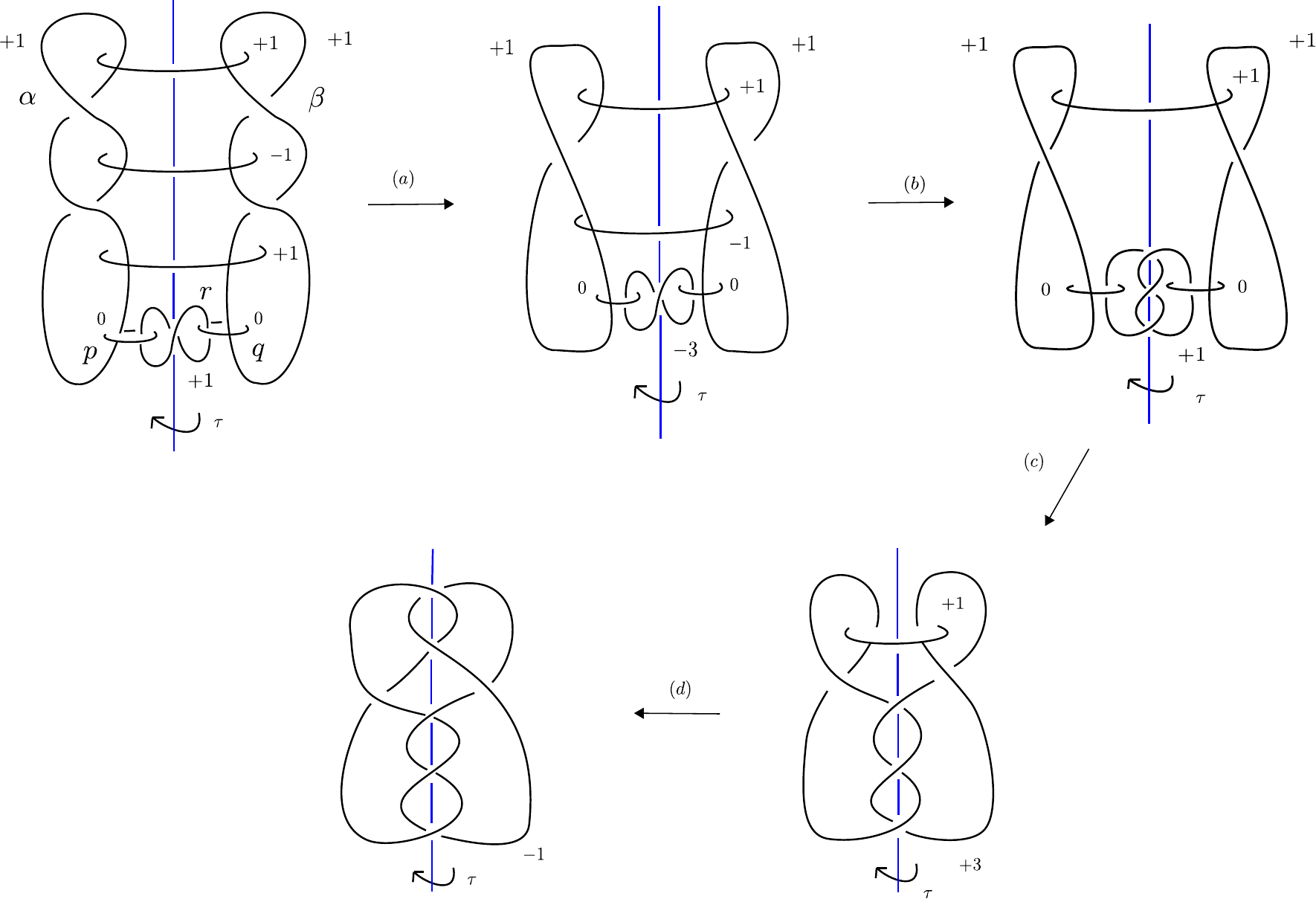}
\caption{Equivariant cobordism used in the proof of Theorem~\ref{thm:P}. The first diagram is obtained by attaching the configuration of Figure~\ref{fig:5.17} to an alternative surgery diagram for $P$. In $\textbf{(a)}$ we slide the nearest $(+1)$-curve over $p$ and $q$, blow down, and transfer two of the half-twists in $\alpha$ and $\beta$ to $r$. In $\textbf{(b)}$ we similarly slide the $(-1)$-curve over $p$ and $q$ and blow down. In $\textbf{(c)}$ we transfer the remaining half-twists in $\alpha$ and $\beta$ to $r$, slide the horizontal $(+1)$-curve over $p$ and $q$, and then blow down the $(+1)$-curves on either side. Finally, in $\textbf{(d)}$ we blow down the remaining $(+1)$-curve. This yields $(-1)$-surgery on a knot which the reader can check is $6_{2}$.}\label{fig:5.18}
\end{figure}

Using the Kirby calculus manipulations shown in Figure~\ref{fig:5.18}, one can prove that our new manifold is equivariantly diffeomorphic to $S_{-1}(6_2)$, equipped with the indicated involution $\tau$. Hence by Theorem~\ref{thm:1.4}, we have
\[
h_\tau(P) \leq h_\tau(S_{-1}(6_{2})) \text{ and } h_{\ita}(P) \leq h_{\iota \circ \tau}(S_{-1}(6_{2})).
\] 
It thus suffices to show that either of the invariants of $S_{-1}(6_2)$ are strictly less than zero. For simplicity, we work on the level of homology by ruling out the existence of an equivariant $\ff[U]$-module map from the trivial module $\ff[U]$ (equipped with the identity involution) to $\HFm(S_{-1}(6_2))$ (equipped with either involution $\tau_*$ or $\iota_* \circ \tau_*$), as in Remark~\ref{rem:2.H}. 


To this end, we first compute the Heegaard Floer homology of $S_{-1}(6_2)$. Since $6_2$ is alternating, its knot Floer complex is determined by its Alexander polynomial. It is then straightforward to calculate $\HFm(S_{-1}(6_2))$ via the usual surgery formula \cite{OSinteger}, although for technical reasons we display the computation for $\HFp(S_{+1}(\overline{6}_2))$ instead. (See Figure~\ref{fig:5.19}.) For convenience, denote $K = \overline{6}_2$. Note that since $K$ has genus two, the desired Floer homology is \textit{not} given by the large surgery formula, but rather the homology of the mapping cone $\mathbb{X^{+}}(1)$ displayed in Figure~\ref{fig:5.20}. In this case, the desired homology is quasi-isomorphic to the kernel of the (truncated) mapping cone map with domain $H_*(A_{-1}^+) \oplus H_*(A_{0}^+) \oplus H_*(A_{+1}^+)$. The resulting calculation is displayed on the right in Figure~\ref{fig:5.19}.

\begin{figure}[h!]
\center
\includegraphics[scale=0.8]{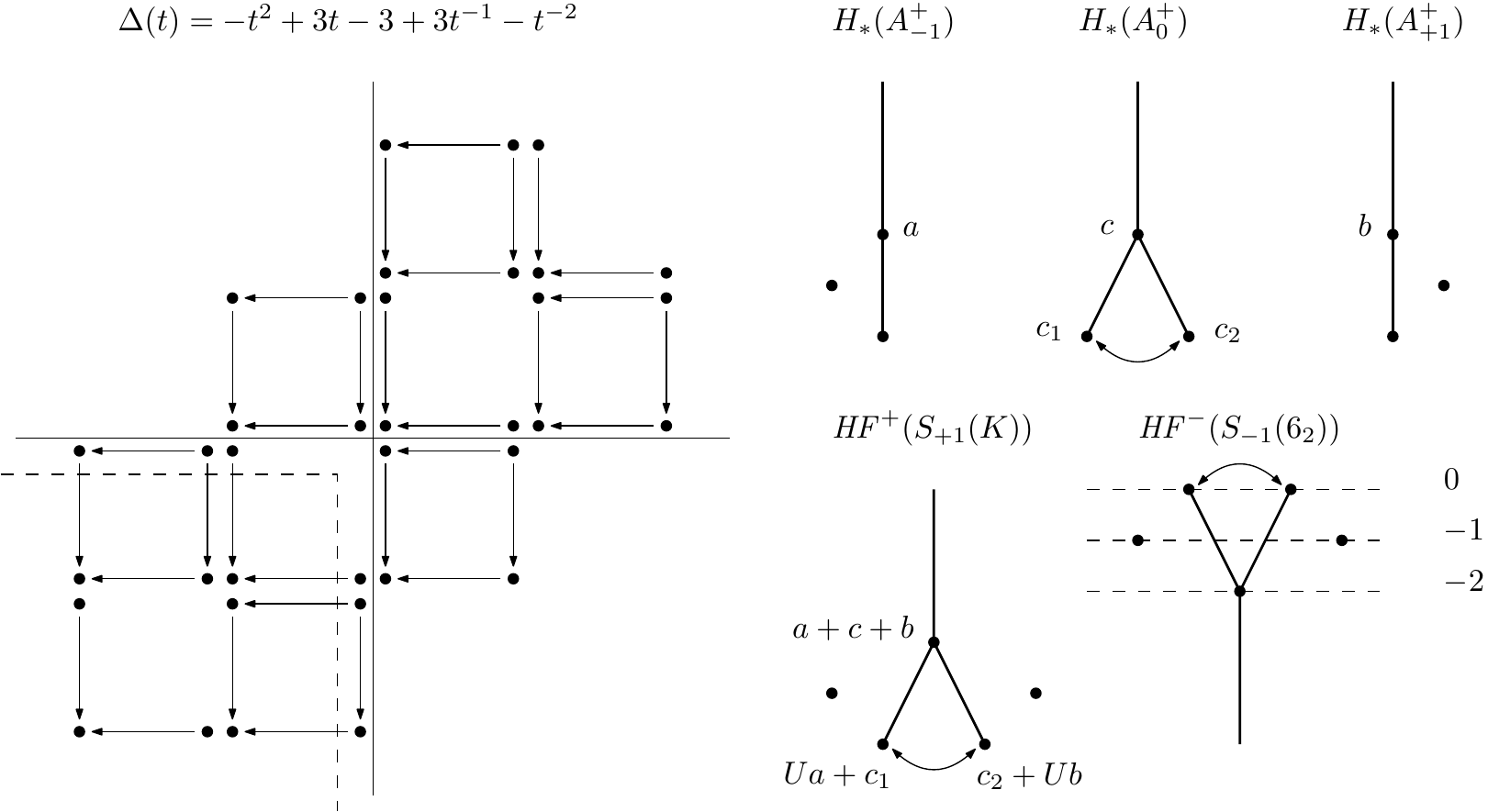}
\caption{Left: the knot Floer complex of $K$, with the dotted line marking the boundary of the quotient complex $A^{+}_{0}$. Right: various homologies $H_*(A_i^+)$, together with the calculation of $\HFp(S_{+1}(K))$.}\label{fig:5.19}
\end{figure}

\begin{figure}[h!]
\center
\includegraphics[scale=0.7]{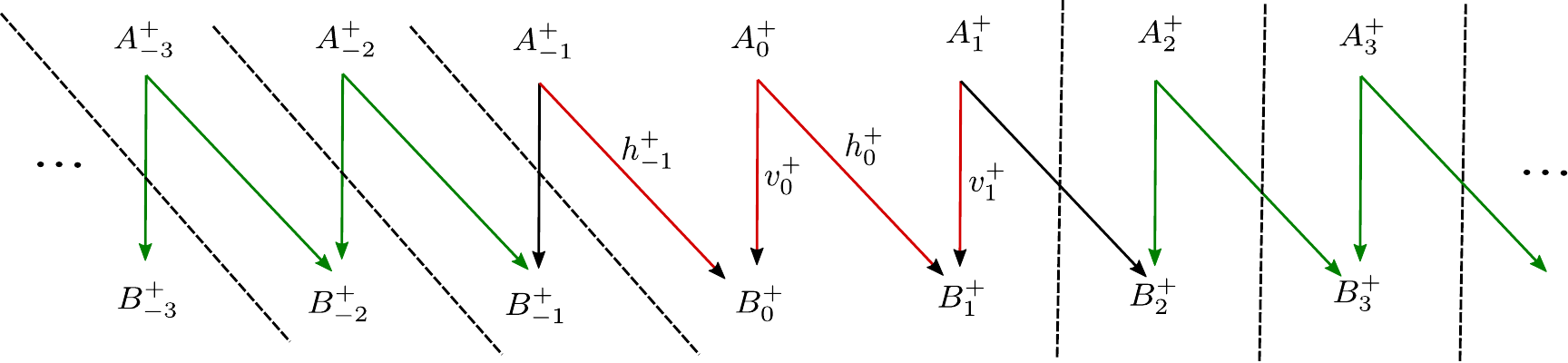}
\caption{The mapping cone $\mathbb{X^{+}}(1)$. Green arrows are homotopy equivalences. The truncated mapping cone (which carries the homology) consists of the red arrows.}\label{fig:5.20}
\end{figure}

We now attempt to obtain partial information regarding the action of $\iota_*$ on $\HFp(S_{+1}(K))$. As before, we can compute the action of $\iota_K$ on the knot Floer complex of $K$; this is given by reflection across the obvious diagonal. However, we cannot use the involutive large surgery formula and (at the time of writing) there is not a general involutive surgery formula. We thus resort to the following trick. Observe that there is a map 
\begin{equation*}
q: \mathbb{X}^+(1) \longrightarrow A^{+}_0
\end{equation*}
formed by quotienting out $\mathbb{X}^{+}(1)$ by everything other than $A^+_0$. In the basis of Figure~\ref{fig:5.19}, the induced map $q_* : H_*(\mathbb{X}^+(1)) \rightarrow H_*(A_0^+)$ sends the two obvious unmarked generators to zero and acts as an isomorphism on the rest of the homology. According to the proof of integer surgery formula in \cite{OSinteger}, under the identification of $H_*(\mathbb{X}^+(1))$ with $\HFp(S_{+1}(K))$, the quotient map $q$ coincides (on homology) with the triangle-counting map
\begin{equation*}
\Gamma^{+}_{0}: \CFp(S_{+1}(K)) \longrightarrow A^{+}_0
\end{equation*}
defined in \cite{OSinteger}. Furthermore, following the proof of \cite[Theorem 1.5]{HM}, one can show that $(\Gamma^{+}_{0})_*$ intertwines the actions of $\iota_*$ on $\HFp(S_{+1}(K))$ and $(\iota_0)_*$ on $H_*(A_0^+)$; that is,
\[
(\iota_0)_* \circ (\Gamma^+_0)_* = (\Gamma^+_0)_* \circ \iota_*.
\]
More precisely, Hendricks and Manolescu consider the map $\smash{\Gamma^+_{0,p}} : \CFp(S_p(K)) \rightarrow A^+_0$ when $p$ is large, and show that this intertwines $\iota$ and $\iota_0$. However, their proof of this fact does not depend on the surgery coefficient $p$. Of course, $\smash{\Gamma^+_{0,p}}$ no longer induces an isomorphism for small surgeries. See \cite[Equation 26]{HM} and \cite[Section 6.6]{HM} .
 
Using Hendricks and Manolescu's computation of $\iota_K$ for thin knots \cite{HM}, we can calculate that the action of $(\iota_0)_*$ on $H_*(A_0^+)$ interchanges the two elements of lowest grading. Hence $\iota_*$ on $\HFp(S_{+1}(K))$ must also interchange the two elements of lowest grading. Reflecting $\HFp(S_{+1}(K))$ over a horizontal line gives $\HFm(S_{-1}(6_2))$, with the action of $\iota_*$ exchanging the two elements of (shifted) grading zero, as displayed in Figure~\ref{fig:5.19}. Hence one of $\tau_*$ or $(\iota \circ \tau)_*$ on $\HFm(S_{-1}(6_2))$ must also exchange the pair of elements in grading zero. Clearly, there is no map (satisfying the properties of Remark~\ref{rem:2.H}) from the trivial $\ff[U]$-module, equipped with the identity involution, to $\HFm(S_{-1}(6_2))$, equipped with an involution acting nontrivially on the two elements of highest grading.


This completes the proof that $(P, \tau)$ is a strong cork. Moreover, it is clear that if $P'$ is constructed from $P$ by introducing any number of symmetric pairs of negative full twists (as in Figure~\ref{fig:1.Pos}), then $P'$ admits a sequence of interchanging $(-1, -1)$-cobordisms to $P$.
\end{proof}

\bibliographystyle{amsalpha}
\bibliography{bib}

\end{document}